\documentclass[utopia]{nmd/nmd-article} 
\pdfoutput=1
\usepackage[letterpaper,headsep=10bp,headheight=15bp,left=1.7cm,right=1.7cm,top=2.2cm,bottom=2.5cm]{geometry}

\usepackage{amsfonts}
\usepackage{amsmath}
\usepackage{amsthm}
\usepackage{amssymb}
\usepackage{stackengine}
\usepackage{longtable}
\usepackage{array}
\usepackage{float}
\usepackage{enumitem}
\usepackage{titletoc}
\usepackage{changepage}
\usepackage{wasysym}
\usepackage{multirow}
\usepackage{doiMatthias}

\newcommand{\myTrig}{\mathcal{T}\!}
\newcommand{\postNorm}{\widehat{\phantom{x}}~}
\newcommand{\posLightCone}{\mathbb{L}_{+}^3}
\newcommand{\emphasisText}[1]{{\bf #1}}
\newcommand{\emphasisDef}[1]{{\bf #1}}
\newcommand{\from}{\colon}
\newcommand{\setOfGeodesics}{G}
\newcommand{\muCensus}[1]{\mu\textnormal{\texttt{(#1\!)}}}

\newcommand{\myComment}[1]{}

\floatstyle{ruled}
\newfloat{algorithm}{thp}{}
\floatname{algorithm}{Algorithm}

\makeatletter
\let\c@algorithm=\c@subsubsection
\makeatother

\newlist{algorithmSteps}{enumerate}{6}
\setlist[algorithmSteps]{label={\arabic*.}, ref={\arabic*}, nosep}
\setlist[enumerate]{label={\arabic*.}, ref={\arabic*}}

\DeclareMathOperator{\RePart}{Re}
\DeclareMathOperator{\ImPart}{Im}
\DeclareMathOperator{\Id}{Id}
\DeclareMathOperator{\myPSL}{PSL}
\DeclareMathOperator{\mySL}{SL}
\DeclareMathOperator{\Eq}{Eq}
\DeclareMathOperator{\EpsilonCandidate}{E}
\DeclareMathOperator{\systole}{sys}

\title{Verified Length Spectrum and\\ Margulis number for\\ Hyperbolic 3-Manifolds}
\author{Matthias Goerner}
\email{enischte@gmail.com}

\author{Robert C. Haraway III}
\email{bobbycyiii@fastmail.com}

\author{Neil R. Hoffman}
\email{neilhoff@d.umn.edu}

\author{Maria Trnkova}
\email{mtrnkova@ucdavis.edu}

\begin{document}

\begin{abstract}
We present a new algorithm to compute a stream for the length spectrum of an oriented, finite-volume hyperbolic $3$-manifold $M$. The advantage over the existing algorithm by Hodgson--Weeks is that our algorithm does not require finding a Dirichlet domain and thus can use interval arithmetic to verify the result. Our algorithm is also significantly faster in some cases. We also present a new algorithm to compute the optimal Margulis number of $M$. The authors have made these algorithms available in SnapPy.
\end{abstract}

\maketitle

\vspace{-0.4cm}

\dottedcontents{section}[3.2em]{\bfseries}{1.8em}{0.7pc}
\setcounter{tocdepth}{1}
\tableofcontents

\newpage

\section{Introduction} \label{sec:defLenSpecStream}

\subsection{Motivation}

Let $M$ be an oriented, finite-volume hyperbolic $3$-manifold. The length spectrum is the multi-set of complex lengths $\lambda$ of all primitive, unoriented closed geodesics $\gamma$ in $M$. It is an important invariant with deep connections to arithmetic, spectral geometry and Dehn surgery. The Selberg trace formula, for example, relates the length spectrum to the spectrum of the Laplace--Beltrami operator and implies that the length spectrum follows a distribution similar to that of the prime numbers. The systole is the smallest $\RePart(\lambda)$ of the length spectrum. Reid and Adams show in \cite{AR:Systoles} that the systole needs to be small for a hyperbolic $3$-manifold to be a link complement or to contain incompressible surfaces of small genus. In \cite{BGR:PrinCong}, this is used, for example, to classify all principal congruence links. In \cite{Haraway:Dehn}, Haraway uses the systole for the Dehn parental test which asks whether two given hyperbolic manifolds are related through a single Dehn filling. More recently, using the verified algorithm for the systole from this paper, Futer, Purcell and Schleimer gave a method in \cite{FPS:Cosmetic} to verify the cosmetic surgery conjecture \cite[Conjecture~6.1]{gordon:surgery} for an oriented, one-cusped hyperbolic $3$-manifold $N$. 
Recall that the cosmetic surgery conjecture states that an oriented, compact $3$-manifold $N$ with $\partial N\cong S^1\times S^1$ has no purely cosmetic surgeries. That is, there is no pair of distinct slopes $s$ and $t$ in $\partial N$ such that there is an orientation-preserving homeomorphism $N(s)\to N(t)$.
 They applied their method to the SnapPy census up to 9 tetrahedra. By running their code \cite{Schleimer:CosmeticCode} and settling the cases not handled by their code in Section~\ref{sec:cosmetic}, we extend their result to the recent census of 10 and 11 tetrahedra by \cite{li:10tet,li:11tet}:
\begin{theorem}[(Compare to \protect{\cite[Theorem~4.3]{FPS:Cosmetic}})]\label{thm:cosmetic}
Of the 686,095 one-cusped orientable manifolds with up to 11 tetrahedra in the SnapPy census, none has purely cosmetic surgeries. 
\end{theorem}

Another application of the length spectrum is the isomorphism problem for closed hyperbolic 3-manifolds. This paper is part of a series of three papers solving the isomorphism problem. The first paper \cite{goerner:tiling} introduces a new tiling algorithm. This paper is the second paper and uses the new tiling algorithm to compute the length spectrum. The last paper \cite{goerner:drilling} shows how to drill geodesics. Together, this allows us to compute the isometry signature of a closed oriented hyperbolic 3-manifold by drilling each of a canonical set of geodesics from the length spectrum. The isometry signature is a complete invariant and thus solves the isomorphism problem.

As explained in \cite[Section~1.1.12]{goerner:tiling}, the length spectrum is also related to the Margulis numbers. In Section~\ref{sec:margulis}, we give an algorithm to compute the optimal Margulis number $\mu(M)$ of an oriented, finite-volume hyperbolic $3$-manifold $M$. The algorithm is available in SnapPy \cite{SnapPy} as of version~3.3. For example, \texttt{Manifold("m004").margulis()} gives $\mu(M)$ for the figure-eight knot complement. We use it in Section~\ref{sec:proofcensusMargulisConstant} to prove:
\begin{theorem} \label{thm:censusMargulisConstant}
Among the orientable cusped (up to 11 tetrahedra) and closed SnapPy census without \texttt{m007(3,1)}, the Weeks manifold \texttt{m003(-3,1)} has the smallest optimal Margulis number $\mu\left(M \right)=0.77442660700998\dots$~.
\end{theorem}

We excluded the manifold \texttt{m007(3,1)} (also known $\text{Vol}_3$) in the theorem because it has no known geometric spun-triangulation. An unverified computation indicates though that $\muCensus{m007(3,1)}$ is equal to its systole $0.8314429455\dots$; see Remark~\ref{rem:margulisVol3}. Also, the manifold \texttt{m003(-3,1)} is the minimum volume orientable hyperbolic 3-manifold; see \cite{gmm:minVol}.

The Margulis constant $\mu_n$ is the infimum of the optimal Margulis number $\mu(M)$ across all orientable, finite-volume hyperbolic $n$-manifolds $M$. Thus, we have:
\begin{corollary}
The Margulis constant $\mu_3$ is less than or equal to $\muCensus{m003(-3,1)}=0.77442660700998\dots$~.
\end{corollary}

The above argument is evidence for the following conjecture.

\begin{conjecture}
The Margulis constant $\mu_3$ is $\muCensus{m003(-3,1)}$.
\end{conjecture}

David Futer, David Gabai and Andrew Yarmola have independently found the same upper bound for $\mu_3$ (unpublished).

\begin{remark}
Note that \cite{shalen:genericMargulis} seems to indicate that \texttt{m027(-4,1)} is a counter example to the above conjecture. However, Culler most likely used the injectivity radius rather than diameter when obtaining an upper bound for $\muCensus{m027(-4,1)}$ (private communication with Culler and Shalen). When accounting for this factor of two, their result becomes $\muCensus{m027(-4,1)}<1.232$ which is consistent with our verified computation $\muCensus{m027(-4,1)}=1.035360533301757\dots$ and with the above conjecture.
\end{remark}

Having computed the systole for the cusped census manifolds, a natural question is: what is the smallest systole among all manifolds with a specified number $n$ of tetrahedra? For small $n$, Table~\ref{tab:systoles} lists it together with the census manifold achieving it. 

\begin{table}[h]
\begin{center}
\begin{tabular}{|r|c|l|l|}
\hline
Tetrahedra & Manifold & Systole & Dual 1-skeleton\\
\hline
\hline
2 & \texttt{m003} & 0.86255462766.... & \includegraphics[height=0.75cm]{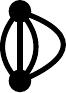}\\ \hline
3 & \texttt{m019} & 0.43153441294... & \includegraphics[height=0.75cm]{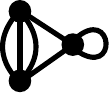}\\ \hline
4 & \texttt{m064} & 0.17593809986... & \includegraphics[height=0.75cm]{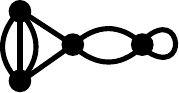} \\ \hline
5 & \texttt{m113} & 0.06942760262... & \includegraphics[height=0.75cm]{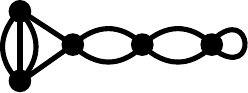} \\ \hline
6 & \texttt{s079} & 0.02679163676... & \multirow{5}{3em}{\includegraphics[height=0.75cm]{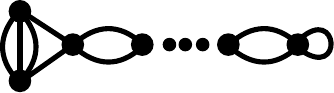}} \\ \cline{1-3}
7 & \texttt{v0159} & 0.01027531350...   & \\ \cline{1-3}
8 & \texttt{t00319} & 0.00393070922... & \\ \cline{1-3}
9 & \texttt{o9\_00639}   & 0.00150226276... & \\ \cline{1-3}
10 & \texttt{o10\_001279} & 0.00057393889... & \\ \cline{1-3}
11 & \texttt{o11\_002559} & 0.00021924349... & \\ \hline
\end{tabular}
\end{center}
\caption{\label{tab:systoles} The census manifolds achieving the shortest systole with the dual 1-skeletons of the respective triangulations.}
\end{table}

\newpage

These data lead us to the following conjecture:

\begin{conjecture} \label{conj:minSys}
Let $\Omega_n$ be the set of orientable hyperbolic manifolds which decompose into $n$ or fewer ideal tetrahedra. Let $F_1=1, F_2=1, F_{n}=F_{n-1}+F_{n-2}$ be the Fibonacci numbers. Then,
$$ \lim_{n\to\infty} \min_{M \in \Omega_n} \systole(M) \cdot \phi^{2n} = 1.266\dots$$ 
where $\systole(M)$ is the systole of $M$ and $\phi=\big(\,1+\sqrt{5}\,\big)\big/2$. Furthermore, for $n\geq 2$, the minimum is achieved by the manifold \texttt{m125(0,0)(}$F_n$\texttt{,}$F_{n+1}$\texttt{)} which has an ideal triangulation with $n$ tetrahedra obtained from the triangulation \texttt{cHcbbdh} by joining the two unglued faces (for $n=2$) or attaching a solid layered torus $(F_{n-1},F_n, F_{n+1})$ (for $n>2$). This triangulation is also the unique minimal ideal geometric triangulation and the canonical cell decomposition of \texttt{m125(0,0)(}$F_n$\texttt{,}$F_{n+1}$\texttt{)}.
\end{conjecture}


Note that the conjecture describes the manifold achieving the minimal systole as both a Dehn filling of the sister \texttt{m125} of the Whitehead link \texttt{m129} and as a triangulation. To see that these two indeed match, we follow the construction in \cite{gueritaud2010canonical} (also see \cite{FHH}): \texttt{m125} has a 4 tetrahedra triangulation (namely, with isomorphism signature \texttt{eLMkbbddddhapu}) such that the link of one of its two ideal vertices consists of two triangles belonging to two distinct tetrahedra. These two tetrahedra form an ``ananas''  which is the cone over the two triangle triangulation of a torus. We can now perform a Dehn filling on \texttt{m125} by removing the ananas (resulting in the triangulation \texttt{cHcbbdh} used by the conjecture) and replacing it by a solid layered torus.

Among the solid layered tori with a fixed number of layers, the one with the highest coefficients is obtained by an alternating pattern resulting in consecutive Fibonacci numbers. Heuristically, this also gives the highest Dehn filling and thus shortest length of the core curve. Thus, it is not surprising that these solid layered tori show up in and dominate the census triangulations with the minimal systole. We expect the pattern to continue and the conjecture to be true.


\begin{remark}
Conjecture~\ref{conj:minSys} uses a particular gluing pattern to attach the layered solid torus to \texttt{cHcbbdh}. Changing the pattern results in \texttt{m125(0,0)(}$F_{n+1}$\texttt{,} $F_n$\texttt{)} which achieves the second smallest systole in $\Omega_n$ for $n=5,\dots, 11$. We conjecture that this generalizes to all $n\geq 5$ and that the quotient of the first and second smallest systole in $\Omega_n$ tends to 1 as $n\to\infty$.
\end{remark}
  
\subsection{New algorithm}

In this paper, we propose a new algorithm to compute the length spectrum given a geometric triangulation of $M$. It builds on the tiling algorithm introduced in \cite{goerner:tiling} which tiles $\H^3\cong\tilde{M}$ by lifts of the geodesic tetrahedra. This is in contrast to the existing algorithm to compute the length spectrum by Hodgson and Weeks in \cite{hwcensus} which tiles by translates of a Dirichlet domain. Thus, the new algorithm avoids the problems that can arise when tiling with a Dirichlet domain. In particular, some Dirichlet domains cannot be verified with interval arithmetic (see \cite[Section~1.2]{goerner:tiling} for a more thorough discussion of the issues arising here).

The implementation of the Hodgson--Weeks algorithm in SnapPy \cite{SnapPy} is not verified and our original motivation was to make it verified using interval arithmetic. However, their algorithm requires a Dirichlet domain and, when trying to verify a Dirichlet domain, we run into the following problem: some Dirichlet domains (such as the ones computed by SnapPy for the census manifolds \texttt{m125} and \texttt{m129}; see \cite[Conjecture~1.2.2]{goerner:tiling}) have finite vertices of valence four or higher. To verify that the combinatorics of the Dirichlet domain about such a vertex is correct, we need to show that four or more planes in hyperbolic 3-space intersect in a common point. This cannot be done with interval arithmetic (unless all plane equations are exactly representable by floating point numbers) and requires exact arithmetic; also see Remark~\ref{rem:intervalEqualities}.

Thus, we instead base the new length spectrum algorithm on the tiling algorithm in \cite{goerner:tiling}. As Section~\ref{subsub:verifiedChain} explains, this allows us to make the entire chain of computations verified using interval arithmetics. This chain starts with verifying the geometric structure on a given ideal or finite triangulation and thus relies on a collaboration involving the third author \cite{hikmot} or previous work of the first author \cite{matthiasVerifyingFinite}, respectively. Note that we cannot use interval arithmetic to determine whether two geodesics have exactly the same length and thus obtain the multiplicity of a length in the length spectrum. However, we can still use interval arithmetic to verify that we have found all geodesics up to a certain length. The output of our algorithm is actually a stream of geodesics and we describe its semantics in Section~\ref{sec:putAlgo}.

Unverified and verified implementations of the new algorithm are available in SnapPy \cite{SnapPy} as of version 3.3\footnote{The methods \texttt{Manifold.length\_spectrum\_alt} and \texttt{Manifold.length\_spectrum\_alt\_gen} which use the new algorithm first appeared in SnapPy version~3.2 but have potential bugs related to the spine radius which are fixed in version~3.3.}. 

\subsection{Speed of the new algorithm} \label{sec:speedNew}

Tiling on a tetrahedral level (rather than with Dirichlet domains) also makes the new algorithm significantly faster in some cases. We show this in the following examples. All computation times are measured using SnapPy~3.3 \cite{SnapPy}, SageMath~10.5 \cite{SageMath} and Linux on an AMD Ryzen 5 7640U at 3.5Ghz.

Let $r$ be the spine radius of the Dirichlet domain that SnapPy computes for a manifold $M$. To compute any part of the length spectrum, the Hodgson--Weeks algorithm needs to a tile a ball of radius at least $2r$ whose volume grows as $e^{4r}=54.598\dots^r$ (see \cite[Proposition~3.5]{hwcensus}). This makes the computation of the length spectrum prohibitively expensive in some cases. As the following example shows, this particularly applies to manifolds with short geodesics or, equivalently, manifolds arising as Dehn fillings along long slopes.
\begin{example} \label{example:fast}
We compute the systole for the census manifold \texttt{o9\_00637} in SnapPy. SnapPy chooses a Dirichlet domain with only 38 faces but a large spine radius of $3.7173\dots$~. With the Hodgson--Weeks algorithm, the computation takes 16h 53m: %
\begin{verbatim}
   >>> M = ManifoldHP("o9_00637")
   >>> M.num_tetrahedra()
   9
   >>> M.length_spectrum(0.002)
   mult  length                    topology  parity
   1     0.00164... - 2.401...*I   circle    +
\end{verbatim} %
\end{example}   
The new algorithm can overcome this Achilles heel of the Hodgson--Weeks algorithm:
\begin{example} \label{ex:newAlgIdeal}
With the new algorithm, the computation for the same manifold \texttt{M} from the previous example now only takes 2.4s: %
\begin{verbatim}
   >>> M.length_spectrum_alt(count=1)[0].length.real()
   0.00164...
\end{verbatim}
\end{example}

Note, however, that the performance of the new algorithm is more sensitive to the choice of input spun-triangulation than that of the Hodgson--Weeks algorithm. This can be regarded as a bug --- or as a feature since we can often accelerate the new algorithm by applying some simple heuristics to the triangulation as discussed in Section~\ref{sec:perfTrigDep}. That is, we can often accelerate the new algorithm by guessing a short geodesic, drilling it (see \cite{goerner:drilling}) and immediately filling the corresponding new cusp. In other words, that is by turning a short geodesic into a core curve. We can often find such a short geodesic among the generators of the (unsimplified) fundamental group. Here is an example:
\begin{example} \label{ex:newAlgSpun}
With the new algorithm, the entire computation for the same manifold \texttt{M} from the previous two examples now only takes 120ms:
\begin{verbatim}
   >>> N = M.drill_word("e")
   >>> N.dehn_fill((1,0),-1)   # equivalent to m129(0,0)(-55,34)
   >>> N.num_tetrahedra()
   4
   >>> N.length_spectrum_alt(count=1)[0].length.real()
   0.00164...
\end{verbatim}
\end{example}
The acceleration in the last example is due to the following reasons:
\begin{itemize}
\item The spun-triangulation \texttt{N} of the given manifold has fewer tetrahedra.
\item When tiling, the new algorithm avoids a large tube about the (short) core curve of \texttt{N}. More precisely, the new algorithm chooses a ``thick-like'' part of the manifold which is the complement of embedded and disjoint cusp neighborhoods and tubes about core curves (like the thick-part of the thick-thin composition); see Section~\ref{sec:generalizedCuspCross}. The new algorithm uses truncated tetrahedra to tile this thick-like part rather than the entire manifold; see Section~\ref{sec:tilingIncompleteAboutPoint}.
\end{itemize}

In particular, small deformations to the geometric structure of a fixed triangulation barely impact the time to find all geodesics up to a given cut-off length with the new algorithm (note though that very large Dehn filling coefficients run into the problem of just computing the core curve; see Example~\ref{ex:largeDehn} and Section~\ref{sec:largeFillCoef} for future work addressing this):
\begin{example}
The following two examples take 82ms and 83ms, respectively:
\begin{verbatim}
    >>> Manifold("m004").length_spectrum_alt(max_len=1.1)
    [Length                            Core curve  Word
     1.087070... + 1.722768...*I       -           bC,
     1.087070... - 1.722768...*I       -           a]
    >>> Manifold("m004(34,55)").length_spectrum_alt(max_len=1.1)
    [Length                            Core curve  Word
     0.000580... - 2.399137...*I       Cusp 0      abCAcBabCAc...BB...,
     1.086777... + 1.722198...*I       -           bC,
     1.087217... - 1.722144...*I       -           a] 
\end{verbatim}
\end{example}

\begin{example} \label{example:censusTimings}
The following table shows the time to compute the systole for all but one manifold of the SnapPy orientable census (including the recent extension to 10 tetrahedra by \cite{li:10tet}). For the new algorithm and cusped manifolds, we apply the heuristics described in Section~\ref{sec:perfTrigDep} to the input triangulation (see function in Section~\ref{sec:codeSurgeryDescription}). For the closed manifolds, we use the geometric spun-triangulations from \cite[\texttt{ClosedManifolds.zip}]{hikmot}.
\begin{center}
\begin{tabular}{r||r|r||r|r|r}
Census & Tetrahedra & Manifolds & Hodgson--Weeks & New algorithm & New algorithm\\
 & & & 212bits & 500bits & verified, 500bits \\
 &  &  & C++ & Python & Python \\ \hline \hline
Cusped & $\leq 9$ & 61911 & 11d 13h 23m & 4h 19m & 6h \phantom{0}3m \\
& 10 & 150730 & >100d\rlap{\footnotemark}\phantom{ 13h 23m} & 13h 50m & 19h 11m \\ \hline
Closed &  & 11030 & 1h \phantom{0}0m & 1h \phantom{0}2m & 1h 34m\\
without \texttt{m007(3,1)} & & & &
\end{tabular}
\end{center}
\footnotetext{Conservative estimate: the Dirichlet domain computed by SnapPy has a spine radius larger than the one in Example~\ref{example:fast} for over 300 of these manifolds.}
Note that the Hodgson--Weeks algorithm takes more time for the single manifold in Example~\ref{example:fast} than the new algorithm took for the entire census up to 9 tetrahedra.
\end{example}

\begin{remark} \label{rem:finAccel}
Future work includes implementing the new length spectrum algorithm for geometric finite triangulation which might further accelerate the algorithm for closed manifolds; see Section~\ref{sec:futureFinTrig}.
\end{remark}

\subsection{Output of the new algorithm} \label{sec:putAlgo}

The length spectrum of an oriented, finite-volume hyperbolic 3-manifold $M$ has infinitely many elements. To reduce it to a finite problem, the Hodgson--Weeks algorithm takes a real cut-off length $\mu$ and lists the elements up to that length. We are taking a different approach here and abstract the length spectrum as a stream. In other words, we provide an iterator that can be queried for the next element of the length spectrum; also see Example~\ref{example:nextUse} later. Note that the abstraction as a stream is not perfect: the precision is given as input to the algorithm and the algorithm only produces finitely many geodesics before failing with an error indicating that the precision is insufficient. Future work might automatically increase the precision as needed. The output stream of the new algorithm is 
\[
(\lambda_0, w_0), (\lambda_1, w_1), \dots
\]
where $\lambda_i$ is the complex length of a primitive, unoriented closed geodesic $\gamma$ represented by a word $w_i$ in the face-pairing representation of $\pi_1(M)$. Here, we assume that the input geometric triangulation $\myTrig$ of $M$ comes with the necessary choices to define a fundamental polyhedron $P$ with labeled face-pairings; see \cite{goerner:tiling} for details.

We focus on verified computations with interval arithmetic from now on. Note that the intervals for the lengths of two geodesics might overlap but this does not prove that they have exactly the same length. Hence, we cannot group the geodesics by length or list them in exactly the right order. However, the algorithm still ensures the following guarantees for the output stream:
\begin{itemize}
\item For each $i$, every primitive, unoriented closed geodesics $\gamma$ in $M$ of real length less than or equal to the left-endpoint $\underline{\RePart(\lambda_i)}$ of $\RePart(\lambda_i)$ has a representative among $(\lambda_0, w_0), \dots, (\lambda_{i-1}, w_{i-1})$.
\item No unoriented closed geodesic $\gamma$ in $M$ appears more than once in the stream.
\item The left-endpoint $\underline{\RePart(\lambda_i)}$ of $\RePart(\lambda_i)$ is non-decreasing. We require this to avoid confusion --- even though it is not strictly necessary for many applications such as the ones we list immediately below. \label{item:ensureProgressProperty}
\item For any primitive, unoriented closed geodesics $\gamma$ in $M$, we can pick a high enough precision such that $\gamma$ has a representative $(\lambda_i, w_i)$ with an arbitrary tight interval for $\lambda_i$ before the algorithm fails (because of insufficient precision).
\end{itemize}

\subsection{Immediate applications} \label{sec:immApp}

In particular, an application can use this stream to compute:
\begin{itemize}
\item The systole which is contained in the interval $\RePart(\lambda_0)$ --- even though $(\lambda_0, w_0)$ might not represent a geodesic realizing the systole. See Examples~\ref{ex:newAlgIdeal} and~\ref{ex:newAlgSpun}.
\begin{proof}
The left-endpoint $\underline{\RePart(\lambda_0)}$ is a lower bound for the systole since 
no geodesic has real length less than $\underline{\RePart(\lambda_0)}$. The right-endpoint $\overline{\RePart(\lambda_0)}$ is an upper bound for the real length of some closed geodesic.
\end{proof}
\item A superset $(\lambda_0, w_0), \dots, (\lambda_{i-1}, w_{i-1})$ of all primitive, unoriented geodesics shorter than a given cut-off length $\mu$ by finding $i$ with $\RePart(\lambda_i)\geq \mu$ (also see Remark~\ref{rem:intermediateElements}). See Example~\ref{ex:cutOff}. \label{item:applicationCutoff}
\item A superset $(\lambda_0, w_0), \dots, (\lambda_{i-1}, w_{i-1})$ of the $j$ shortest primitive, unoriented geodesics by finding $i\geq j$ with
$\max(\RePart(\lambda_0), \dots, \RePart(\lambda_{i-1}))<\RePart(\lambda_i)$ or, equivalently,
$\RePart(\lambda_0) < \RePart(\lambda_i)$, \dots and  $\RePart(\lambda_{i-1})<\RePart(\lambda_i)$. See Examples~\ref{example:nextUse} and~\ref{ex:geoCount}.
\end{itemize}

\begin{example} \label{example:nextUse}
Here is an example of finding the two shortest geodesics (which happen to have the same complex length but interval arithmetic cannot prove it). Note that SageMath \cite{SageMath} is required for verified computations.
\begin{verbatim}
  sage: from snappy import Manifold
  sage: M=Manifold("m125")
  sage: l=M.length_spectrum_alt_gen(verified=True, bits_prec=75)
  sage: g0=next(l); g0
  Length                                      Core curve  Word
  0.962423650119207?  - 3.141592653589793? *I -           cD
  sage: g1=next(l); g1
  0.9624236501192069? - 3.1415926535897933?*I -           a
  sage: g2=next(l)
  sage: g0.length.real().max(g1.length.real()) < g2.length.real()
  True
\end{verbatim}
\begin{remark}
Do not use \texttt{max(g0.length.real(), g1.length().real())} since Python's built-in \texttt{max} produces the \emphasisText{wrong} result for SageMath intervals!
\end{remark}

\begin{example} \label{ex:geoCount}
A convenience method to do the same as in the previous example:
\begin{verbatim}
  sage: Manifold("m125").length_spectrum_alt(count=1, verified=True, bits_prec=75)
  [Length                                      Core curve  Word
   0.962423650119207?  - 3.141592653589793? *I -           cD,
   0.9624236501192069? - 3.1415926535897933?*I -           a]
\end{verbatim}
\end{example}
\end{example}

\begin{example} \label{ex:cutOff}
We can also pass a cut-off length $\lambda$ to the new method:
\begin{verbatim}
    Manifold("m125").length_spectrum_alt(max_len=1.0, verified=True, bits_prec=75)
    [Length                                      Core curve  Word
     0.962423650119207?  - 3.141592653589793? *I -           cD,
     0.9624236501192069? - 3.1415926535897933?*I -           a]
\end{verbatim}
\end{example}

\begin{remark} \label{rem:intermediateElements}
The new length spectrum algorithm can easily be modified to emit intermediate $(\lambda_i, \text{--})$ in the stream that do not correspond to geodesics but still serve as indicator that all geodesics up to real length less than $\RePart(\lambda_i)$ have been enumerated. The above applications can use this as a performance optimization.
\end{remark}

\subsection{Steps of the length spectrum algorithm} \label{sec:defPrelength}

We first consider the case where the triangulation has a complete geometric structure. Let $\Gamma\subset\SO(1,3)$ be such that $M\cong\Gamma\backslash\H^3$ where $\H^3=\{ x \in \R^{1,3} : x_0>0~\mbox{and}~x\cdot x = -1\}$ is the hyperboloid model. Let $P\subset\H^3$ be a fundamental polyhedron for $M$. For simplicity, whenever we write $m\in\Gamma$, we also carry along a corresponding word $w$ in the face-pairing representation corresponding to $P$ --- starting with the generators of $\Gamma$ in the algorithms in \cite{goerner:tiling}.

The new length spectrum algorithm follows similar steps as the Hodgson--Weeks algorithm but differs in the implementation of each step:
\begin{itemize}
\item \label{step:Spine} Compute (an upper bound for) the radius of a \emphasisText{spine} $S\subset P$ (see Algorithm~\ref{algo:spine}). We will use that each closed geodesic intersects the image of $S$ in $M$ (or, if $M$ is incomplete, is a core curve); see Section~\ref{sec:spineDef}.\\
This step differs from Hodgson--Weeks in that the radius is per-tetrahedron rather than for the entire fundamental domain. 
\item \label{step:preLen} Tile to compute a \emphasisText{pre-length spectrum stream $(\mu_0, m_0), (\mu_1, m_1), \dots$ for $S$} (see Algorithm~\ref{algo:preLen}) with $\mu_i\in\R^{\geq 0}$ and $m_i\in\Gamma\setminus\{\Id\}$. It is an intermediate result that includes all geodesics up the cut-off length $\mu_i$ but might also contain parabolics, non-primitives, duplicates and longer geodesics.\\
More precisely, the pre-length spectrum stream satisfies the following for each $i$: Let $\gamma$ be a primitive, unoriented closed geodesic in $M$ of real length $\leq \mu_i$ that intersects $S$ in $M$ (which is automatic if $M$ is complete). Then $\gamma$ has a (not necessarily unique) representative $m_j$ among $m_0, \dots, m_{i-1}$ such that $m_j$ fixes a line $x_0x_1\subset\H^3$ that intersects the spine $S\subset P$.\\
This step differs from Hodgson--Weeks in that we tile with the tetrahedra rather than the fundamental domain using the tiling algorithm from \cite{goerner:tiling}.
\item \label{step:filterPreLen} Filter and sort the pre-length spectrum stream (see Algorithm~\ref{algo:lenSpec}):
\begin{itemize}
\item Sort so that (the left end-point) of the real length of the geodesics is non-decreasing.
\item Remove elements where $m\in\Gamma$ is not fixing a line intersecting the spine $S$. In particular, remove elements where $m\in\Gamma$ is parabolic.
\item Remove elements duplicating an earlier unoriented geodesic --- including multiples of an earlier geodesic.
\end{itemize}
This step differs from Hodgson--Weeks in that we de-duplicate geodesics using the data structure in Section~\ref{sec:setOfGeos} rather than by applying a large set of Deck transformations.
\end{itemize}

In terms of further optimization, 
we have the following question:
\begin{question} \label{question:earlyPrune}
Is there a more efficient algorithm to compute the length spectrum algorithm that prunes early, that is during the tiling to obtain the pre-length spectrum?
\end{question}

\subsection{Overview}

We use the terminology and notation from \cite{goerner:tiling} throughout this paper.

Section~\ref{sec:verified} shows the chain of verified computations leading up to the length spectrum stream in.

Section~\ref{sec:spineDef} explains the role of a spine to compute the length spectrum.

Sections~\ref{sec:spineComplete} to~\ref{sec:setOfGeos} show how to tile with (untruncated) tetrahedra to compute the length spectrum for finite and ideal triangulations with a complete geometric structure: Sections~\ref{sec:spineComplete} to~\ref{sec:filtAndSort} explain each one of the three steps outlined in Section~\ref{sec:defPrelength} already. They are followed by implementation details: utility functions in Section~\ref{sec:geoUtilities} and the data structure to de-duplicate geodesics in Section~\ref{sec:setOfGeos}.

Sections~\ref{sec:generalizationToSpun} to~\ref{sec:coreCurveAvoidance} show how to generalize the length spectrum algorithm to ideal triangulations with an incomplete geometric structure that can be completed by attaching circles. Section~\ref{sec:generalizationToSpun} lists the problems arising from the existence of the incompleteness locus. In particular, the tiling algorithm in \cite{goerner:tiling} produces infinitely many ideal (untruncated) tetrahedra without ever reaching the incompleteness locus. To overcome this, we pick embedded and disjoint cusp neighborhoods and tubes about core curves in Section~\ref{sec:generalizedCuspCross}. We can think of their union as ``thin-like part'' and their complement as ``thick-like part'' of the manifold. The idea is to only tile the ``thick-like part''. More precisely, we use truncated tetrahedra to tile the lift $U\subset\H^3$ of the thick-like part rather than all of $\H^3$. After Section~\ref{sec:spineIncomplete} adjusts the spine to avoid the thin-like part, Section~\ref{sec:tilingIncompleteAboutPoint} can use this idea to modify the tiling algorithm and compute a pre-length spectrum stream for a geometric spun-triangulation. Sections~\ref{sec:addCoreCurves} to~\ref{sec:coreCurveAvoidance} revisit the steps and data structures of the length spectrum algorithm (Algorithm~\ref{algo:lenSpec}) to solve the remaining problems arising from the incompleteness locus.

Note that SnapPy always uses the generalized implementation from Sections~\ref{sec:generalizationToSpun} to~\ref{sec:coreCurveAvoidance}.

Section~\ref{sec:margulis} shows the algorithm to compute the optimal Margulis number of a hyperbolic $3$-manifold.

Section~\ref{sec:cosmetic} verifies the surgery conjecture for the SnapPy census to prove Theorem~\ref{thm:cosmetic}.

Section~\ref{sec:performance} discusses the performance of the new algorithm and compares it to that of the Hodgson--Weeks algorithm. In particular, we point users of the new algorithm to Section~\ref{sec:perfTrigDep} for choosing spun-triangulations resulting in better performance.

We conclude with a discussion of support for geometric finite triangulations as future work in Section~\ref{sec:futureFinTrig}.

Appendix~\ref{App:Notation} lists the notation used throughout this paper.

\subsection{Acknowledgements}

The authors thank David Futer, Jessica Purcell, and Saul Schleimer for their encouragement and interest in our work and for testing our algorithm. The authors also grateful to Henry Segerman for many helpful discussions on topics throughout the paper and Ken Baker and Qiuyu Ren for helpful discussions about the exceptional fillings that needed to be analyzed by hand for the cosmetic surgery conjecture. The authors especially wish to acknowledge Marc Culler and Nathan Duffield for their continuing work on and maintenance of SnapPy \cite{SnapPy}. This project was initiated using support from U.S. National Science Foundation grants DMS 1107452, 1107263, 1107367, ``RNMS: GEometric structures And Representation varieties (the GEAR Network)'' which funded a short-term visit for the authors at the University of California Davis. The authors also thank the Oklahoma State University for hosting them as a group in the formational stages of this project.

\subsection{Use of AI tools}

The authors used ChatGPT in the late stages of writing to tweak the grammar. No LLMs were used to structure arguments or generate citations. In particular, the authors take full intellectual responsibility for the content of this article.

\section{Verified computation} \label{sec:verified}

\subsection{Conventions} \label{subsec:intervalNotation}

We use the same conventions as \cite[Section~2.2]{goerner:tiling}. In particular, we denote by $\underline{a}$ and $\overline{a}$ the left and right endpoint of an interval $a=[\underline{a}, \overline{a}]$ and $a<b$ means $\overline{a} < \underline{b}$.

\subsection{Chain of verified computation} \label{subsub:verifiedChain}

The chain of verified computations with interval arithmetic leading to the length spectrum is as follows:
\begin{enumerate}
\item Verified geometric structure for the given triangulation $\myTrig$.\\
Here, we are primarily focusing on the geometric structure given by cross ratios $z_0,\dots, z_{n-1}\in\C\setminus\{0,1\}$ (also called shapes) parameterizing ideal geodesic tetrahedra. These cross ratios are given implicitly as (the unique) solution to the gluing equations. As such, we need to use methods such as the interval Newton method or the Krawczyk test to obtain (complex) intervals guaranteed to contain a solution; see \cite{hikmot}. Note that these methods require a system of equations with an invertible (and, in particular, square) Jacobian matrix. Since the full system of edge and cusp gluing equations is, however, overdetermined, we need to remove some of the gluing equations. \cite{moser} shows that we can do this without introducing unsound solutions; also see \cite[Appendix]{matthiasVerifyingFinite}. For finite triangulations, we refer the reader to \cite{matthiasVerifyingFinite}.
\end{enumerate}
In all subsequent steps in the chain, new quantities (such as the vertices of a developed fundamental polyhedron) are given explicitly in terms of previous quantities (such as the shapes). If the interval implementation of each operator fulfills the inclusion principle, then, inductively, the intervals for all new quantities also contain their true values (that is, the values obtained when using exact arithmetic). Furthermore, we can use these intervals to prove inequalities (see \cite[Section~2.2]{goerner:tiling}) necessary to ensure the correctness. The subsequent steps are:
\begin{enumerate}[resume]
\item \cite[Section~4.7]{goerner:tiling} computes a developed fundamental polyhedron $P$.
\item Algorithm~\ref{algo:spine} computes (lower bounds for) the radii $r(S_t)$ for a spine $S$.
\item Algorithm~\ref{algo:preLen} computes the pre-length spectrum stream using the verified tiling algorithm \cite[Algorithm~7.3.1]{goerner:tiling}.
\item Algorithm~\ref{algo:lenSpec} then processes the pre-length spectrum to compute the length spectrum stream.
\end{enumerate}

We revisit these steps later to generalize to geometric spun triangulations.

\begin{remark} \label{rem:intervalEqualities}
Note that we cannot prove equalities using interval arithmetic (unless the intervals have zero length which can only happen if the quantities can be exactly represented by a floating point number and, thus, are necessarily rational).
This is, for example, necessary to make the algorithm in SnapPy to compute a Dirichlet domain verified when the domain has a finite vertex adjacent to more than 3 faces. That is, we can obtain intervals for the vertex coordinates by intersecting planes supporting 3 of the adjacent faces. However, we cannot use interval arithmetic to prove that these vertex coordinates fulfill the plane equation for any of the remaining faces.
\end{remark}

\section{Spine and tiling radius} \label{sec:spineDef}

We use the terminology and notation from \cite[Section~4]{goerner:tiling}. Let $M$ be an oriented finite-volume hyperbolic $3$-manifold that is either complete or can be completed to $M_\text{filled}$ by attaching circles we call \emphasisText{core curves}. Let $\myTrig$ be an ideal or finite triangulation with charts embedding each tetrahedron $T_t$ into $\H^3$ such that they form a developed fundamental polyhedron $P=\cup T_t\subset\H^3$ for $M\subset M_\text{filled}$ with respect to $\Gamma\subset\SO(1,3)$ where $M_\text{filled}\cong\Gamma\backslash\H^3$.

\subsection{Spine} \label{sec:spineDefDef}

Here, a \emphasisText{spine} $S$ for the triangulation $\myTrig$ is the union of a topologically embedded butterfly $S_t$ in each tetrahedron $T_t$ as in \cite[Figure~1.5]{matveev:AlgorithmicTop} matching up with the face-pairings. In other words, a spine $S$ is an embedding of the dual $2$-skeleton of $\myTrig$ in $M$. For concreteness, we consider spines consisting of totally geodesic triangles until Section~\ref{sec:generalizationToSpun}. We abuse notation and denote by $S_t$ both the embedded butterfly in the manifold $M$ or the developed fundamental polyhedron $P\subset \H^3$. Note that each connected component of $M\setminus S$ corresponds to a vertex of $\myTrig$ and topologically is either a ball if the vertex is finite or the product of the vertex link and $(0,1)$ otherwise. We refer the reader to \cite[Chapter~1]{matveev:AlgorithmicTop} for the more general definition of spine.  We pick a basepoint $\vec{s}_t\in S_t$ for each $S_t$ and denote by $r(S_t)$ the \emphasisDef{spine radius}, that is the radius of $S_t\subset\H^3$ about $\vec{s}_t$.

Note that the pre-length spectrum algorithm only needs an upper bound $\overline{r}(S_t)$ for the spine radius $r(S_t)$ and not the spine itself. However, we will construct a spine in the proofs of Propositions~\ref{prop:spine} and~\ref{prop:spunSpine} to show that the upper bound used in these propositions is correct.

\subsection{Tiling radius}

We can use the spine radius to relate the tiling radius and the cut-off length for the length spectrum. For the new algorithm, we use a version of \cite[Proposition~3.5]{hwcensus} that is per-tetrahedron and uses $d(\vec{s}_t, mS_t)$ instead of $d(\vec{s}_t, m\vec{s}_t)$ and thus has different constants:
\begin{proposition} \label{prop:spineAndTilingRadius}
Let $M$, $M_\text{filled}$, $\myTrig$, $P=\cup T_t$ and $\Gamma$ as above. Pick a spine $S=\cup S_t$ with $S_t\subset T_t$ for each tetrahedron $T_t$.
Let $\mu>0$. Let $\gamma$ be a primitive, unoriented closed geodesic in $M_\text{filled}$ of real length $\leq\mu$. Then $\gamma$ is either a core curve or intersects $S$ in, say, $S_t$. In the latter case, $\gamma$ is represented by $m\in\Gamma$ with
\[
d(\vec{s}_t, mS_t)\leq \cosh^{-1}\left(\cosh(r(S_t))\cosh(\mu)\right).
\]
\end{proposition}

\begin{figure}[ht]
\begin{center}
\begingroup%
  \makeatletter%
  \providecommand\color[2][]{%
    \errmessage{(Inkscape) Color is used for the text in Inkscape, but the package 'color.sty' is not loaded}%
    \renewcommand\color[2][]{}%
  }%
  \providecommand\transparent[1]{%
    \errmessage{(Inkscape) Transparency is used (non-zero) for the text in Inkscape, but the package 'transparent.sty' is not loaded}%
    \renewcommand\transparent[1]{}%
  }%
  \providecommand\rotatebox[2]{#2}%
  \newcommand*\fsize{\dimexpr\f@size pt\relax}%
  \newcommand*\lineheight[1]{\fontsize{\fsize}{#1\fsize}\selectfont}%
  \ifx\svgwidth\undefined%
    \setlength{\unitlength}{359.12533257bp}%
    \ifx\svgscale\undefined%
      \relax%
    \else%
      \setlength{\unitlength}{\unitlength * \real{\svgscale}}%
    \fi%
  \else%
    \setlength{\unitlength}{\svgwidth}%
  \fi%
  \global\let\svgwidth\undefined%
  \global\let\svgscale\undefined%
  \makeatother%
  \begin{picture}(1,0.30432017)%
    \lineheight{1}%
    \setlength\tabcolsep{0pt}%
    \put(0,0){\includegraphics[width=\unitlength,page=1]{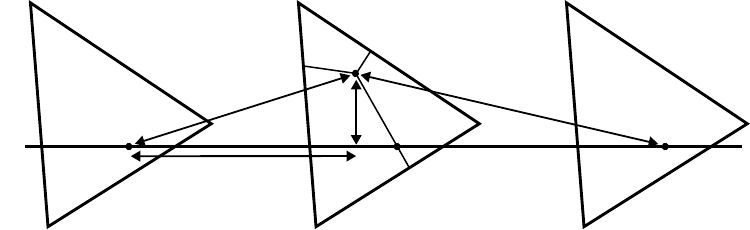}}%
    \put(0.45461159,0.22073605){\makebox(0,0)[lt]{\lineheight{1.25}\smash{\begin{tabular}[t]{l}$p$\end{tabular}}}}%
    \put(-0.00050826,0.10550535){\makebox(0,0)[lt]{\lineheight{1.25}\smash{\begin{tabular}[t]{l}$\gamma$\end{tabular}}}}%
    \put(0.52711522,0.12290651){\makebox(0,0)[lt]{\lineheight{1.25}\smash{\begin{tabular}[t]{l}$s$\end{tabular}}}}%
    \put(0.14160874,0.12290651){\makebox(0,0)[lt]{\lineheight{1.25}\smash{\begin{tabular}[t]{l}$ms$\end{tabular}}}}%
    \put(0.86009832,0.12496826){\makebox(0,0)[lt]{\lineheight{1.25}\smash{\begin{tabular}[t]{l}$m^{-1}s$\end{tabular}}}}%
    \put(0.28774431,0.07258973){\makebox(0,0)[lt]{\lineheight{1.25}\smash{\begin{tabular}[t]{l}$\leq\mu$\end{tabular}}}}%
    \put(0.42659587,0.14789877){\makebox(0,0)[lt]{\lineheight{1.25}\smash{\begin{tabular}[t]{l}$\leq r$\end{tabular}}}}%
  \end{picture}%
\endgroup%

\end{center}
\caption{Tiling radius from spine radius. Compare to \cite[Figure~2]{hwcensus}. \label{fig:radFromSpine}}
\end{figure}

\begin{proof}
Assume that $\gamma$ does not intersect $S$. Then $\gamma$ is contained in a dual 3-domain, that is a connected component of $M\setminus S$. Thus, $\gamma$ is a peripheral curve in $M$ and, thus, either contractible or a core curve in $M_\text{filled}$. Let $\lambda$ be complex the length of $\gamma$. Let $s\in M$ be a point where $\gamma$ intersects $S_t$. Let $p=\vec{s}_t$ and $r=r(S_t)$. Lift $s$, $p$ and $\gamma$ to $\H^3$ such that $s, p\in T_t\subset P$ and $s\in\gamma\subset\H^3$. Let $m\in\Gamma$ be one of the two primitive elements fixing $\gamma$. See Figure~\ref{fig:radFromSpine}. The points $p$, $ms$ and $m^{-1}s$ span a hyperbolic triangle. Split the triangle along the line $sp$. At least one of the resulting triangles has angle $\leq \pi / 2$ at $s$. Assume it is the triangle $s$, $p$ and $ms$ by replacing $m$ by $m^{-1}$ if necessary. The edge between $s$ and $p$ has length $\leq r$. The edge between $s$ and $ms$ has length $\RePart(\lambda)\leq \mu$. The edge between $p$ and $ms$ has length $d(p,ms)$ less than the hypothenuse of a right triangle with lengths $r=r(S_t)$ and $\mu$. We have $d(\vec{s_t},mS_t)=d(p, mS_t)\leq d(p, ms)$.
\end{proof}


\section{Computing a spine radius for a complete triangulation} \label{sec:spineComplete}

In this section, we show how to compute a spine radius $r(S_t)$ for each tetrahedron when the geometric structure is complete. Note that we need to construct a spine $S\subset P$ to prove that the spine radius is sufficient, but we do not need to know the actual spine to compute the length spectrum.

\subsection{Tetrahedral vertex vectors} \label{sec:spineVertexVectors}

We allow $\myTrig$ to be ideal or finite. Let $P$ be a developed fundamental polyhedron for $\myTrig$ as defined in \cite[Section~4.7]{goerner:tiling}. In the ideal case, we start by picking cusp neighborhoods of equal area for all cusps. The choice of cusp neighborhoods gives a consistent lift of each the ideal vertex of each tetrahedron $T_t$ to a light-like vector $\vec{v}_t^v\in\posLightCone$ such that the horoball $B(\vec{v}_t^v)$ corresponds to the respective cusp neighborhood. To compute the $\vec{v}_t^v$, we refer the reader to \cite[Section~4.6]{goerner:tiling}. In the finite case, the vertex $\vec{v}_t^v \in \H^3$ does not need a lift. In both cases, the resulting $\vec{v}_t^v$ are matching under the face-pairing matrices $g^f_t$.

\subsection{Spine radius} \label{sec:spinePoints}

\begin{algorithm}[h]
\begin{tabular}{rp{15.4cm}}
{\bf Input:} & Developed fundamental polyhedron $P=\cup T_t\subset\H^3$ with vertices $\vec{v}_t^v$ matching under the face-pairing matrices $g^f_t$.\\\
{\bf Output:} & For a spine $S\subset P$: For each tetrahedron $T_t$, basepoint $\vec{s}_t\in S_t$ and radius $r(S_t)\in\R^+$ (about $\vec{s}_t$) of the restriction of $S_t=S\cap T_t$ to $T_t$.\\
{\bf Algorithm:}\\
\end{tabular}
\begin{algorithmSteps}
\item For each tetrahedron $T_t$:
\begin{algorithmSteps}
\item $\vec{s}_t$ is the incenter of $T_t$.
\item $\vec{s}_t^{kl}=(\vec{v}_t^k+\vec{v}_t^l)\postNorm$. This is the point on the edge $kl$ with equal (signed) distance to (the horoballs about) $\vec{v}_t^k$ and $\vec{v}_t^l$.
\item $r(S_t)=\max_{0\leq k<l\leq 3}\big\{d(\vec{s}_t,\vec{s}_t^{kl})\big\}$.
\end{algorithmSteps}
\end{algorithmSteps}
\caption{Spine radius for a complete triangulation. \label{algo:spine}}
\end{algorithm}

Using the choices for $\vec{v}_t^v$ from the previous section, Algorithm~\ref{algo:spine} computes the information about a spine needed to compute the length spectrum.

\begin{proposition}\label{prop:spine}
Let $P=\cup T_t$ be a developed fundamental polyhedron for a complete hyperbolic 3-manifold $M$ and $r(S_t)$ as in Algorithm~\ref{algo:spine}. There is a subset $S_t\subset T_t$ for each tetrahedron $T_t$ such that
\begin{itemize}
\item $S=\cup S_t$ forms a spine in $M$,
\item the incenter $\vec{s}_t$ of $T_t$ is in $S_t$ and
\item the radius of each $S_t$ about $\vec{s}_t$ matches the $r(S_t)$.
\end{itemize}
\end{proposition}

\begin{proof}
Let $\vec{s}_t^f\in\H^3$ be the incenter of the triangle spanned by the $\vec{s}_t^{kl}$ with $k,l\not=f$ computed in Algorithm~\ref{algo:spine}. Let $S_t\subset T_t$ be the union of the 12 triangles spanned by $\vec{s}_t$, $\vec{s}_t^f$ and $\vec{s}_t^{kl}$ with $k,l\not=f$ which are in $T_t$, on face $f$ of $T_t$ and on the adjacent edge $kl$, respectively; see Figure~\ref{fig:spine}. The vertices $\vec{v}^v_t$ are matching under the face-pairing matrices $g_t^f$. The points $\vec{s}_t^{kl}$ and $\vec{s}_t^f$ for edges and faces are computed from the adjacent vertices $\vec{v}^v_t$ in an $\SO(1,3)$-equivariant way. Thus, the boundaries of each $S_t\subset T_t$ are matching under the face-pairing matrices $g_t^f$.
\begin{figure}[h]
\begin{center}
\begingroup%
  \makeatletter%
  \providecommand\color[2][]{%
    \errmessage{(Inkscape) Color is used for the text in Inkscape, but the package 'color.sty' is not loaded}%
    \renewcommand\color[2][]{}%
  }%
  \providecommand\transparent[1]{%
    \errmessage{(Inkscape) Transparency is used (non-zero) for the text in Inkscape, but the package 'transparent.sty' is not loaded}%
    \renewcommand\transparent[1]{}%
  }%
  \providecommand\rotatebox[2]{#2}%
  \newcommand*\fsize{\dimexpr\f@size pt\relax}%
  \newcommand*\lineheight[1]{\fontsize{\fsize}{#1\fsize}\selectfont}%
  \ifx\svgwidth\undefined%
    \setlength{\unitlength}{274.43348081bp}%
    \ifx\svgscale\undefined%
      \relax%
    \else%
      \setlength{\unitlength}{\unitlength * \real{\svgscale}}%
    \fi%
  \else%
    \setlength{\unitlength}{\svgwidth}%
  \fi%
  \global\let\svgwidth\undefined%
  \global\let\svgscale\undefined%
  \makeatother%
  \begin{picture}(1,0.47556453)%
    \lineheight{1}%
    \setlength\tabcolsep{0pt}%
    \put(0,0){\includegraphics[width=\unitlength,page=1]{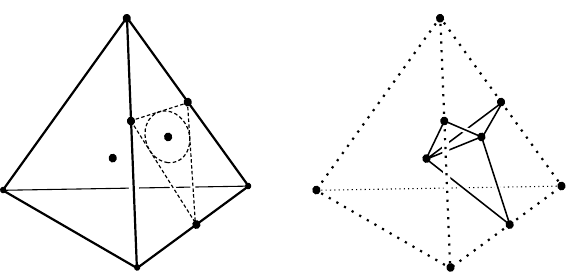}}%
    \put(0.21100023,0.46205357){\makebox(0,0)[lt]{\lineheight{1.25}\smash{\begin{tabular}[t]{l}$\vec{v}_t^k$\end{tabular}}}}%
    \put(0.45031527,0.14002873){\makebox(0,0)[lt]{\lineheight{1.25}\smash{\begin{tabular}[t]{l}$\vec{v}_t^l$\end{tabular}}}}%
    \put(0.88315259,0.30570365){\makebox(0,0)[lt]{\lineheight{1.25}\smash{\begin{tabular}[t]{l}$\vec{s}_t^{kl}$\end{tabular}}}}%
    \put(0.70987752,0.17475676){\makebox(0,0)[lt]{\lineheight{1.25}\smash{\begin{tabular}[t]{l}$\vec{s}_t$\end{tabular}}}}%
    \put(0.85867233,0.21357402){\makebox(0,0)[lt]{\lineheight{1.25}\smash{\begin{tabular}[t]{l}$\vec{s}_t^f$\end{tabular}}}}%
    \put(0.33523392,0.30263753){\makebox(0,0)[lt]{\lineheight{1.25}\smash{\begin{tabular}[t]{l}$\vec{s}_t^{kl}$\end{tabular}}}}%
    \put(0.16195886,0.17169064){\makebox(0,0)[lt]{\lineheight{1.25}\smash{\begin{tabular}[t]{l}$\vec{s}_t$\end{tabular}}}}%
  \end{picture}%
\endgroup%

\end{center}
\caption{Three of the 12 triangles of the spine $S_t$ in the tetrahedron $T_t$.\label{fig:spine}}
\end{figure}
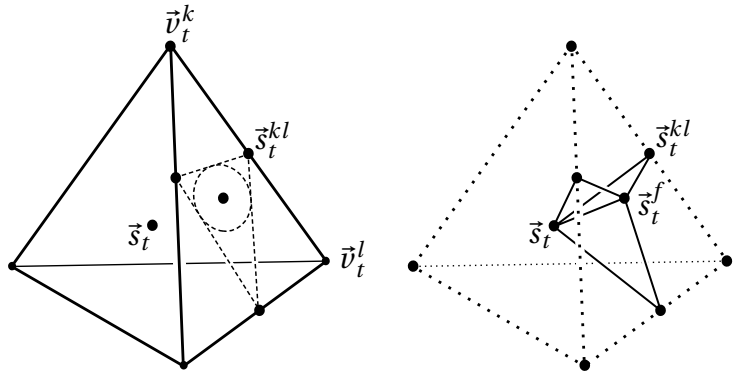
\newpage
Consider only those 3 of the 12 triangles belonging to face $f$. They are contained in the tetrahedron $Q$ spanned by $\vec{s}_t$ and $\vec{s}_t^{kl}$ with $k,l\not=f$. The maximal distance of a point in $Q$ to $\vec{s}_t$ is achieved by a vertex of $Q$, that is $\vec{s}_t^{kl}$ with $k,l\ne f$.
\end{proof}

\section{Pre-length spectrum stream} \label{sec:prelengthSpectrumStream}

In this section, we introduce Algorithm~\ref{algo:preLen} to compute a pre-length spectrum stream $(\mu_0, m_0), (\mu_1, m_1), \dots$ as defined in Section~\ref{sec:defPrelength}.

\subsection{Algorithm} \label{sec:preLengthAlgo}

Since we invoke the tiling algorithm \cite[Algorithm~7.3.1]{goerner:tiling}, we remind the reader of its interface while referring to \cite{goerner:tiling} for further details. Recall that the tiling algorithm can take different types of objects as input. In this section, however, we only need the case where the input is a basepoint $K$ in a tetrahedron $T_t\subset\H^3$. Given this input, the output of tiling algorithm is a stream $(r_0, m_0T_{t_0}), (r_1, m_1T_{t_1}), \dots$ where $m_iT_{t_i}$ are distinct lifted tetrahedra in $\H^3$ and each $r_i>0$ is a ``tiling radius''. That is, $r_i$ (or its left endpoint $\underline{r}_i$ when using intervals) is a lower bound for the radius of the closed ball about $K$ that is covered by the first $i$ lifted tetrahedra $m_0T_{t_0}, \dots, m_{i-1}T_{t_{i-1}}$ in the stream.
\begin{algorithm}[h]
\begin{tabular}{rp{15.4cm}}
{\bf Input:} & Developed fundamental polyhedron $P=\cup T_t\subset\H^3$ for hyperbolic manifold $M$.\\
& For a spine $S\subset P$: For each tetrahedron $T_t$, a basepoint $\vec{s}_t\in S_t$ and the radius $r(S_t)\in\R^+$ (about $\vec{s}_t$) of the restriction $S_t=S\cap T_t$ of $S$ to $T_t$.\\
{\bf Output:} & A pre-length spectrum stream $(\mu_0,m_0),(\mu_1,m_1),\dots$ for $S$.\\
{\bf Algorithm:}\\
\end{tabular}
\begin{algorithmSteps}
\item For each tetrahedron $T_t\subset P$, we compute the following pre-length spectrum stream $(\mu_0,m_0),(\mu_1,m_1),\dots$ for $S_t$: \label{step:perTetrahedronPreLen}
\begin{algorithmSteps}
\item Iterate $(r_i,m_iT_{t_i})$ over the output stream of the tiling algorithm about $K=\vec{s}_t$ with seed $\Id T_t (=T_t)$; see \cite[Algorithm~7.3.1]{goerner:tiling}:
\begin{algorithmSteps}
\item If $m_i\ne\Id$ and ${t_i}=t$:
\begin{algorithmSteps}
\item Emit $(\mu_i, m_i)$ where
\[
\mu_i=\cosh^{-1}\left(\max\Big(1, \frac{\cosh r_i }{\cosh r(S_t)}\Big)\right).
\]
\end{algorithmSteps}
\end{algorithmSteps}
\end{algorithmSteps}
\item Interleave the above streams by picking from the stream whose next element has the lowest $\underline{\mu_j}$. If there is a tie between multiple streams, pick from the first (or any) of the tied streams.
\end{algorithmSteps}
\caption{Pre-length spectrum stream. \label{algo:preLen}}
\end{algorithm}

\begin{proposition} \label{prop:perTetPreLen}
Consider the stream produced in Step~\ref{step:perTetrahedronPreLen} of Algorithm~\ref{algo:preLen} for tetrahedron $T_t$. This stream is a pre-length spectrum stream in the sense of Section~\ref{sec:defPrelength} when replacing $S$ by $S_t$. 
\end{proposition}

\begin{proof}
Let $\mu_i$ be as computed from $r_i$ in the algorithm. Let $\gamma$ be a closed geodesic intersecting $S_t\subset M$ of real length less than $\mu_i$. By Proposition~\ref{prop:spineAndTilingRadius}, $\gamma$ has a representative $m\in\Gamma$ with $d(\vec{s}_t, mS_t) \leq r_i$. Since $d(\vec{s}_t, mT_t) \leq d(\vec{s}_t, mS_t)$, we have $d(\vec{s}_t, mT_t)\leq r_i$. Since the $m_0T_{t_0}, \dots, m_{i-1}T_{t_i-1}$ from the tiling algorithm cover $\overline{B}_{r_i}(K)$, they include $mT_t$. Thus, $m$ has already been emitted to the output stream earlier.
\end{proof}

\begin{proposition}
The output $(\mu_0, m_0), (\mu_{1},m_{1}), \dots$ of Algorithm~\ref{algo:preLen} is a pre-length spectrum stream for $S$.
\end{proposition}

\begin{proof}
This follows from Proposition~\ref{prop:perTetPreLen}, $S=\cup S_t$ and the following property of the interleaving process: Let $(\mu^t_0, m^t_0), (\mu^t_1,m^t_1), \dots$ be the input streams and $(\mu_0, m_0),$ $(\mu_1, m_1), \dots$ be the output stream. By the time we encounter $(\mu_i, m_i)$ in the output, we have taken from each input at least all elements up to the first one $(\mu^t_j, m^t_j)$ with $\underline{\mu^t_j}\geq \underline{\mu_i}$.
\end{proof}

\begin{remark}
We expect the $\mu_i$ to be non-decreasing. However, there might be fluctuations in the lengths of the computed intervals for $\mu_i$. Thus, we do not guarantee that the $\underline{\mu_i}$ in the output of Algorithm~\ref{algo:preLen} are non-decreasing as the underlying tiling algorithm did not make the respective guarantee to begin with; see \cite[Section~7.4]{goerner:tiling}. At each point, we still have that $\underline{\mu_i}$ is a lower bound on the length of geodesics found so far. It might just be a lower bound worse than one computed earlier. We could enforce the non-decreasing property by taking the maximum. However, we choose not to since this is not needed for most applications.
\end{remark}


\section{Filter and sort the pre-length spectrum} \label{sec:filtAndSort}

This section introduces the algorithm to produce the length spectrum stream from the pre-length spectrum stream when the triangulation $\myTrig$ has a complete geometric structure.

As a warm-up case, we would compute the length spectrum up to a given cut-off length $\mu$ as follows: Find the pre-length spectrum up to cut-off length $\mu$. That is, find the first $\mu_i\geq\mu$ in the pre-length spectrum stream and take all $m_j$ with $j<i$. Compute the complex length $\lambda_j$ for each such $m_j$. Sort the pairs $(\lambda_j, m_j)$ by $\underline{\RePart(\lambda_j)}$. We can do this by, for example, first adding them all to a priority queue $Q$ and then popping them off one by one.

We need to reject pairs $(\lambda_j, m_j)$ where $m_j$ is parabolic. Note that our implementation does not actually determine whether $m_j$ is parabolic. Instead, it rejects any pair $(\lambda_j,m_j)$ where $m_j$ is either parabolic or fixes a line that does not intersect the spine $S$. We can reject the latter since there is another representative $(\lambda_{j'}, m_{j'})$ of the same geodesic ($m_{j'}$ conjugate to $m_j$) in the pre-length spectrum intersecting the spine $S$.
For details, we refer the reader to the discussion of Steps~\ref{algStep:fixedPts} and~\ref{algStep:rejectFarSpine} in Section~\ref{sec:innerLoopChecks}.

We also need to reject pairs $(\lambda_j, m_j)$ where $m_j$ duplicates a previous geodesic or is a non-primitive. We can detect this by checking whether $m_j$ is a conjugate of a power $m_{j'}^k$ with $k\ne 0$ of a previous $m_{j'}$ with $j'<j$. Performing this check requires finding a bound on $k$ and finding a finite subset $\Gamma'\subset\Gamma$ that $m_j$ and $m_{j'}^k$ are related through conjugation by an element in $\Gamma'$ if and only if they are related by an element in $\Gamma$. Such a bound for $k$ is given by $\RePart(\lambda_j)/\RePart(\lambda_{j'})$. Such a subset $\Gamma'$ can be computed by applying the tiling algorithm to the lines fixed by $m_j$ and $m_{j'}$. We refer the reader to Sections~\ref{sec:matrixPowerEquality} and~\ref{sec:setOfGeos} for details.

Note that we also might have geodesics $(\lambda_j, m_j)$ that are longer than the cut-off length, that is $\RePart{\lambda_j} > \mu$. We also need to reject such a $(\lambda_j,m_j)$ for the following reasons: we give the wrong impression that we also found all geodesics with real length between $\mu$ and $\RePart{\lambda_j}$; $m_j$ can be a multiple of a geodesic but we are not detecting this since we have not seen the corresponding primitive $m_{j'}$.

\subsection{Algorithm}

Algorithm~\ref{algo:lenSpec} follows this idea but as an online algorithm with input and output being a stream.
The conversion to an online algorithm requires us to answer: Is it safe to emit a geodesic to the output stream or do we need to consume more elements from the input stream? As we have pointed out above, it is only safe to emit a geodesic if we can ensure that we have seen a representative for all geodesic shorter than the geodesic we are about to emit. This translates into $\RePart(\lambda) < \mu_i$ where $\lambda$ is the complex length of the next geodesic we consider emitting and $\mu_i$ is the latest element we took from the pre-length spectrum stream.

\newsavebox{\innerLoopBox}
\savebox{\innerLoopBox}{%
\begin{minipage}{16.3cm}
\begin{algorithmSteps}
\item Remove $(\lambda, m, c)$ from $Q$.
\item $\big[$ If $c=\top$, skip to Step~\ref{algStep:ensureProgress}. $\big]$
\item If $\RePart(\lambda) + \varepsilon < \RePart(\lambda_\text{prev})$: Continue to next iteration (an optimization).\label{algStep:perf}
\item Let $x_0, x_1\in\posLightCone$ be the (lifted) fixed points of (possibly parabolic) $m$ (see Section~\ref{sec:compFixedPoints}). \label{algStep:fixedPts}
\item If $d(\vec{s}_t,x_0x_1)>r(S_t)$ for every tetrahedron $T_t$: Continue to next iteration (rejects $m$ fixing a line $x_0x_1$ not intersecting the spine, including parabolic $m$). \label{algStep:rejectFarSpine}
\item If $\overline{\RePart(\lambda)} - \underline{\RePart(\lambda)} > \underline{\RePart(\lambda)} / 4$: Fail (if using intervals).\label{algStep:rejectAmbi}
\item $\big[$ If $m$ corresponds to (a multiple of) a core curve (see Section~\ref{sec:coreCurveDetection}): Continue to next iteration. $\big]$ \label{algStep:rejectCoreCurve}
\item If $m$ is a multiple (including $m$ and $m^{-1}$) of (a conjugate of) an element in $\setOfGeodesics$: Continue to next iteration. \label{algStep:rejectDuplicate}
\item Add $m$ to $\setOfGeodesics$.
\item If not $\underline{\RePart(\lambda)}\geq \underline{\RePart(\lambda_\text{prev})}$: Fail.
\label{algStep:ensureProgress}
\item Emit $(\lambda, m)$.
\item $\lambda_\text{prev}\leftarrow \lambda$.
\end{algorithmSteps}
\end{minipage}
}

\begin{algorithm}[ht]
\begin{center}
\begin{tabular}{rp{15.4cm}}
{\bf Input:} & Developed fundamental polyhedron $P=\cup T_t$; $\vec{s}_t$ and $r(S_t)\in\R^+$ for spine $S\subset P$.\\
& Pre-length spectrum stream $(\mu_0, m_0), (\mu_1, m_1),\dots$ for $S$. \\
{\bf Output:} & Length spectrum stream (see Section~\ref{sec:putAlgo}).\\
{\bf Note:} & $\varepsilon = 0$ if using intervals.\\
& Ignore [ \dots ] for complete geometric triangulations.\\
{\bf Algorithm:}\\
\end{tabular}
\begin{algorithmSteps}
\item[Q] Priority queue of $(\lambda, m, c)$ where $m\in\Gamma$ and $\lambda$ is the complex translation length of $m$ [and $c$ a boolean to flag a known core curve]. Use the $\underline{\RePart(\lambda)}$ as key.
\item[\setOfGeodesics] Set of unoriented geodesics represented by $m\in\Gamma$ (see Section~\ref{sec:setOfGeos} [and Section~\ref{sec:coreCurveAvoidance}]).
\item $\big[$ Add core curves $(\lambda, m, \top)$ to $Q$ (see Section~\ref{sec:addCoreCurves}). $\big]$ \label{step:addCoreCurves}
\item $\lambda_\text{prev}\leftarrow 0$
\item Iterate $(\mu_i, m_i)$ over the pre-length spectrum stream:
\begin{algorithmSteps}
\item While $Q$ is non-empty and $\RePart(\lambda) + \varepsilon < \mu_i$ for the head $(\lambda, m, c)$ of $Q$:
\vspace{-0.7cm}
\begin{adjustwidth}{-0.6cm}{}
\[
\begin{minipage}{0.5cm}
\rotatebox{90}{Inner~~loop}
\end{minipage}
\left\{\hspace{-0.3cm}\usebox{\innerLoopBox}\right.
\]
\end{adjustwidth}
\vspace{-0.2cm}
\item $\lambda\leftarrow$ (possibly zero) complex translation length of $m_i$ (see Section~\ref{sec:computeLen}).
\item Add $(\lambda, m_i, \bot)$ to $Q$.
\end{algorithmSteps}
\end{algorithmSteps}
\vspace{-0.4cm}
\end{center}
\caption{Length spectrum. \label{algo:lenSpec}}
\end{algorithm}

Recall that a ``Fail'' in a verified computation and, in particular, in Algorithm~\ref{algo:lenSpec} indicates that we need to compute better intervals for the geometric structure and rerun all subsequent computations with higher precision.

\myComment{

 First, we embellish each $(\mu_i,m_i)$ by computing a quadruple $(\mu_i, \lambda_i, m_i, c)$, where $\mu_i$ and $m_i$ comes from the pre-length spectrum $\lambda_i$ is the complex length of $m_i$ and $c$ is a boolean flag used for marking a core curve. (This flag will be used for geometric spun-triangulations.)  

\todo{up to here}
It uses a priority queue $Q$ (such as a binary heap already used in \cite[Section~7.3]{goerner:tiling}) to sort the pre-length spectrum.
Note that for each $\mu_i$, we have that any geodesic of length $\underline{\mu_i}$ has a representative that 1) previously appeared in the pre-length spectrum and 2) intersects the spine. We also note at this stage parabolic elements (which would have real length zero) could appear in the pre-length spectrum.

By comparing $\underline{\mu_i}$ against $\overline{\RePart(\lambda)}$ for the head $(\lambda, m)$ of $Q$, we know whether it is safe to emit the $(\lambda, m)$ as next (possible) geodesic or whether we need to take further elements from the pre-length spectrum stream.
}

\subsubsection{Inner loop checks} \label{sec:innerLoopChecks}

Here are more explanations of the checks performed inside the inner loop of Algorithm~\ref{algo:lenSpec}:
\begin{itemize}[itemindent=-0.6cm,leftmargin=0.9cm]
\item[] Step~\ref{algStep:perf}: This is a performance optimization. If $\RePart(\lambda)+\varepsilon<\lambda^{\RePart}_\text{prev}$, another representative of the geodesic $m$ has already been emitted earlier. It has also already been added to $\setOfGeodesics$. Thus, if we did not perform this quick check here, $m$ would be rejected in Step~\ref{algStep:rejectDuplicate} (if not earlier in an earlier step).
\item[] Step~\ref{algStep:fixedPts}: We compute the fixed points in a stable way such as described in Section~\ref{sec:compFixedPoints}. That is, we need to address the case where $m$ is parabolic and the two ideal fixed points coincide. From an interval perspective, we cannot distinguish this case from the case where $m$ is close to a parabolic and the two ideal fixed points are close to each other. In both cases, we need to give tight intervals for the two (possibly coinciding) fixed points. These intervals might overlap and necessarily have to if $m$ is parabolic. This is relevant for the next step.
\item[] Step~\ref{algStep:rejectFarSpine}: Recall that we can reject these $m$ because the pre-length spectrum stream has another representative intersecting the spine $S$.\\
We compute $d(\vec{s}_t, x_0x_1)$ in a stable way such as in \cite[Section~3.4]{goerner:tiling}. That is, we need to address the case when the intervals for $x_0$ and $x_1$ overlap. In this case, the left-endpoint of the interval for $d(\vec{s}_t, x_0x_1)$ needs to be positive (and large). The right-endpoint is necessarily $\infty$. The stability property makes this step reject parabolic $m$ (at least if the precision is high enough). Thus, we avoid later problems in Step~\ref{algStep:rejectDuplicate} as the set $\setOfGeodesics$ can only handle loxodromics.\\
As a performance optimization, we can check $d(p,x_0x_1)>R$ before doing the per-tetrahedron checks. Here, $R=\max_{t'} \big\{d(\vec{s}_t,\vec{s}_{t'})+ r(S_t)\big\}$ is a pre-computed upper bound for the radius of the spine $S\subset P$ about $p=\vec{s}_t$ for some fixed $t$.
\item[] Step~\ref{algStep:rejectAmbi}: This check is necessary since we reject non-primitive $m$ by checking against the primitives in $\setOfGeodesics$ in Step~\ref{algStep:rejectDuplicate}. This requires that $(\lambda', m',c)$ is not overtaken by a conjugate power $(\lambda, m,c)=(k\lambda', gm'^kg^{-1},c)$ with $k\geq 2$ in the priority queue $Q$. That is, we cannot allow the interval bounds to be so bad that $\underline{\RePart(\lambda)}\leq \underline{\RePart(\lambda')}$.
\item[] Step~\ref{algStep:ensureProgress}: Ensures non-decreasing $\underline{\RePart(\lambda_i)}$ (one of the guarantees listed in Section~\ref{sec:putAlgo}).
\end{itemize}

\section{Utilities for matrices} \label{sec:geoUtilities}

As some computations are easier in $\myPSL(2,\C)$, we convert between $\SO(1,3)$ and $\myPSL(2,\C)$ as necessary. Accordingly, we also convert between the hyperboloid model $\H^3$ and the upper halfspace model $\{z+t\jmath : t>0\}$.

When using $\myPSL(2,\C)$, we assume that we are given intervals for a matrix in $\mySL(2,\C)$ representing $m\in\myPSL(2,\C)$ for verified computations.


\subsection{Translation length} \label{sec:computeLen}

For $m\in \myPSL(2,\C)$, $\tr(m)$ is only defined up to sign and the complex translation length is given by
\[
\lambda=\pm 2 \cosh^{-1} \left(\frac{\pm \tr (m)}{2}\right).
\]
If $m$ is loxodromic, we pick $\lambda$ to have $\RePart(\lambda)>0$. Given intervals for $m$, we might not be able to make such a pick for $\lambda$ either because $m$ is not loxodromic or because $m$ is close to a non-loxodromic element. Whether or not we are in such a case, we always want $\overline{\RePart(\lambda)}$ to be an upper bound of the (possibly zero) real translation length. To achieve this, we evaluate $\lambda$ as follows:
\begin{algorithmSteps}
\item Among $+\tr(m)$ and $-\tr(m)$, pick the value $t$ with $\RePart(t)>0$ if possible. Pick any otherwise.\\
This is to avoid an idiosyncrasy of SageMath \cite{SageMath} (version 10.5). Its implementation of $\cosh^{-1}$ returns a large interval containing the values of both branches when crossing the branch cut. This is not consistent with the typical complex interval arithmetic implementation of multi-valued functions such as $\log$. The typical implementation picks a branch when crossing the branch cut to return a tight interval.
\item Let $\lambda'=2\cosh^{-1}(t/2)$. Among $+\lambda'$ and $-\lambda'$, pick the value $\lambda$ with $\RePart(\lambda)>0$ if possible. Otherwise, let $\lambda = +\lambda \cup -\lambda$ where $\cup$ denotes the smallest interval containing both complex intervals.
\end{algorithmSteps}

\subsection{Fixed points} \label{sec:compFixedPoints}

We assume that $m\in\myPSL(2,\C)$ is not an elliptic and, in particular, not the identity. We conjugate $m$ by a unit lower-triangular matrix $t_x$ to avoid a fixed point at $\infty$. More precisely, let
\[
m'=\left(\begin{array}{cc}a & b\\c & d\end{array}\right) = t^{-1}_x m t_x\quad\text{with}\quad t_x=\left(\begin{array}{cc}1 &0 \\ x & 1\end{array}\right)
\]
where we pick the $x$ among $-1,0,1$ with $c$ bounded away from $0$ the furthest.

In the upper half-space model, the two (possibly coinciding) fixed points of $m'$ are now given by $z=\left((a-d) \pm \sqrt{w}\right)/2c$ with $w = (a-d)^2+4bc.$

For verified computations, we evaluate $\sqrt{w}$ as follows: If $w\not=0$ (the interval does not contain zero), we simply use the typical complex interval implementation which picks a branch if $w$ crosses the branch cut. In this case, we have two provably distinct fixed points and $m$ is a loxodromic element. Otherwise, we use we use $[-u,u]+[-u,u] i$ where $u$ is an upper bound for $\sqrt{\RePart(w)^2+\ImPart(w)^2}.$ The latter case occurs when $m$ is parabolic but might also occur when $m$ is loxodromic and the two fixed points are far away from the origin $\jmath$ of the upper half-space model.

To obtain the (possibly coinciding) fixed points of $m$ in the hyperboloid model $\H^3$, we convert each fixed point of $m'$ to the Klein model, lift to a light-like vector in $\posLightCone$ and apply the $\SO(1,3)$-matrix corresponding to $t_x$.

If needed, we can determine which fixed point $x\in\H^3$ is attractive by converting $m$ to an $\SO(1,3)$-matrix and checking for each of the two fixed points whether the action of $m$ increases the time component, that is $(mx)_0>x_0$.

\subsection{Deciding whether two matrices $m, m'\in\Gamma$ are the same} \label{sec:matrixEquality}

We now discuss how to implement the predicate $\Eq_\Gamma(m, m')$ which is true when $m$ and $m'$ are (interval estimates) for the same matrix in $\Gamma\subset\SO(1,3)$. We proceed as in \cite[Section~9.3]{goerner:tiling} and compose the predicate $\Eq_b$ described in \cite[Section~9.2]{goerner:tiling} with the map $\Gamma\to\H^3, g\mapsto gp$. That is $\Eq_\Gamma(m, m')=\Eq_b(mp, m'p)$ with $b=\cosh r$ where $p$ and $r$ are the incenter and inradius, respectively, of a fixed tetrahedron $T_t\subset \H^3$ of the triangulation $\myTrig$.

\subsection{Deciding whether a matrix $m'\in\Gamma^\text{hyp}$ is a power of $m\in\Gamma^\text{hyp}$} \label{sec:matrixPowerEquality}

Let $\Gamma^\text{hyp}\subset\Gamma$ be the subset of hyperbolic elements in $\Gamma\subset\SO(1,3)$. Consider the predicate $\Eq^\N(m', m)$ for $m, m'\in \Gamma^\text{hyp}$ which is true if $m'$ is a positive power of $m$. We implement an interval version of $\Eq^\N(m', m)$ as follows:
\begin{algorithmSteps}
\item Let $\lambda$ and $\lambda'$ be the respective complex translation lengths of $m$ and $m'$ (see Section~\ref{sec:computeLen}).
\item If the interval for $\RePart(\lambda')/\RePart(\lambda)$ contains a unique natural number $k\geq 1$: Return $\Eq_\Gamma(m^kp, m'p)$ (with $\Eq_\Gamma$ from previous section).
\item If the interval for $\RePart(\lambda')/\RePart(\lambda)$ contains no natural number: Return false.
\item (Otherwise) Fail.
\end{algorithmSteps}
We let $\Eq^\Z(m',m)=\Eq^\N(m', m)\vee \Eq^\N(m', m^{-1})$ be the predicate that is true if $m'$ is a non-trivial power of $m$.

\section{Set of geodesics for a complete triangulation} \label{sec:setOfGeos}

Recall that Algorithm~\ref{algo:lenSpec} requires a set $\setOfGeodesics$ of closed geodesics in $M=\Gamma\backslash \H^3$ represented by $m\in\Gamma^\text{hyp}$ where $\Gamma^\text{hyp}\subset\Gamma$ is the subset of loxodromic elements. More precisely, $\setOfGeodesics$ is a set of elements in $\Gamma^\text{hyp}/\sim$ where $m\sim m'$ if $m$ and $m'$ are conjugate in $\Gamma$. A special requirement for $\setOfGeodesics$ is that we can look-up elements by power. That is, given $m'\in\Gamma^\text{hyp}$, determine whether $m$ is conjugate (in $\Gamma$) to a power $m^k$ where $m$ is in $\setOfGeodesics$ and $k\in\Z\setminus\{0\}$.

We use the notation from \cite[Section~9]{goerner:tiling}. We describe how to implement a dictionary $\left(\Gamma^\text{hyp}/\sim\right)\rightharpoonup Y$. The set $\setOfGeodesics$ can then be implemented by letting $Y$ be any one-element set.

\subsection{Dictionary $\Gamma^\text{hyp}\rightharpoonup Y$ with power look-up}

To implement $\Gamma^\text{hyp}\rightharpoonup Y$ we proceed similarly to \cite[Section~9.3]{goerner:tiling} and use the same predicate but a different hash function to enable look-ups by positive power. The hash function $h$ is as follows: Given $m\in\Gamma^\text{hyp}$, let $x$ be its attractive fixed point on the boundary $S^2$ of the Klein model. We set $\pi_\text{fix}(m)=x\cdot \mu$ where $\mu$ is some fixed $3$-vector. The dictionary $\Gamma^\text{hyp}\rightharpoonup Y$ is now implemented as hash dictionary $D_{\Eq_\Gamma, \pi_\text{fix}}$ as defined in \cite[Section~8.2]{goerner:tiling} where $\Eq_\Gamma$ is the predicate from Section~\ref{sec:matrixEquality}. 

To look-up $m'\in\Gamma^\text{hyp}$ by positive power, we first use the hash $\pi_\text{fix}(m')$ to look-up candidates $(I,(m,v))$ in the underlying $D_\R$. We then evaluate the predicate $\Eq^\N(m', m)$ from Section~\ref{sec:matrixPowerEquality} for each candidate and return $v$ on a match.

\subsection{Dictionary $\left(\Gamma^\text{hyp}/\sim\right)\rightharpoonup Y$} \label{sec:graphTraceSet}

We can use $\Gamma^\text{hyp}\rightharpoonup Y$ to implement $\left(\Gamma^\text{hyp}/\sim\right)\rightharpoonup Y$ as quotient dictionary as defined in \cite[Section~8.3]{goerner:tiling} by specifying a (multi-valued) choice function $c\from\Gamma^\text{hyp}\to\Gamma^\text{hyp}$. Given $m\in\Gamma^\text{hyp}$, $c(m)$ is (a superset of) the set of all conjugates of $m$ and $m^{-1}$ that fix a line intersecting the developed fundamental polyhedron $P$.

We compute $c(m)$ as follows. Let $K$ be the line fixed by $m$. Pick a generic point $x$ (also see \cite[Section~5.2]{goerner:tiling} on $K$. Use graph-tracing \cite[Section~6]{goerner:tiling} to find a lifted tetrahedron $m_0T_{t_0}$ 
or a pair of lifted tetrahedra $m_0T_{t_0}$ and $m_1T_{t_1}$ such that at least one contains $x$.

Using these lifted tetrahedra as seeds (see \cite[Sections~5.2]{goerner:tiling} for definition), we run the tiling algorithm in \cite[Algorithm~7.3.1]{goerner:tiling} to tile about $K$. Let $(r_0, m_0T_{t_0}), (r_1,m_1T_{t_1}), \dots$ be the resulting stream. Pick $i$ large enough such that $r_i>0$. The lifted tetrahedra $m_0T_{t_0}, m_1T_{t_1}, \dots, m_{i-1}T_{t_{i-1}}$ cover $K$ in $\Gamma_K\backslash\H^3$ where $\Gamma_K=\{m^j: j\in\Z\}$. In other words, they cover the geodesic $\gamma$ in $M$. Thus, all translates of $K$ intersecting the fundamental polyhedron $P$ (and potentially additional translates) are  given by the lifted objects $m_0^{-1}K,\dots,m_{i-1}^{-1}K$ (also called tetrahedra view in \cite[Section~7.2]{goerner:tiling}). Thus, $c(m)$ is given by
\[
m_0^{-1}m^{\pm 1}m_0,\dots,m_{i-1}^{-1}m^{\pm 1}m_{i-1}
\]
which might include duplicates.

\begin{remark}
This suggests that raytracing inside $P$ is an alternative method to computing $c(m)$. Note, however, that it is hard to resolve the ambiguities in a verified way when $\gamma$ intersects the $1$-skeleton.
\end{remark}

\begin{remark} \label{rem:infGraphTrace}
Assume that $M$ is incomplete and that $K$ corresponds to a core curve $\gamma$. Then, there is no lifted tetrahedron containing $x\in K$. This is because $K$ corresponds to the incompleteness locus (defined later in Section~\ref{sec:generalizationToSpun}). Thus, the heuristic step of graph-tracing (\cite[Algorithm~6.1.2]{goerner:tiling}) either runs into an infinite loop (if using exact arithmetic) or the verification step (\cite[Algorithm~6.1.3]{goerner:tiling}) fails. Hence, we have to guard against core curves when generalizing the length spectrum algorithm to spun-triangulations, see Sections~\ref{sec:introCoreCurvesInfiniteGraphTrace} and \ref{sec:coreCurveDetection}.
\end{remark}

\begin{remark} \label{rem:infTilingWhenInterCore}
Assume that $M$ is incomplete and that $K$ intersects a core curve $\gamma$ without coinciding with a core curve. Then the points of $K$ where it intersects $\gamma$ are not covered by any lifted tetrahedron. In fact, there are infinitely many lifted tetrahedra intersecting the line segment from $x$ to $mx$ which covers $\gamma$ once. Hence, we need to detect and deal with this case when generalizing the length spectrum algorithm to spun-triangulations, see Sections~\ref{sec:introGeoCrossCore} and \ref{sec:coreCurveAvoidance}.
\end{remark}

\section{Generalization to spun-triangulations} \label{sec:generalizationToSpun}

So far, we have only considered geometric triangulations $\myTrig$ forming complete hyperbolic structures. We now extend the result to geometric spun-triangulations. That is, we consider ideal triangulations $\myTrig$ with a hyperbolic structure resulting in a hyperbolic manifold $M$ that might be incomplete but can be completed by attaching circles. We call the completed manifold the \emphasisText{filled manifold} $M_\text{filled}$ and call the circles forming the \emphasisText{incompleteness locus} $M_\text{filled}\setminus M$ the \emphasisText{core curves}. For each core curve $\gamma$, we assume that the input to the length spectrum algorithm also includes a combinatorial description of a peripheral curve in $M$ homotopic to $\gamma$ in $M_\text{filled}$; see Section~\ref{sec:coreFromFilling} for details. Of course, we want to include the geodesics that intersect the incompleteness locus $M_\text{filled}\setminus M$ (including the core curves themselves) when computing the length spectrum.

The generalization to geometric spun-triangulations is both motivated by the desire to compute the length spectrum for closed manifolds as well as the ability to accelerate the computation for some cusped manifolds as shown in Example~\ref{ex:newAlgSpun}.

In this section, we list the problems with the length spectrum algorithm arising from the incompleteness locus. In the subsequent sections, we revisit the previous sections to address these problems. Note that the parts ``$[\dots]$'' in Algorithm~\ref{algo:lenSpec} are now relevant.

\subsection{Tiling about a point limited by incompleteness locus} \label{sec:introSwissCheese}

Recall that Algorithm~\ref{algo:preLen} computes the pre-length spectrum using the tiling algorithm \cite[Algorithm~7.3.1]{goerner:tiling}. As pointed out in \cite[Section~7.2]{goerner:tiling}, its maximal tiling radius is limited by the incompleteness locus.

We fix this by tiling a space $U\subset\H^3$ instead that avoids the incompleteness locus. To make the tiling algorithm work, $U$ needs to have the \emphasisDef{swiss cheese property}. Section~\ref{sec:generalizedCuspCross} defines this property and computes a suitable subspace $U\subset\H^3$ having this property.

Since we now tile $U$ rather than $\H^3$, we need to modify the spine so that it is contained in $U$ in Section~\ref{sec:spineIncomplete}. This gives us a new upper bound $\overline{r}(S_t)$ for the spine radius. Finally, Section~\ref{sec:tilingIncompleteAboutPoint} modifies the tiling algorithm to tile $U$.

\subsection{Core curves missing from the pre-length spectrum}

Consider a core curve $\gamma\subset M_\text{filled}\setminus M$ of complex length $\lambda$. Recall the definition of the pre-length spectrum stream $(\mu_0, m_0), (\mu_1, m_1), \dots$ in Section~\ref{sec:defPrelength}. Note that $\gamma$ does not intersect the spine $S\subset M$. Thus, we cannot assume that $\gamma$ has a representative (among $m_0, \dots, m_{i-1}$ where $i$ is such that $\mu_i\geq \RePart(\lambda)$).

We fix this by computing the core curves ahead of time in Section~\ref{sec:addCoreCurves} and explicitly adding them to length spectrum in Step~\ref{step:addCoreCurves} of Algorithm~\ref{algo:lenSpec}.

\subsection{Core curves cause infinite graph-trace} \label{sec:introCoreCurvesInfiniteGraphTrace}

Neither can we assume that a core curve $\gamma$ does not have a representative in the pre-length spectrum. In fact, since we use the tiling algorithm to compute the pre-length spectrum stream, every $m\in\Gamma\setminus\{\Id\}$ appears in it eventually. This includes the representatives of $\gamma$. We just do not know whether any such representative appears before or after the first $i$ with $\mu_i\geq \RePart(\lambda)$ since $\gamma$ does not intersect the spine.

Note that Step~\ref{algStep:rejectDuplicate} of Algorithm~\ref{algo:lenSpec} looks up $m\in\Gamma^\text{hyp}$ in the set $\setOfGeodesics$. Recall from Remark~\ref{rem:infGraphTrace} that this look-up causes an infinite or failed graph-trace when $m$ fixes a line $K$ corresponding to a core curve. That is, we cannot find a lifted tetrahedron containing a point $x$ on $K$ which is part of the incompleteness locus.

To fix this, we modify the graph-trace algorithm to also stop when $K$ coincides with a lift of a core curve. If this happens, we reject $m$ in Step~\ref{algStep:rejectCoreCurve} of Algorithm~\ref{algo:lenSpec}. We refer the reader to Section~\ref{sec:coreCurveDetection} for details.


\begin{remark}
We could also use that a core curve does not intersect the spine and strengthen Step~\ref{algStep:rejectFarSpine} instead. That is, reject $m$ whenever $x_0x_1\cap S_t=\varnothing$ for every $T_t$. However, this might be hard since $S_t$ can have complicated geometry.
\end{remark}

\subsection{Geodesics intersecting core curves cause infinite tiling} \label{sec:introGeoCrossCore}

Let $m\in\Gamma^\text{hyp}$. Let $K$ be the line fixed by $m$ and $x$ be a point on $K$. Consider the line segment from $x$ to $mx$ covering the corresponding closed geodesic in $M_\text{filled}$ once.

Note that $K$ can intersect one or more core curves without coinciding with any core curve. In that case, the line segment intersects infinitely many lifted tetrahedra. Recall from Remark~\ref{rem:infTilingWhenInterCore} that this is again a problem looking up $m$ in the set $\setOfGeodesics$ in Step~\ref{algStep:rejectDuplicate} of Algorithm~\ref{algo:lenSpec}.

To fix this, we remove disjoint and embedded tubes about all core curves from the line segment. The result is non-empty and only finitely many lifted tetrahedra are required to cover it. We refer the reader to Section~\ref{sec:coreCurveAvoidance} for details.



\section{Generalized cusp cross sections} \label{sec:generalizedCuspCross}

\subsection{Swiss cheese} \label{sec:swissCheese}

As already stated in Section~\ref{sec:introSwissCheese}, our goal is to modify \cite[Algorithm~7.3.1]{goerner:tiling} to tile a subspace $U\subset\H^3$. If $U$ is just any subspace, we might not find all tiles intersecting the ball $B_r(u)$ by just recursively tiling; see Figure~\ref{fig:failedSwissCheese}. To avoid this, we require the following property of $U$:
\begin{definition} \label{def:swissCheese}
Let $U$ be a subset of\/ $\H^3$\!. We say that $U$ has the \emphasisDef{swiss cheese property} if $B_r(x)\cap U$ is connected for every $r>0$ and $x\in U$.
\end{definition}

\begin{figure}[h]
\begin{center}
\begingroup%
  \makeatletter%
  \providecommand\color[2][]{%
    \errmessage{(Inkscape) Color is used for the text in Inkscape, but the package 'color.sty' is not loaded}%
    \renewcommand\color[2][]{}%
  }%
  \providecommand\transparent[1]{%
    \errmessage{(Inkscape) Transparency is used (non-zero) for the text in Inkscape, but the package 'transparent.sty' is not loaded}%
    \renewcommand\transparent[1]{}%
  }%
  \providecommand\rotatebox[2]{#2}%
  \newcommand*\fsize{\dimexpr\f@size pt\relax}%
  \newcommand*\lineheight[1]{\fontsize{\fsize}{#1\fsize}\selectfont}%
  \ifx\svgwidth\undefined%
    \setlength{\unitlength}{150.81241878bp}%
    \ifx\svgscale\undefined%
      \relax%
    \else%
      \setlength{\unitlength}{\unitlength * \real{\svgscale}}%
    \fi%
  \else%
    \setlength{\unitlength}{\svgwidth}%
  \fi%
  \global\let\svgwidth\undefined%
  \global\let\svgscale\undefined%
  \makeatother%
  \begin{picture}(1,1.17142346)%
    \lineheight{1}%
    \setlength\tabcolsep{0pt}%
    \put(0,0){\includegraphics[width=\unitlength,page=1]{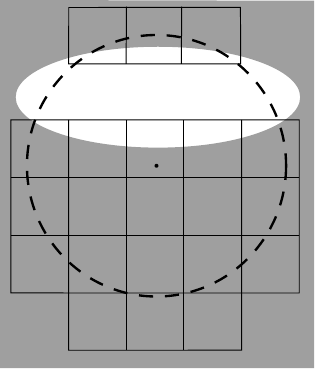}}%
    \put(0.50745373,0.62306931){\makebox(0,0)[lt]{\lineheight{1.25}\smash{\begin{tabular}[t]{l}$x$\end{tabular}}}}%
    \put(0.00854323,1.05707694){\makebox(0,0)[lt]{\lineheight{1.25}\smash{\begin{tabular}[t]{l}\Huge $U$\end{tabular}}}}%
    \put(0.18459985,0.86000144){\makebox(0,0)[lt]{\lineheight{1.25}\smash{\begin{tabular}[t]{l}$B_r(x)$\end{tabular}}}}%
  \end{picture}%
\endgroup%

\end{center}
\caption{A subspace $U$ violating the swiss cheese property causes recursive tiling to miss some tiles intersecting $B_r(x)$.\label{fig:failedSwissCheese}}
\end{figure}
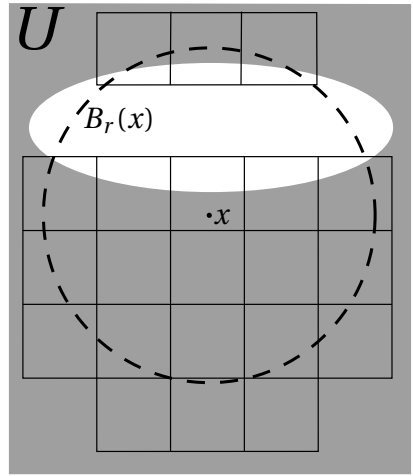

We choose to tile the subspace $U\subset\H^3$ given in the following definition and proposition:
\begin{definition} \label{def:genCuspNbhd}
Let $M$ be an oriented, finite volume hyperbolic 3-manifold that can be completed to $M_\text{filled}$ by attaching circles. For a cusp $i$, we call $C_i$ a \emphasisDef{generalized cusp neighborhood} if
\begin{itemize}
\item the cusp is complete and $C_i$ is a cusp neighborhood of $M$, or
\item the cusp is incomplete and $C_i$ is given by a tube about a geodesic in $\H^3$ that is a lift of to the respective core curve $\gamma\subset M_\text{filled}\subset M$.
\end{itemize}
\end{definition}
We define generalized cusped neighborhoods to be disjoint or embedded analogously to \cite[Section~4.3]{goerner:tiling}. 
\begin{proposition} \label{prop:complementGeneralizedCuspNeighborhoods}
Let $M$ be an oriented, finite volume hyperbolic 3-manifold that can be completed to $M_\text{filled}$ by attaching circles. Let $C_i$ be embedded and disjoint generalized cusp neighborhoods for all cusps $i$. Let $U$ be the pre-image of $M_\text{filled}\setminus \bigcup_i C_i$ in the universal cover $\H^3\to M_\text{filled}$. Then, $U$ has the swiss cheese property.
\end{proposition}
\begin{proof}
Consider one lift $N$ of one $C_i$. Use the upper half-space model. If $C_i$ is a cusp neighborhood, $N$ is a Euclidean ball. If $C_i$ is a tube about a geodesic, apply a transform such that $N$ is a Euclidean cone. $B_r(x)$ is also a Euclidean ball. Removing another ball or cone from it will not split it into several connected components. $B_r(x)$ only intersects finitely many lifts of any $C_i$. All lifts of all $C_i$ are disjoint. Thus removing them from $B_r(x)$ is not splitting $B_r(x)$ into several connected components either.
\end{proof}

\subsection{Enclosing horoneighborhood} \label{sec:encloHoro}

\begin{figure}[h]
\begin{center}
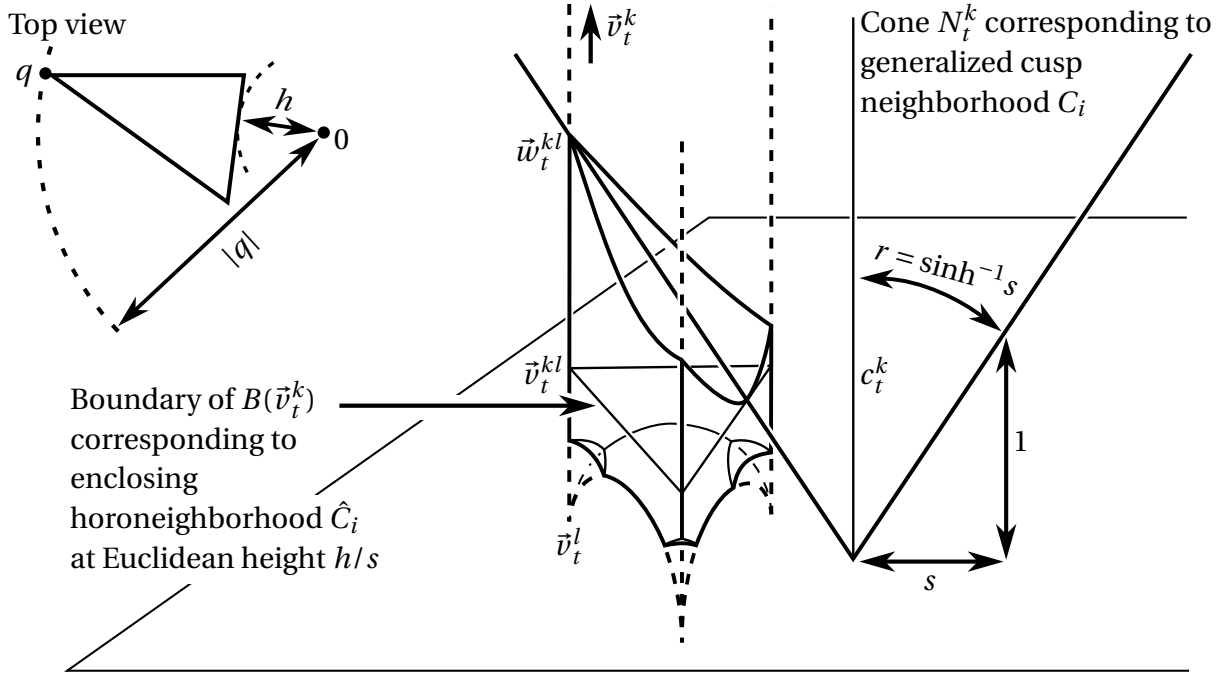
\end{center} 
\caption{Truncated tetrahedron $T_t\cap U$ and enclosing horoneighborhood in upper half-space model.\label{fig:enclosingNbhd}}
\end{figure}

Given an ideal tetrahderon $T\subset\H^3$ and a vertex $\vec{v}$ of $T$, the \emphasisText{extended tetrahedron} is the union of all lines starting at $\vec{v}$ and passing through the opposite face of $T$. A  \emphasisText{horoneighborhood} of $\vec{v}$ is the intersection of a horoball $B(s\vec{v})$ and the extended tetrahedron.

Consider a generalized cusp neighborhood $C_i$ about cusp $i$. For each tetrahedron $T_t$, consider all vertices $\vec{v}$ corresponding to that cusp. For each such vertex $\vec{v}$ take the smallest horoneighborhood containing the intersection of the generalized cusp neighborhood with the tetrahedron; see Figure~\ref{fig:enclosingNbhd}. If the generalized cusp neighborhood is large, we again might need to work in the universal cover $\H^3$ and extend the tetrahedron as above. We call the union of all these horoneighborhoods the \emphasisText{enclosing horoneighborhood $\hat{C}_i$} of $C_i$.

In the complete case, the enclosing horoneighborhood is just the same as the cusp neighborhood. In the incomplete case, it also depends on the triangulation. Analogously to cusp neighborhoods and cusp cross sections in the complete case, an enclosing horoneighborhood is also specified by the edge lengths of the horotriangles. These edge lengths are now per edge of a triangle in the cusp triangulation instead of per edge of the cusp triangulation.

\newpage

We use the enclosing horoneighborhood for simplicity. For example, we can use criterion in \cite[Figure~4.5.3]{goerner:tiling} to ensure that enclosing horoneighborhoods and thus the underlying generalized cusp neighborhoods are embedded or disjoint (if they are in standard form).

\subsection{Generalized cusp cross section}

We generalize \cite[Section~4.4]{goerner:tiling} to include the incomplete case.

We assume that the geometric ideal triangulation $\myTrig$ is oriented. This induces an orientation on the cusp triangulations and thus their triangles and the edges of each triangle. Note that an edge $e$ of a triangle is identified with an edge $e'$ of another triangle in an anti-parallel way.

Pick a cusp. A \emphasisText{generalized cusp cross section} assigns to complex lengths to the edges of each triangle in the cusp triangulation such that:
\begin{itemize}
\item The ratio of complex lengths of two consecutive edges of a triangle needs to be related through the respective negative cross ratio $-z_t$, $-1/(1-z_t)$ or $-(1-1/z_t)$, similar to \cite[Section~4.4]{goerner:tiling}.
\item There is a spanning tree in the dual 1-skeleton of the cusp triangulation such that: if the edges $e$ and $e'$ of two triangles are identified and cross an edge in the spanning tree, their complex lengths add to zero.
\end{itemize}
If the cusp is incomplete, the generalized cusp cross section also assigns a position in $\C^*$ to each vertes of each triangle in the cusp triangulation such that:
\begin{itemize}
\item The complex length of an edge of a triangle is the difference between the positions assigned to its endpoints in the respective order.
\item If $e$ and $e'$ of two triangles are identified, the positions assigned to the endpoints of $e$ and $e'$ are related by a similarity of $\C$ fixing $0$.
\item If, furthermore, $e$ and $e'$ cross an edge in the spanning tree, the positions are identical.
\end{itemize}

To obtain a generalized cusp cross section, pick a spanning tree and assign some complex length to one edge of one triangle. All other edge lengths are now recursively determined. In the incomplete case, also assign some position to some vertex, compute all other vertex positions recursively, then determine the fixed point of a non-trivial similarity to move all vertex positions such that the fixed point ends up at $0$.

Together with $\infty$, the vertices of a generalized cusp cross section span tetrahedra spinning about the line from $0$ to $\infty$ (corresponding to the core curve) under the appropriate identification of the universal cover $\tilde{M}_\text{filled}$ with the upper half space model $\left\{z+t\jmath\colon t>0\right\}$; again, see Figure~\ref{fig:enclosingNbhd}.

\subsection{Computing some enclosing horoneighboorhoods} \label{sec:computeHorotriangles}

Our next goal is to compute an enclosing horoneighborhood $\hat{C}_i$ from a generalized cusp cross section. More precisely, we compute the edge lengths of the horotriangles in $\hat{C}_i$. They then uniquely determine the horoneighborhood about each vertex of each tetrahedron corresponding to cusp $i$.

Consider the edges of one such horotriangle. Their lengths are given by the absolute values of the corresponding edges in the generalized cusp cross section divided by $h$. Here, $h$ is given as follows:
\begin{itemize}
\item For a complete cusp: $h=1$. Note that the corresponding cusp neighborhood $C_i$ depends on the choice of generalized cusp neighborhoods we made earlier. This is ok as we normalize and rescale $C_i=\hat{C}_i$ later anyway.
\item For an incomplete cusp: $h$ is the Euclidean distance to $0$ of the triangle in $\C$ spanned by the positions of the corresponding vertices in the generalized cusp cross section. $h$ is also the Euclidean height of a horotriangle touching the Euclidean cone with angle $\pi/4$ and tip at $0$ in the upper half-space model; again, see  Figure~\ref{fig:enclosingNbhd}. Thus, the resulting horoneighborhood encloses a tube about the corresponding core curve of radius $r=\sinh^{-1}(1)$. This is independent of the choice of generalized cusp cross section.
\end{itemize}

\subsection{Scaling enclosing horocusp neighborhoods} \label{sec:scalingNeighborhoodsIncomplete}

Note that scaling an enclosing horoneighborhood $\hat{C}_i$ by scaling all horotriangle edge lengths by the same factor $s$ corresponds to changing the size of the underlying generalized cusp neighborhood $C_i$. For an incomplete cusp and starting with the above $\hat{C}_i$, scaling by $s$ results in a tube $C_i$ of radius $r=\sinh^{-1}(s)$. For a complete cusp, scaling by $s$ scales the area of the cusp neighborhood $C_i$ by $s^2$.

The goal of this section is to scale enclosing horoneighborhoods such that they
\begin{itemize}
\item intersect the tetrahedra in standard form; see \cite[Definition~4.5.1]{goerner:tiling},
\item are embedded and disjoint and
\item avoid the incenter of each tetrahedron.
\end{itemize}

\cite[Lemma~4.5.2 and Figure~4.5.3]{goerner:tiling} give a way to check the first two conditions. For the third condition, we similarly use:
\begin{lemma}
Let $T$ be an ideal geodesic tetrahedron with cross ratio $z$. Pick a vertex $v$ of $T$. Let $H$ be any horoball about $v$ not containing the incenter of $T$. The supremal area of the intersection of $\partial H$ with $T$ across all such horoballs is
\[
A(z)=A\left(\frac{1}{1-z}\right)=A\left(1-\frac{1}{z}\right)=\frac{1}{4}\cdot\frac{1 + |z| + |1-z|}{1\cdot |z| \cdot |1-z |}\cdot\ImPart(z).
\]
\end{lemma}
\begin{proof}
We use the upper half-space model such that the vertices are at $0$, $1$, $z$ and $\infty\in\C\cup\{\infty\}$. Let $a=1$, $b=|z|$ and $c=|1-z|$ be the lengths of the Euclidean triangle formed by the vertices in $\C$. It has area $A_\text{Eucl}=\ImPart(z)/2$, inradius $r=2A_\text{Eucl}/(a+b+c)$ and circumradius $R=abc/4A_\text{Eucl}$. Euler's triangle formula gives the distance $d=\sqrt{R(R-2r)}$ between incenter and circumcenter. The insphere of $T$ projects to the disk bound by the incircle of the Euclidean triangle. Similarly, the face of $T$ opposite $\infty$ projects to the Euclidean triangle and the plane supporting the face projects to the disk bound by the circumcirlce of the Euclidean triangle. Regarding insphere of $T$ as a Euclidean sphere, its center, lowest and highest points have heights $h_\text{Eucl}=\sqrt{(R+r)^2-d^2}$, $h_\text{Eucl}-r$ and $h_\text{Eucl}+r$. The geometric mean of the latter two gives the Euclidean height $h$ of $T$'s incenter. We evaluate $A(z)=A_\text{Eucl}/h^2$. 
\end{proof}

Disjoint and embedded generalized cusp neighborhoods are determined by the area of their enclosing horoneighborhoods. Here, the area of an enclosing horoneighborhood is simply the sum of the Euclidean areas of all the horotriangles. The area scales quadratically when scaling the edge lengths. Thus, each of the above conditions can be expressed as an inequality $A(\hat{C}_i)A(\hat{C}_j)\leq A_{ij}$ in the areas $A(\hat{C}_i)$ and $A(\hat{C}_j)$ of two enclosing horoneighborhoods $\hat{C}_i$ and $\hat{C}_j$ about cusps $i$ and $j$ where $i$ and $j$ can be equal. Thus, these conditions give rise to a matrix $(A_{ij})$ akin to the maximal cusp area matrix in \cite[Definition~1.1.2]{goerner:tiling}. We apply \cite[Algorithm~11.1.1]{goerner:tiling} to this matrix to obtain areas corresponding to maximal disjoint and embedded generalized cusp neighborhoods in an unbiased way.

\begin{remark} \label{rem:perfDepNeigh}
The choice of generalized cusp neighborhoods affects the radius of the spine and truncated tetrahedron computed in Sections~\ref{sec:spunSpine} and~\ref{sec:radiusTrunc} and, thus, the performance of the length spectrum algorithm when there are multiple cusps. The above choice is canonical and easy to compute but not necessarily the one giving the best performance. We do not expect that optimizing this choice for performance is worth the effort except in some extreme examples. We expect that a judicious choice of center for the spine will have similarly minor effects on performance. 
\end{remark}

\section{Computing a spine radius for spun-triangulations} \label{sec:spineIncomplete}

From now on, let $M$ be an oriented, finite volume hyperbolic 3-manifold that can be completed to $M_\text{filled}$ by attaching circles. Let $T_t\subset\H^3$ be ideal geodesic tetrahedra such that $P=\cup T_t$ is a developed fundamental polyhedron for $M$. Let $C_i$ be generalized cusp neighborhoods such that the corresponding enclosing horocusp neighborhoods $\hat{C}_i$ are embedded and disjoint, intersect each tetrahedron $T_t$ of $\myTrig$ in standard form and avoid the incenter $\vec{s}_t$ of each tetrahedron $T_t$. Section~\ref{sec:scalingNeighborhoodsIncomplete} describes how to find such $C_i$ and $\hat{C}_i$. Let $U$ be the pre-image of $M_\text{filled}\setminus \bigcup_i C_i$ in the universal cover $\H^3\to \Gamma\backslash \H^3\cong M_\text{filled}$ as in Proposition~\ref{prop:complementGeneralizedCuspNeighborhoods}.

\subsection{Tetrahedral vertex vectors} \label{sec:perTetCoreCurve}

Similar to Section~\ref{sec:spineVertexVectors}, we lift the ideal vertices of each tetrahedron $T_t$ to light-like vectors $\vec{v}_t^v\in\posLightCone$. This time, however, we use the enclosing horoneighborhoods $\hat{C}_i$ rather than the cusp neighborhoods $C_i$.

We denote by $N^v_t$ the lift of the generalized cusp neighborhood corresponding to vertex $v$ of $T_t$; see Figure~\ref{fig:enclosingNbhd}. If the corresponding cusp is complete, $N^v_t$ is equal to the horoball $B\left(\vec{v}^v_t\right)$. Otherwise, $N^v_t$ is a tube about a geodesic and $N^v_t\cap T_t$ is contained in its enclosing cusp neighborhood $B\left(\vec{v}^v_t\right)$. We denote this geodesic by $c^v_t$. It corresponds to a core curve in $M_\text{filled}$ and one of its endpoints coincides with $\vec{v}^v_t$. For a complete cusp, we set $c^v_t=\varnothing$.

\subsection{Spine radius} \label{sec:spunSpine}

We proceed similarly to Section~\ref{sec:spinePoints} but need to change the term for $\vec{s}^{kl}_t\in\H^3$ in 
Algorithm~\ref{algo:spine} and might need to include additional points when computing an upper bound $\overline{r}(S_t)$ for the spine radius as maximum of several distances. Here, we again compute the radius about the incenter $\vec{s}_t$ of the tetrahedron $T_t$.

Let $\vec{w}^{kl}_t=\vec{v}^k_t\vec{v}^l_t\cap \partial N^k_t\in\H^3$ be the intersection of edge $kl$ of $T_t$ with the boundary of the respective lift $N^k_t$ of the generalized cusp neighborhood corresponding to vertex $k$ of $T_t$; see Figure~\ref{fig:enclosingNbhd}. We set $\vec{s}^{kl}_t=\left(\vec{w}^{kl}_t+\vec{w}^{lk}_t\right)\postNorm$ to be the midpoint of the edge when truncated by the generalized cusp neighborhoods. Since the enclosing horocusp neighborhood is equal to the cusp neighborhood if the corresponding cusp is complete, this is a generalization of the expression for $\vec{s}^{kl}_t$ in Algorithm~\ref{algo:spine}. Recall that, in general, the enclosing horocusp neighborhood for $N^k_t$ is the horoball $B\left(\vec{v}^k_t\right)\subset\H^3$ and let $\vec{v}^{kl}_t=\vec{v}^k_t\vec{v}^l_t\cap \partial B\left(\vec{v}^k_t\right)\in\H^3$ be the corresponding intersection. Note that $\vec{v}^{kl}_t$ and $\vec{w}^{kl}_t$ coincide if a cusp is complete.

We now compute an upper bound $\overline{r}(S_t)$ for the spine radius as:
$$\overline{r}(S_t)=\max\left\{
\begin{array}{cccc}
d\left(\vec{s}_t, \vec{s}^{kl}_t\right) &\text{for} &0\leq k < l \leq 3\\
d\left(\vec{s}_t, \vec{v}^{kl}_t\right) &\text{for} &0\leq k, l \leq 3, k\ne l & \text{if not (verified that) $\vec{s}^{kl}_t\in T_t\setminus  B\left(\vec{v}^k_t\right)$}.
\end{array}\right\}$$

To compute $\vec{w}^{kl}_t$, go back to Section~\ref{sec:computeHorotriangles} and consider the corresponding embedded cusp triangle. Let $q$ be the position of its vertex corresponding to the edge $kl$ of $T_t$ and $h$ be its distance from $0$, see Figure~\ref{fig:enclosingNbhd}. Let $\mu^{kl}_t=|q|/h$. We have $\vec{w}^{kl}_t=\vec{v}^k_t\vec{v}^l_t\cap \partial B\left(\mu^{kl}_t\vec{v}^k_t\right)\in\H^3$.

To determine whether we can skip $\vec{v}^{kl}_t$ in the above expression, we use
\[
\vec{s}^{kl}_t\in T_t\setminus  B\left(\vec{v}^k_t\right)\quad \Leftrightarrow\quad \log\left(\mu^{kl}_t\right)-\log\left(\mu^{lk}_t\right)< d\left(B\left(\vec{v}^k_t\right), B\left(\vec{v}^l_t\right)\right)
\] which is automatic if the cusp corresponding to vertex $v$ of $T_t$ is complete. We can use \cite[Figure~4.5.3]{goerner:tiling} to compute $d\left(B\left(\vec{v}^k_t\right), B\left(\vec{v}^l_t\right)\right)$ from the real edge lengths in the generalized cusp cross section. To see this equivalence, note that $\vec{s}^{kl}_t\in T_t\setminus  B\left(\vec{v}^k_t\right)$ is equivalent to $\vec{s}^{kl}_t$ having distance less than $d\left(B(\vec{v}^k_t), B(\vec{v}^l_t)\right)/2$ to the midpoint of $\vec{v}^{kl}_t$ and $\vec{v}^{kl}_t$. Here, we use the signed distance that is positive if $\vec{s}^{kl}_t$ is closer to $\vec{v}^k_t$ than the midpoint. The factor $\mu^{kl}_t$ moves $\vec{w}^{kl}_t$ closer to $\vec{v}^k_t$ by distance $\log\left(\mu^{kl}_t\right)$ and thus moves $\vec{s}^{kl}_t$ by distance $\log\left(\mu^{kl}_t\right)/2$. The opposite applies to the factor $\mu^{lk}_t$.




The following proposition can be used with the generalized cusp neighborhoods computed in Section~\ref{sec:scalingNeighborhoodsIncomplete}.

\begin{proposition} \label{prop:spunSpine}
Let $M$, $T_t$, $U$ and $\overline{r}(S_t)$ as above. For each $\varepsilon>0$, there are subsets $S_t\subset T_t$ for all tetrahedra $T_t$ such that 
\begin{itemize}
\item $S=\cup S_t$ forms a spine in $M$,
\item $S_t$ is contained in an $\varepsilon$-neighborhood of $U$ where $U\subset\H^3$ is the complement of the generalized cusp neighborhoods used to construct $S_t$,
\item the radius of each $S_t$ about $\vec{s}_t$ is less than or equal to $\overline{r}(S_t)$.
\end{itemize}
\end{proposition}

The spine we construct in the proof of the above proposition no longer consists of totally geodesic triangles as it did in Proposition~\ref{prop:spine}. This gives us the flexibility to isotope the spine into an $\varepsilon$-neighborhood of $U$. We use the following definition in the proof of the above proposition.

\begin{definition}
A \emphasisDef{horocyclic line segment} between $a$ and $b\in\H^n$ is a segment between $a$ and $b$ in a horocycle containing $a$ and $b$. A set $\Omega\subset\H^n$ is \emphasisDef{horoconvex} if for any two distinct $a,b\in \Omega$, all horocyclic line segments between $a, b$ are contained in $\Omega$. The \emphasisDef{horoconvex closure} of a set $S\subset\H^n$ is the smallest horoconvex $\Omega\subset\H^n$ with $S\subset\Omega$.
\end{definition}

\begin{example}
A ball and a horoball are horoconvex. The horoconvex closure of two points is banana shaped. Every horoconvex set is also convex. Note that the horoconvex closure of $S$ is also the intersection of all balls containing $S$. The horoconvex closure of $S$ contains the convex closure of $S$.
\end{example}

\begin{proof}[Proof of Proposition~\ref{prop:spunSpine}]
We start with the spine $S=\cup S_t$ for $M$ constructed in the proof of Proposition~\ref{prop:spine}; see Figure~\ref{fig:deformedSpine}. Note that the radius of each $S_t$ is equal to $\max_{0\leq k<l\leq 3}\left\{d\left(\vec{s}_t,\vec{s}_t^{kl}\right)\right\}\leq\overline{r}(S_t)$. However, $S_t$ might not be contained in $U$.\\
\begin{figure}[h]
\begin{center}
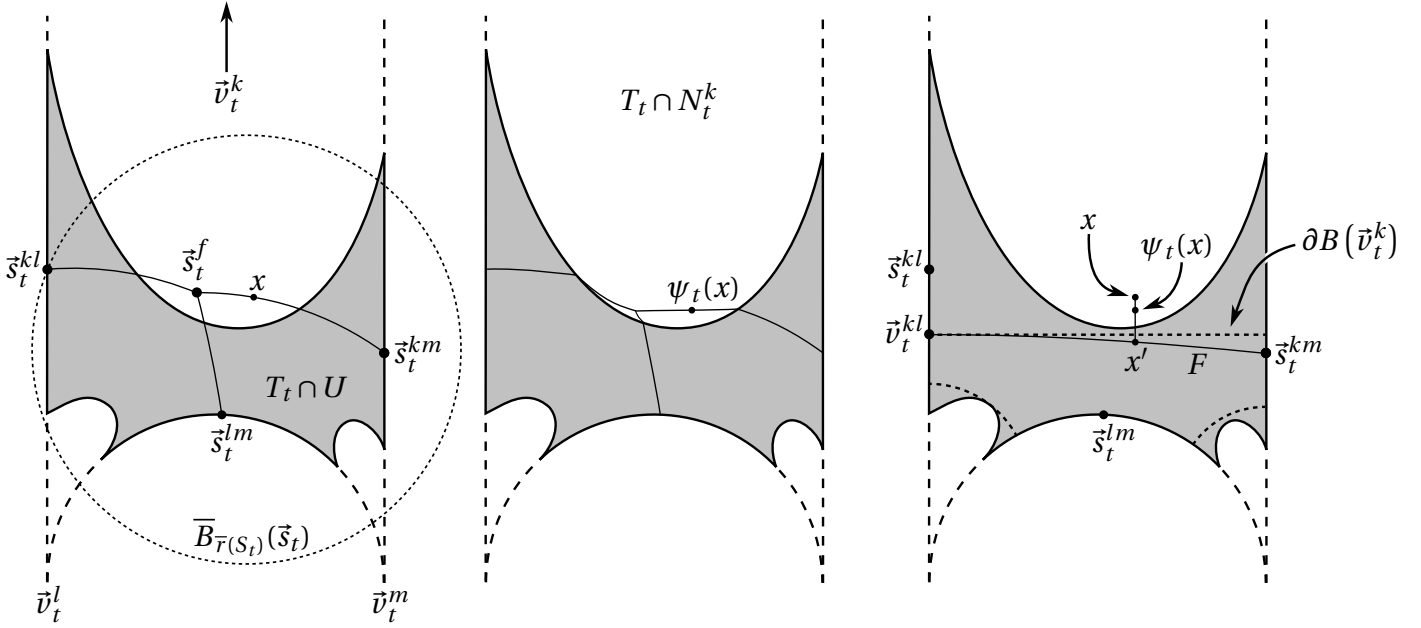
\end{center}
\caption{From left to right: spine $S_t\subset T_T$ (restricted to face $T^f_t$) contained in $\overline{B}_{\overline{r}(S_t)}\left(\vec{s}_t\right)$ but not contained in an $\varepsilon$-neighborhood of $U$, a point $x$ on it; deformed spine $\psi_t(S_t)\subset T_t$ in an $\varepsilon$-neighborhood of $U$; points proving that $\psi_t(x)$ is in the horoconvex hull of $x$, $\vec{b}^{kl}_t=\vec{v}^{kl}_t$ and $\vec{b}^{km}_t=\vec{s}^{km}_t$ and thus $\psi_t(x)\in \overline{B}_{\overline{r}(S_t)}\left(\vec{s}_t\right)$. \label{fig:deformedSpine}}
\end{figure}
It suffices to apply an embedding $\psi_t\from T_t\to T_t$ to each $S_t\subset T_t$ so that the image $\psi_t(S_t)$ is contained in an $\varepsilon$-neighborhood of $U$ and in $\overline{B}_{\overline{r}(S_t)}\left(\vec{s}_t\right)$. The $\psi_t$ will be compatible with the face-pairing matrices so that $\cup \psi_t(S_t)$ forms another spine in $M$ which we will call a deformed spine; again see Figure~\ref{fig:deformedSpine}.\\
Without loss of generality, assume $\varepsilon < \log(1+1/e)$. To construct $\psi_t\from T_t\to T_t$, we split $T_t$ into the truncated tetrahedron $T_t\cap U$ and $T_t\cap N^k_t=T_t\cap \left(N^k_t\setminus c^k_t\right)$ with $k=0,\dots, 3$. On $T_t\cap U$, we set $\psi_t$ to the identity. To describe $\psi_t$ on $T_t\cap N^k_t$, recall that $N^k_t\setminus c^k_t$ is a horoball or a tube about a geodesic $c^k_t$ without $c^k_t$ (depending on whether the corresponding cusp is complete).  We use the upper-halfspace model $\left\{z+t \jmath : t >0\right\}$. We can then identify $N^k_t\setminus c^k_t$ with $N=\left\{z+t \jmath : t>1\right\}$ or $N=\left\{z + t \jmath : 0 < |z| < st\right\}$ for some $s>0$ such that the vertex $\vec{v}^k_t$ of $T_t$ is at $\infty$. Let
\[
\pi_N\from N\to N, 
z+t \jmath \mapsto z+\theta(t)\jmath \quad\text{or}\quad z+t\jmath \mapsto z + \theta\left(\frac{st}{|z|}\right) \frac{|z|}{s}\jmath\text{, respectively,}
\]
where $\theta\from [1,\infty)\to [1,e^\varepsilon), t\mapsto 1 + \big(e^\varepsilon -1\big) \big(1-e^{1-t}\big)$. On $T_t\cap N^k_t$, we set $\psi_t$ to $\pi_N$ under the identification of $N^k_t\setminus c^k_t$ with $N$. This gives a well defined embeddings $\psi_t\from T_t\to T_t$ which extend to an embedding $M\to M$. To see this, note that $\pi_N$ is an embedding moving points downwards (since $\varepsilon < \log(1+1/e)$) such that the image of a point  $x$ is determined by $x$ itself and the downward projection of $x$ onto $\partial \overline{N}$. Thus, $T_t$ is mapped into $T_t$ under the identification. Furthermore, $\pi_N$ can be extended to $\partial\overline{N}$ where it is the identity. Thus, we can set $\psi_t$ to be $\pi_N$ on $T_t\cap N^k_t$.
The range of $\pi_N$ is contained in the $\varepsilon$-neighborhood of $\partial \overline{N}$. Thus, the range of $\psi_t$ and, in particular, $\psi_t(S_t)\subset T_t$ is contained in the $\varepsilon$-neighborhood of $U$.

It is left to show that that $\psi_t(S_t)\subset \overline{B}_{\overline{r}(S_t)}\left(\vec{s}_t\right)$. Note that the incenter $\vec{s}_t$ of $T_t$ is in $U$ due to the choices in Section~\ref{sec:scalingNeighborhoodsIncomplete}. Let $x\in S_t$. Thus, by construction of $S_t$, we have $x\in \overline{B}_{\overline{r}(S_t)}\left(\vec{s}_t\right)$. If $x\in U$, then $\psi_t(x)=x$ and, thus, $\psi_t(x)\in \overline{B}_{\overline{r}(S_t)}\left(\vec{s}_t\right)$. Otherwise, let $k$ be the vertex of $T_t$ such that $x\in T_t\cap N^k_t$. We again work in the upper-halfspace model such that the vertex is at $\infty$. Consider $\partial B\left(\vec{v}^k_t\right)$, the horosphere that is a Euclidean horizontal plane and that bounds $T_t\cap N^k_t$. For each $l\ne k$, there is a point $\vec{b}^{kl}_t\in\overline{B}_{\overline{r}(S_t)}\left(\vec{s}_t\right)$ on the edge $kl$ of $T_t$ below or on $\partial B\left(\vec{v}^k_t\right)$: $\vec{s}^{kl}_t$ if $\vec{s}^{kl}_t$ is below $\partial B\left(\vec{v}^k_t\right)$ or $\vec{v}^{kl}_t$ otherwise. Consider the horoconvex closure of $\vec{b}^{kl}_t$ with $l\ne k$. It contains a topological triangle $F$ below $\partial B\left(\vec{v}^k_t\right)$ having the same vertical projection as $T_t$. Let $x'$ be the vertical projection of $x$ onto $F$. $\psi_t(x)$ is on the vertical line segment between $x$ and $x'$. Thus, $\psi_t(x)$ is in the horoconvex closure of $x\in \overline{B}_{\overline{r}(S_t)}\left(\vec{s}_t\right)$ and $\vec{b}^{kl}_t\in \overline{B}_{\overline{r}(S_t)}\left(\vec{s}_t\right)$ with $l\ne k$. Hence, $\psi_t(x)\in \overline{B}_{\overline{r}(S_t)}\left(\vec{s}_t\right)$.
\end{proof}

\begin{remark}
In the proof of Proposition~\ref{prop:spunSpine}, note that $\psi_t$ is not an embedding when $\varepsilon=0$ and $\psi_t(S_t)$ might no longer be an embedded spine. If the generalized cusp neighborhoods do not touch $\vec{s}_t$, we can use slightly enlarged generalized cusp neighborhoods to construct $\psi_t$ to push the spine into $T_t\cap U$. However, this is not possible if, for example, we have a cusp neighborhood of a complete cusp touching $\vec{s}_t$.
\end{remark}

\begin{proposition} \label{prop:spunSpine2}
Let $M$, $T_t$ and $\overline{r}(S_t)$ as above and $\vec{s}_t$ be the incenter of $T_t$. Let $\gamma$ be a (non-core curve) primitive, unoriented closed geodesic in $M$ of real length $\leq \mu$. Then there is a tetrahedron $T_t$ and an $m\in \Gamma$ representing $\gamma$ with
$$d(\vec{s}_t, m(T_t \cap U)) \leq \cosh^{-1}\left(\cosh\left(\overline{r}(S_t)\right)\cosh\left(\mu\right)\right).$$
\end{proposition}

\begin{proof}
By Proposition~\ref{prop:spunSpine}, for any $\varepsilon>0$, there is a point $x\in\pi^{-1}(\gamma)\subset\H^3$ and a tetrahedron $T_t$ such that  $x$ is in the $\varepsilon$-neighborhood of $T_t\cap U$ and $x\in S_t\subset T_t$, thus $x\in \overline{B}_{\overline{r}(S_t)}\left(\vec{s}_t\right)$. Pick such an $x$ for each $\varepsilon$ in a sequence converging to $0$. Since each of the finitely many $T_t\cap \pi^{-1}(\gamma)$ is compact, there is a tetrahedron $T_t$ and a subsequence of $x$'s converging to a point $s\in T_t\cap \pi^{-1}(\gamma)$. We also have $s\in U$ and $x\in \overline{B}_{\overline{r}(S_t)}\left(\vec{s}_t\right)$. Let $m\in\Gamma$ be a representative of $\gamma$ fixing a line containing $s$. We can use the same argument as in the proof of Proposition~\ref{prop:spineAndTilingRadius}. 
\end{proof}

\section{Avoiding core curves when tiling about a point} \label{sec:tilingIncompleteAboutPoint}

This section addresses the problem that any bounded neighborhood of a point on a lifted core curve intersects infinitely many lifted tetrahedra $mT_t$. We can reduce this to a finite problem by considering instead all lifted truncated tetrahedra $mT_t\cap U$ intersection such a neighborhood where $U$ is as computed in Section~\ref{sec:generalizedCuspCross}. That is, we modify the tiling algorithm to use truncated tetrahedra $mT_t\cap U$ to tile $U$. We are allowed to modify the tiling algorithm this way since $U$ has the swiss cheese property.

Note that we need to compute an upper bound for the distance of a truncated lifted tetrahedron $mT_t\cap U$ to the point $x$ we tile about. We do this by computing the radius of a truncated tetrahedron $T_t\cap U$. Note that this radius is used for tiling and is distinct from the spine radius.







\subsection{Radius of a truncated tetrahedron} \label{sec:radiusTrunc}

\begin{proposition} \label{prop:radiusTruncTet}
Let $T_t$, $C_i$ and $U$ be as in Section~\ref{sec:spineIncomplete}. Let $\vec{s}_t$ be the incenter of $T_t$ and $\vec{w}^{kl}_t$ as in Section~\ref{sec:spunSpine}. When truncated by the generalized neighborhoods $C_i$, the radius of a tetrahedron $T_t$ about $\vec{s}_t$ is given by
\begin{equation*}
r(T_t\cap U)=\max_{k,l}\left\{d\left(\vec{s}_t, \vec{w}^{kl}_t\right)\right\}.
\end{equation*}
\end{proposition}

\begin{proof}
Since the truncated tetrahedron $T_t\cap U$ is compact, the distance function $F\from T_t\cap U\to\R^{\geq 0}, x\mapsto d(\vec{s}_t,x)$ has a maximum. It suffices to show that this maximum is realized by a vertex $\vec{w}^{kl}_t$ of $T_t\cap U$. Or, equivalently, that no other point in $T_t\cap U$ can be a maximum of $F$.\\
The interior of $T_t\cap U$ cannot contain a local maximum of $F$. Topologically, the boundary $\partial (T_t\cap U)$ splits as $T_t\cap \partial U$ consisting of four triangles $\Delta_v$ (corresponding to the vertices $v=0,\dots,3$ of $T_t$) and $\partial T_t\cap U$ consisting of four hexagons $\hexagon_f=T^f_t\cap U$ (corresponding to the faces $f=0,\dots,3$ of $T_t$). Each $\Delta_v$ is supported by a horosphere or cylinder (depending on whether the vertex $v$ of $T_t$ corresponds to a complete or incomplete cusp) and each $\hexagon_f$ is supported by a hyperbolic plane. Therefore, the interior of any $\Delta_v$ or $\hexagon_f$ cannot contain a local maximum of $F$. Furthermore, an edge shared by two hexagons is supported by a hyperbolic line. Therefore, the interior of such an edge cannot contain a local maximum of $F$, either. Next, consider an edge shared by a $\Delta_v$ and a $\hexagon_f$. If vertex $v$ of $T_t$ corresponds to a complete cusp, the edge is supported by a horocycle. 
\begin{figure}[ht]
\begin{center}
\begingroup%
  \makeatletter%
  \providecommand\color[2][]{%
    \errmessage{(Inkscape) Color is used for the text in Inkscape, but the package 'color.sty' is not loaded}%
    \renewcommand\color[2][]{}%
  }%
  \providecommand\transparent[1]{%
    \errmessage{(Inkscape) Transparency is used (non-zero) for the text in Inkscape, but the package 'transparent.sty' is not loaded}%
    \renewcommand\transparent[1]{}%
  }%
  \providecommand\rotatebox[2]{#2}%
  \newcommand*\fsize{\dimexpr\f@size pt\relax}%
  \newcommand*\lineheight[1]{\fontsize{\fsize}{#1\fsize}\selectfont}%
  \ifx\svgwidth\undefined%
    \setlength{\unitlength}{258.4624307bp}%
    \ifx\svgscale\undefined%
      \relax%
    \else%
      \setlength{\unitlength}{\unitlength * \real{\svgscale}}%
    \fi%
  \else%
    \setlength{\unitlength}{\svgwidth}%
  \fi%
  \global\let\svgwidth\undefined%
  \global\let\svgscale\undefined%
  \makeatother%
  \begin{picture}(1,0.67771947)%
    \lineheight{1}%
    \setlength\tabcolsep{0pt}%
    \put(0,0){\includegraphics[width=\unitlength,page=1]{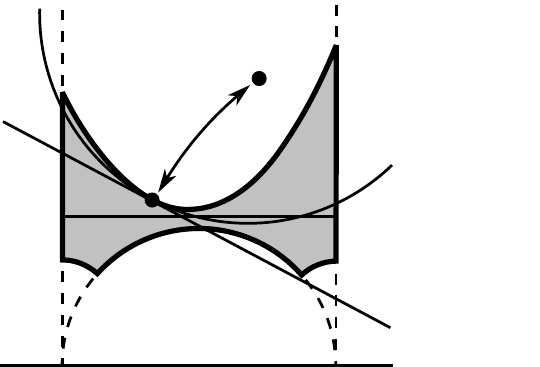}}%
    \put(0.4405486,0.55729174){\makebox(0,0)[lt]{\lineheight{1.25}\smash{\begin{tabular}[t]{l}$p$\end{tabular}}}}%
    \put(0.25786156,0.33361724){\makebox(0,0)[lt]{\lineheight{1.25}\smash{\begin{tabular}[t]{l}$b$\end{tabular}}}}%
    \put(0.02058883,0.24934768){\makebox(0,0)[lt]{\lineheight{1.25}\smash{\begin{tabular}[t]{l}$\vec{v}^{kl}_t$\end{tabular}}}}%
    \put(0.63470834,0.24934768){\makebox(0,0)[lt]{\lineheight{1.25}\smash{\begin{tabular}[t]{l}$\vec{v}^{km}_t$\end{tabular}}}}%
    \put(0.43496942,0.38384409){\makebox(0,0)[lt]{\lineheight{1.25}\smash{\begin{tabular}[t]{l}$H$\end{tabular}}}}%
    \put(0.73505972,0.37990499){\makebox(0,0)[lt]{\lineheight{1.25}\smash{\begin{tabular}[t]{l}$C$\end{tabular}}}}%
    \put(0.72892634,0.05438424){\makebox(0,0)[lt]{\lineheight{1.25}\smash{\begin{tabular}[t]{l}$T$\end{tabular}}}}%
  \end{picture}%
\endgroup%

\end{center}
\caption{Using the upper-halfplane model, a shaded hexagon $\hexagon_f$ in a face of a truncated tetrahedron $T_t\cap U$.\label{fig:tetRad}}
\end{figure}
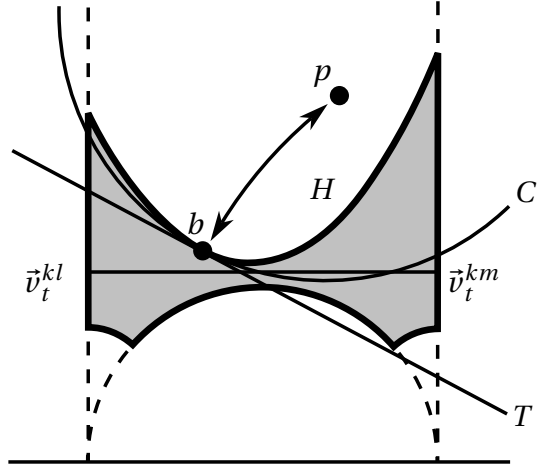
Therefore, the interior of such an edge cannot contain a local maximum of $F$. Otherwise, assume that the maximum of $F$ is realized by an interior point of the edge, say $b$. We identify the plane supporting $\hexagon_f$ with the upper-halfplane model in such a way that vertex $k\ne f$ of $T_t$ is at $\infty$; see Figure~\ref{fig:tetRad} (also see in Figure~\ref{fig:enclosingNbhd}). Let $p$ be the point in this plane closest to $\vec{s}_t$. Consider the two points $\vec{v}_t^{kl}$ and $\vec{v}_t^{km}$ where the boundary $\partial B\left(\vec{v}_t^k\right)$ of the corresponding enclosing horoneighborhood intersects the straight edges of $\hexagon_f$. $\vec{v}_t^{kl}$ and $\vec{v}_t^{km}$ have the same Euclidean height. Note that the top edge of $\hexagon_f$ is a Euclidean hyperbola $H$. The points having the same distance to $p$ as $b$ form a Euclidean circle $C$. $H$ and $C$ have the same Euclidean tangent line $T$ at $b$. We distinguish the cases where $T$ is above or below $C$. In the former case, $b$ is not a maximum of $F$. In the latter case, at least one of $\vec{v}_t^{kl}$ and $\vec{v}_t^{km}$ is below $T$ and thus outside $C$ and further away from $p$ as $b$. Thus, $b$ is not a maximum of $F$.
\end{proof}

\subsection{Modified tiling algorithm} \label{sec:modTiling}

Let $T_t$ and $U$ be as in Section~\ref{sec:spineIncomplete}. Let us fix a tetrahedron a $T_t$ and use its incenter $\vec{s}_t$ as base point $x$. In this section, we modify \cite[Algorithm~7.3.1]{goerner:tiling} to tile the space $U\subset\H^3$ with lifted truncated tetrahedra $m'(T_{t'}\cap U)$ about $K=x\in U$. That is, the output stream $(r_0, m_0T_{t_0}), (r_1, m_1T_{t_1}),\dots$ fulfills that $m_0(T_{t_0}\cap U),\dots, m_{i-1}(T_{t_{i-1}}\cap U)$ cover $\overline{B}_{r_i}(x)\cap U$.

The modifications to \cite[Algorithm~7.3.1]{goerner:tiling} are:
\begin{itemize}
\item Drop $f_\text{entry}$ and change the for-loop in $f$ to run for $f=0,\dots,3$ without checking $f=f_\text{entry}$.
\item Change the computation of $r'$ to be a lower bound for $d(x, m'(T_{t'}\cap U))$. Note that the algorithm never evaluates this for the truncated lifted tetrahedron $m'(T_{t'}\cap U)$ containing $x$.
\end{itemize}

We can do these modifications since $U$ has the swiss cheese property, see Section~\ref{sec:swissCheese}. Thus, for any $r\geq 0$, any tile $m'(T_{t'}\cap U)$ intersecting the ball $\overline{B}_r(x)\cap U$ in $U$ can be obtained from $\Id (T_t\cap U)=(T_t\cap U)\ni x=\vec{s}_t$ by applying a sequence of face-pairing matrices such that all intermediate tiles also intersect the ball $\overline{B}_r(x)\cap U$. We also refer the reader to \cite[Sections~1.2, 1.3 and~7]{goerner:tiling}.

We compute the lower bound $r'$ for $d(x, m'(T_{t'}\cap U))= d(m'^{-1}x, (T_{t'}\cap U))$ as
\begin{eqnarray*}
r'&=&\max\Big\{ d\big(m'^{-1}x, T_{t'}\big), ~~ d\big(m'^{-1}x, \overline{B}_{r\left(T_{t'}\cap U\right)}(\vec{s}_{t'})\big)\Big\}\\
&=&\max\Big\{\min\left\{d\left(m'^{-1}x, T^0_{t'}\right), \dots, d\left(m'^{-1}x, T^3_{t'}\right)\right\}, ~~ d\left(m'^{-1}x, \vec{s}_{t'}\right) - r(T_{t'}\cap U)\Big\}
\end{eqnarray*}
where $r(T_t\cap U)$ is given by Proposition~\ref{prop:radiusTruncTet} and where we assume that $x\not\in m'(T_{t'}\cap U)$. We can make this assumption since the tiling algorithm does not use this bound for the initial tile.


\subsection{Pre-length spectrum}

Using the modified tiling algorithm, we proceed as in Section~\ref{sec:prelengthSpectrumStream} replacing $r(S_t)$ with the upper bound $\overline{r}(S_t)$ from Section~\ref{sec:spunSpine}. By Proposition~\ref{prop:spunSpine2}, this results in a pre-length spectrum stream.

\section{Adding core curves} \label{sec:addCoreCurves}

The goal of this section is to compute the following information for the core curves in $M_\text{filled}$:
\begin{itemize}
\item For each core curve $\gamma$ corresponding to an incomplete cusp $i$, the complex length $\lambda$ and either a representative matrix $m\in\Gamma$ or a word $w$ in the face-pairing representation associated with $P$ depending on what output is desired. We use $\lambda$ and $m$ or $w$ in Step~\ref{step:addCoreCurves} of Algorithm~\ref{algo:lenSpec}.
\item For each vertex $\vec{v}^v_t\in\posLightCone$ of each tetrahedron $T_t$ corresponding to an incomplete cusp, the lift $c^v_t$ of the respective core curve that ends at the vertex of $\vec{v}$; also see Figure~\ref{fig:enclosingNbhd} and Section~\ref{sec:perTetCoreCurve}. We use $c^v_t$ to detect and avoid core curves and their multiples in the subsequent sections. Note that for this purpose, we do not need to determine an orientation of $c^v_t$ compatible with the core curve.
\end{itemize}

\subsection{Core curves from filling curves} \label{sec:coreFromFilling}

We call the peripheral curve in the cusp used for the Dehn filling to obtain $M_\text{filled}$ the filling curve.

For every cusp, a SnapPy triangulation $\myTrig$ of $M$ includes two primitive peripheral curves called meridian and longitude spanning the cusp's homology. These peripheral curves are encoded as chains in the dual $1$-skeleton of the cusp triangulation.

We focus on one incomplete cusp $i$ of $M$. A SnapPy triangulation gives the corresponding filling curve as linear combination $(p,q)$ of meridian and longitude. We obtain a peripheral curve homotopic to the core curve in $M_\text{filled}$ corresponding to the incomplete cusp $i$ as linear combination $(r,s)$ by solving for $ps-qr=1$.

\subsection{Core curve orientations}

The shapes of geometric spun-triangulations depends on the choice of core curve orientations. This is because the shapes do not just depend on a representation $\rho\from\pi_1(M)\to\myPSL(2,\C)$ but also a choice of $B$-decoration as defined in \cite[Definition~4.3]{gtz:vol} where $B$ are the diagonal matrices. A $B$-decoration can be thought of as a consistent way of picking one of the two Eigenvalues $d$ and $d^{-1}$ of $\rho(\gamma)$ for any peripheral curve $\gamma$ of $M$. For the geometric representation, orientations of the core curves uniquely determine a $B$-decoration by picking the Eigenvalue $d$ with $|d|>1$ for each oriented core curve.

\begin{example}
When filling the only cusp of the census manifold \texttt{m004}, we obtain a closed manifold such as \texttt{m004(6,1)}. Note that \texttt{m004(6,1)} and \texttt{m004(-6,-1)} give the same closed hyperbolic manifold $M_\text{filled}$ with the same geometric representation $\rho\from\pi_1(M_\text{filled})\to\myPSL(2,\C)$ up to conjugacy. However, the two spun-triangulations \texttt{m004(6,1)} and \texttt{m004(-6,-1)} differ in the orientation of the filling and core curve. Both spun-triangulations \texttt{m004(6,1)} and \texttt{m004(-6,-1)} actually admit a geometric structure but with different shapes.
\end{example}

\begin{example} \label{example:spunOneWay}
Here, we give the triangulation by its isomorphism signature \cite{burton:encode} which is also understood by SnapPy \cite{SnapPy}.
The spun-triangulation
\texttt{gLLMQcbeefefpjaqupw(1,1)} admits a geometric structure. Reversing the core curve, one of the shapes for the resulting \texttt{gLLMQcbeefefpjaqupw(-1,-1)} has negative imaginary part.  
 Note that the resulting closed manifold is homeomorphic to the closed census manifold \texttt{m247(-1,3)} and to the 3-fold cover of \texttt{m007(3,1)} which has no known geometric spun-triangulation.
\end{example}

\subsection{Lifted core curves}

Choose a point $b$ in the cusp triangulation where the chains for the meridian and longitude intersect. There is a unique path (up to homotopy) connecting $b$ the base point $x$ in the fundamental polyhedron $P$. Thus, the choice of $b$ gives us elements in $\pi_1(M,x)$ for the meridian and longitude. Let $g$ and $h$ be corresponding words in the face-pairing representation of $\pi_1(M,x)$ associated with $P$. We refer the reader to \cite[\texttt{fundamental\_group.c}]{SnapPy} for further details.

Let $\rho\from\pi_1(M,x)\to\SO(1,3)$ the geometric representation associated with $P$. The word $w$ is then given by $g^rh^s$ and the matrix $m$ by $\rho(w)=\rho(g)^r\rho(h)^s$. We can use the techniques in Section~\ref{sec:geoUtilities} to compute the complex translation length $\lambda$ and the fixed points of $m$. The fixed points span $c^v_t$ for the vertex $v$ of the tetrahedron $T_t$ corresponding to the cusp triangle containing $b$. We can recursively apply the face-pairing matrices to $c^v_t$ to compute $c^v_t$ for the other vertices of other tetrahedra corresponding to the same cusp.

\subsection{Large Dehn filling coefficient support} \label{sec:largeFillCoef}

This section describes future work not implemented in SnapPy version~3.3.

Note that computing the matrix $m$ as $\rho(w)=\rho(g)^r\rho(h)^s$ can be prohibitively expensive and is not even needed when an application is asking for the word $w$ instead:
\begin{example} \label{ex:largeDehn}
Consider the manifold \texttt{m004(10000001,1000002)}. For a peripheral curve parallel to the core curve, the smallest coefficients $(r,s)$ are $(1578947, 157895)$ and the corresponding word is $$w=b^{1578947} \left(c^{-1}acb^{-1}a^{-1}b\right)^{157895}.$$
\end{example}

Thus, we want to avoid computing the matrix $m$. Similarly, when computing the word $w$ as above, we want to avoid expanding the powers to a flat string of generators. Instead, we proceed as follows.

To compute $\lambda$, recall that a peripheral curve yields a monomial in the cross ratios $z_t, 1/(1-z_t), 1-1/z_t$. Let $L$ and $M$ the monomial for the meridian and longitude, respectively. Note that these are also the monomials making up the respective cusp gluing equation. Namely, the equation is $L^pM^q=1$. We can now give $\lambda$ as $r\log(L)+s\log(M)$. Note that this relies on picking the orientations of the manifold, of the meridian and longitude and, thus, of the filling and core curve in a consistent way.

To compute (the unoriented) $c^v_t$ for the cusp triangle containing $b$, we can use that at least one of $\rho(g)$ or $\rho(h)$ has non-zero real translation length and that the corresponding fixed points coincide with those of $\rho(w)$ (when ignoring whether a fixed point is attractive or repulsive). We obtain the other $c^v_t$ by applying the face-pairing matrices as above.

\section{Detecting core curves} \label{sec:coreCurveDetection}

We need to detect whether a given $h\in\Gamma^\text{hyp}$ corresponds to (a multiple of) a core curve in Step~\ref{algStep:rejectCoreCurve} of the inner loop of Algorithm~\ref{algo:lenSpec}.

We proceed similar to \cite[Section~5.2]{goerner:tiling}: compute the line $K$ fixed by $h$, pick a generic point $x\in K$ and use graph-tracing and graph-tracing verification algorithms of (\cite[Algorithm~6.1.2 and~6.1.3]{goerner:tiling}) to find one or two lifted tetrahedra containing $x$. However, we modify the graph tracing algorithm of \cite[Algorithm~6.1.2]{goerner:tiling} to detect core curves as follows:
\begin{itemize}
\item The output is an algebraic sum type: a lifted tetrahedron or an (incomplete) cusp.
\item At the beginning of the loop, we check that $mc^v_t$ coincides with $xx'$ where $x'=hx$. If we can verify this, we return the corresponding cusp. Otherwise, we continue.
\end{itemize}
If the output of this modified graph-trace algorithm is a lifted tetrahedron, we can use \cite[Algorithm~6.1.3]{goerner:tiling} to verify that $x$ is in the interior of $M$ and, thus, that $h$ does not correspond to a core curve. If the output is a cusp, then $h$ corresponds to a multiple of its core curve.\\
We integrate this modified graph-trace algorithm into Algorithm~\ref{algo:lenSpec} as follows: if the modified graph-trace algorithm returns an (incomplete) cusp, we simply skip the remained of the inner loop. Otherwise, we continue with Step~\ref{algStep:rejectDuplicate}. This step is described in Section~\ref{sec:graphTraceSet} and can re-use the result of the graph-tracing algorithm.

A priori, checking the condition that $x$ and $x'$ are on $mc^v_t$ requires exact arithmetic. However, we can use that any loop in an embedded tube about a core curve has to be homotopic to a multiple of the core curve. Thus, it is sufficient to verify that the line segment $xx'$ is contained in $mN^v_t$ to prove that $xx'$ coincides $mc^v_t$. That is, it is sufficient to verify that $d\left(x, mc^v_t\right)<r$ and $d\left(x', mc^v_t\right)<r$ where $r$ is the radius of the tube $N^v_t$. We can compute the left-hand side of these inequalities using the equation for $d(x, x_0x_1)$ in \cite[Section~3.2]{goerner:tiling}. We can compute $r=\sinh^{-1}\/s$ from the scaling factor $s$ in Section~\ref{sec:scalingNeighborhoodsIncomplete}.

\section{Avoiding core curves when tiling about a geodesic } \label{sec:coreCurveAvoidance}

While we rejected $m\in\Gamma^\text{hyp}$ corresponding to core curves in the previous section, the line $K$ fixed by $m\in\Gamma^\text{hyp}$ can still intersect a core curve. This causes a problem when we look-up $m$ in the set $\setOfGeodesics$ in Section~\ref{sec:graphTraceSet}: we need to find all lifted tetrahedra in $\left\{m^j:j\in\Z\right\}\backslash\H^3$ that intersect $K$. There are infinitely many such tetrahedra, if $K$ intersects a core curve; also see Remark~\ref{rem:infTilingWhenInterCore}.

We fix this by instead looking instead at the lifted tetrahedra that intersect $K\cap U$. The number of these lifted tetrahedra is finite (and non-zero). It is also at least one since every point in $U$ is covered by some lifted tetrahedron and $K$ cannot be completely contained in $\H^3 \setminus U$ since the corresponding loop in $M_\text{filled}$ is neither a peripheral or core curve. We might include additional lifted tetrahedra for which we could not verify that they do not intersect $K\cap U$.

Recall that Section~\ref{sec:graphTraceSet} is using \cite[Algorithm~7.3.1]{goerner:tiling} to tile about $K$. We need to modify that tiling algorithm to skip the lifted tetrahedra whose intersection $I$ with $K$ is entirely outside of $U$. If we skip such a lifted tetrahedron, we continue tiling at the two points in $K\cap\partial U$ closest to $I$. That is, we continue tiling with the lifted tetrahedra containing either one of these two points. This is necessary for the following reasons: the initial lifted tetrahedron might have the intersection $I$ entirely outside of $U$; the closed geodesic $\gamma\subset M_\text{filled}$ corresponding to $K$ might be cut into multiple connected components when removing the generalized cusp neighborhoods. Note that it is ok to revisit an earlier lifted tetrahedron since the algorithm uses a set to de-duplicate them. 





We implement this by adding the following step just before ``Emit $(r,mT_t)$.'' in \cite[Algorithm~7.3.1]{goerner:tiling} with $m$ referring to the variable bound inside the algorithm and not the above $m$:
\begin{algorithmSteps}
\item For each vertex $v$ of $mT_t$ corresponding to an incomplete cusp:
\begin{algorithmSteps}
\item If the intersection $I=K\cap mT_t$ is (verified to be) contained in the core curve tube $mN^v_t$:
\begin{algorithmSteps}
\item For each point $x\in K\cap \partial mN^v_t$:
\begin{algorithmSteps}
\item For each lifted tetrahedron $m'T_{t'}$ from the graph-trace algorithm \cite[Algorithm~6.1.2 and~6.1.3]{goerner:tiling} applied to $x$.
\begin{algorithmSteps}
\item Add $(-\infty, m'T_{t'}, -1)$ to $Q$.
\end{algorithmSteps}
\end{algorithmSteps}
\item Continue to the next iteration.
\end{algorithmSteps}
\end{algorithmSteps}
\end{algorithmSteps}

We now describe these steps in more detail.

$I$ is the intersection of $K$ with the four half-spaces of $\H^3$ defined by the normals $m\vec{n}_t^f$ of the faces of $mT_t$. We parametrize $K=x_0x_1$ by $\gamma(\tau)=\big((1-\tau)x_0+\tau x_1\big)\postNorm$ and solve for the times $\tau$ when $K$ enters or exits a half-space. If $K$ does not (or cannot be verified to) intersect a face transversally, we ignore the corresponding time (which we can do because it is safe to return a larger $I$). We can then compute The minimum $\tau_0$ and maximum $\tau_1$ of all entry and exit times, respectively. $I$ is now given as the line segment $\gamma(\tau_0)\gamma(\tau_1)$.

Similar to Section~\ref{sec:coreCurveDetection}, we have $I\subset mN^v_t$ if and only if $d\left(\gamma(\tau_0),mc^v_t\right)<r$ and $d\left(\gamma(\tau_1), mc^v_t\right)<r$ where $r$ again is the radius of $N^v_t$ again.

Solving $d\left(\gamma(\tau),mc^v_t\right)=r$ for $\tau$, the two points $x\in K\cap \partial mN^v_t$ are given by $\gamma(\tau)$.

As already pointed out in \cite[Remark~7.3.3]{goerner:tiling}, some of these computations might be faster in the tetrahedra view.

This finishes all the technical details needed to generalize the length spectrum algorithm to geometric spun-triangulations.

\section{Optimal Margulis number} \label{sec:margulis}

Let $M=\Gamma\backslash\H^3$ be an orientable, finite-volume hyperbolic $3$-manifold $M$. The $\varepsilon$-thin part of $M$ is the union of all essential loops of length less than $\varepsilon$ (or, equivalently, consists of all points with injectivity radius less than $\varepsilon/2$). We call $\varepsilon$ a \emphasisText{Margulis number} for $M$ if the $\varepsilon$-thin part is a union of embedded and disjoint tubes about simple closed geodesics and cusp neighborhoods. We already discussed how to determine whether $\varepsilon$ is a Margulis number for $M$ in \cite[Section~1.1.11]{goerner:tiling}. Here, we give an algorithm to compute (an interval containing) the \emphasisText{optimal Margulis number} $\mu(M)$ which is the supremum of all Margulis numbers $\varepsilon$ for $M$.

\subsection{Auxilary functions} \label{sec:margulisAux}

We use the same notation as in \cite[Sections~5.1 and~7.2]{goerner:tiling}. Let $K$ be a geodesic or horoball in $\H^3$ corresponding to a closed geodesic or cusp neighborhood in $M$. Let $\Gamma_K=\{m\in\Gamma: mK=K\}$ the stabilizer of $K$.

We denote by $B_r(K)$ the open neighborhood about $K$ of size $r$. If $K$ is a geodesic, $B_r(K)$ is empty when $r\leq 0$. If $K$ is a horoball, we use the signed distance so that $B_r(K)$ is a smaller horoball when $r<0$.

Given a real number $\varepsilon$, we denote by $r(K,\varepsilon)$ the radius $r$ such that $B_r(K)$ is the $\varepsilon$-thin part of $\Gamma_K\backslash\H^3$. If $K$ corresponds to a geodesic of real length less than $\varepsilon$ in $M$, we set $r(K,\varepsilon)=-\infty$. If $K$ is a horoball, $r(K,\varepsilon)$ can be negative.

Given a real number $r$, we denote by $\varepsilon(K, r)$ the supremal $\varepsilon$ such that $B_r(K)$ is $\varepsilon$-thin in $\Gamma_K\backslash\H^3$. If $K$ corresponds to a geodesic of length $\lambda$ in $M$, we set $\varepsilon(K,r)=\RePart(\lambda)$ when $r\leq 0$. In other words, $r(K,\_)$ and $\varepsilon(K,\_)$ are almost inverses of each other in that $\varepsilon(K, r)$ is the supremal $\varepsilon$ with $r(K,\varepsilon) \leq r$.

Let $K'$ be defined analogously to $K$. Given a real number $r$, we denote by $\mu(K, K', r)$ the supremal $\varepsilon$ such that $r(K,\varepsilon)+r(K',\varepsilon)\leq r$.

We discuss how to compute $r(K,\varepsilon)$, $\varepsilon(K, r)$ and $\mu(K, K', r)$ in Section~\ref{sec:margulisImpl}.

If $K$ and $K'$ correspond to the same cusp, we require $K$ and $K'$ to be the same cusp neighborhood in $M$. If $K$ and $K'$ are the same in $M$, let $d_M(K, K')$ be embedding diameter (which is twice the embedding radius) of $K$ (and $K'$) in $M$. Otherwise, $d_M(K, K')$ is the distance between $K$ and $K'$ in $M$; also see \cite[Section~10.1]{goerner:tiling}. Then, $\mu(K,K',d_M(K,K'))$ is the supremal $\varepsilon$ such the $\varepsilon$-thin neighborhoods of $K$ and $K'$ are disjoint (if $K$ and $K'$ are different in $M$) or embedded (otherwise) in $M$.

The optimal Margulis number $\mu(M)$ is the minimum of $\mu(K,K',d_M(K,K'))$ across all pairs $K$ and $K'$ where $K$ and $K'$ is each as above. Note that $\mu(K,K',d_M(K,K'))$ only depends on what cusp or geodesic $K$ and $K'$ correspond to in $M$. Furthermore, $\mu(K,K',d_M(K,K'))\geq \RePart(\lambda)$ if $K$ or $K'$ corresponds to a geodesic of length $\lambda$.

\subsection{Idea}

\subsubsection{As continuous growth process}

As a continuous process, we find the optimal Margulis number $\mu(M)$ by growing $\varepsilon$-thin neighborhoods about a system $L$ of cusps and geodesics for as long as they are embedded and disjoint. The system $L$ starts with just the cusps of $M$. When $\varepsilon$ is the real length of one or several geodesics, we need to add these geodesics to $L$.

\subsubsection{Using tiling steps}

We emulate continually growing $\varepsilon$ by running an instance of the tiling algorithm for each object $K_j$ representing a cusp or geodesic in such a system $L$. For now, assume we use exact arithmetic. For each $K_j$, we can convert the tiling radius $r$ to $\varepsilon=\varepsilon(K_j, r)$ such that the tiles cover an $\varepsilon$-thin neighborhood of $K_j$. The smallest $\varepsilon$ among all $K_j$ takes the role of the $\varepsilon$ in the continuous process. We grow it by either:
\begin{itemize}
\item adding one more tile for the cusp or geodesic $K_j$ corresponding to the smallest $\varepsilon$,
 or
\item adding the next geodesic $(\lambda_i, w_i)$ from the length spectrum stream (see Section~\ref{sec:putAlgo}) to $L$.
\end{itemize}
To decide which one to do, we check whether $\varepsilon < \RePart(\lambda_i)$ for each $K_j$ of $L$. If this is true, we add one more tile. Otherwise, we add the next geodesic.

We can use \cite[Proposition~10.1.1]{goerner:tiling} to determine whether the tiled $\varepsilon$-thin neighborhoods are embedded and disjoint. If they are, the smallest $\varepsilon$ is a Margulis number and thus a lower bound for $\mu(M)$. Otherwise, we might have overshot and the smallest $\varepsilon$ might be larger than $\mu(M)$ since this is not a continuous process.

At this point, we look at the tiled neighborhoods which are not embedded or disjoint. Let $K_j$ and $K_{j'}$ be a pair of such neighborhoods. That is, either $j\ne j'$ and the closed $\varepsilon$-thin neighborhoods of $K_j$ and $K_{j'}$ are not disjoint or $j=j'$ and the closed $\varepsilon$-thin neighborhood of $K_j=K_{j'}$ is not embedded. For each such pair, \cite[Proposition~10.1.1]{goerner:tiling} gives $d_M(K_j, K_{j'})$ which we can use to compute $\mu(K_j, K_{j'}, d_M(K_j, K_{j'}))$. The optimal Margulis number $\mu(M)$ is now given as minimum $\mu$ of all these $\mu(K_j, K_{j'}, d_M(K_j, K_{j'})$).

\subsubsection{Implementation using interval arithmetic}

We now modify the procedure to work with intervals. If we can verify that $\varepsilon < \RePart(\lambda_i)$ for each $K_j$ of $L$, we add one more tile, namely for the $K_j$ with the smallest $\underline{\varepsilon}$. Otherwise, we add the next geodesic. Note that even if $\varepsilon < \RePart(\lambda_i)$ is true for each $K_j$ of $L$ but cannot be verified, we can still add the next geodesic without compromising correctness.

At some point, we obtain tiled $\varepsilon$-thin neighborhoods for which we cannot verify that they are embedded and disjoint. \cite[Proposition~10.1.1]{goerner:tiling} gives us an interval $d$ for $d_M(K_j, K_{j'})$ for such pairs $(K_j, K_{j'})$. We compute $\mu$ as the minimum of $\mu(K_j, K_{j'}, d)$ across all such pairs $(K_j, K_{j'})$. Note that we might have multiple such pairs of $K_j$ and $K_{j'}$ because the values of the $\varepsilon$'s where their $\varepsilon$-thin neighborhoods are colliding are equal or so close together that we obtain overlapping intervals.

Note that while the $\mu$ computed this way is always an upper bound for $\mu(M)$ it might be larger than $\mu(M)$. That is, when computing $\mu$ as a minimum, we might actually miss pairs $(K_j,K_{j'})$ with the corresponding $\mu(M)$-thin neighborhoods touching in $M$. To fix this, we continue growing the neighborhoods and updating $\mu$ as we find new pairs $(K_j,K_{j'})$ for which we cannot verify that their $\varepsilon$-thin neighborhoods are embedded or disjoint. We stop when we can verify that $\mu$ is smaller than the $\varepsilon$ of any tiled $\varepsilon$-thin neighborhood.

To see that this suffices, note that when we stop, every $\varepsilon$ is greater than $\mu$ and thus $\mu(M)$. Thus, for each pair of cusps or geodesics with $\mu(M)$-thin neighborhoods touching in $M$, the corresponding tiled neighborhoods are intersecting and thus the corresponding $\mu(K_j, K_{j'}, d_M(K_j, K_{j'}))$ is included in the computation of $\mu$.

\subsection{Algorithm}

Algorithm~\ref{algo:optimalMargulis} implements the above idea.

\begin{algorithm}[h]
\begin{center}
\begin{tabular}{rp{14.9cm}}
{\bf Input:} & Developed fundamental polyhedron $P$ for $M\subset M_\text{filled}$.\\
& Iterator $(\lambda, w)$ for length spectrum stream.\\
{\bf Pre-condition:} & Boundary of interior of $\mu(M)$-thin part avoids core curves in $M_\text{filled}$.\\
{\bf Output:} & Optimal Margulis number $\mu=\mu(M_\text{filled})$.\\
{\bf Note:} & Uses notation $d_{i_j,i_{j'}}(K_j, K_{j'})$, $r\big[K_j\big]_{i_j}$ and $r\big[K_{j'}\big]_{i_{j'}}$ from \cite[Section~10.1]{goerner:tiling} which uses the tiling algorithm \cite[Algorithm~7.3.1]{goerner:tiling} and \cite[Algorithm~6.1.2 and~6.1.3]{goerner:tiling}) to seed the tiling algorithm.\\
\multicolumn{1}{l}{\bf Algorithm:}\\
\end{tabular}
\begin{algorithmSteps}
\item[L] List $K_0,\dots,K_{s-1}$ of standard geometric objects as defined in \cite[Definition~5.1.1]{goerner:tiling}. Starts empty. We instantiate the tiling algorithm whenever adding a $K_j$ to $L$. We denote the number of tiles pulled from the  stream for $K_j$ by $i_j$.
\item[S] Set of unordered pairs of indices  into $L$.
\item $\mu\leftarrow\infty$ [also see Remark~\ref{rem:margulisStopperFlag}]
\label{algoStep:margulisInit}
\item For each complete cusp:
\begin{algorithmSteps}
\item Add to $L$ a horoball $K$ corresponding to a cusp neighborhood of the cusp.
\end{algorithmSteps}
\item Repeat:
\begin{algorithmSteps}
\item Let $j$ be the index of the smallest $\underline{\varepsilon(K_0,r\big[K_0\big]_{i_0})}, \dots, \underline{\varepsilon(K_{s-1},r\big[K_{s-1}\big]_{i_{s-1}})}$.
\item If $\underline{\RePart(\lambda)}\leq \underline{\varepsilon(K_j,r\big[K_j\big]_{i_j})}$:
\begin{algorithmSteps}
\item Add to $L$ the line $K$ in $\H^3$ fixed by the matrix corresponding to word $w$.
\item Advance iterator for length spectrum stream.
\item Continue to next iteration.
\end{algorithmSteps}
\item If $\mu<\varepsilon(K_j,r\big[K_j\big]_{i_j})$: Return $\mu$.
\item Pull one more tile from the stream for $K_j$ (incrementing $i_j)$.
\item For each $j' = 0,\dots,s-1$:
\begin{algorithmSteps}
\item If $\left\{j, j'\right\}$ is not in $S$: 
\begin{algorithmSteps}
\item $d_\text{max} \leftarrow d_{i_j, i_{j'}}(K_j,K_{j'})$.
\item $R \leftarrow r\big[K_j\big]_{i_j} + r\big[K_{j'}\big]_{i_{j'}}$.
\item If not $R < d_\text{max}$:
\begin{algorithmSteps}
\item $d\leftarrow \left[\underline{\min(R, d_\text{max})}, \overline{R}\right]$
\item $\mu\leftarrow \min\left(\mu, \mu\big(K_j,K_{j'}, d \big)\right)$
\item Add $\left\{j,j'\right\}$ to $S$.
\end{algorithmSteps}
\end{algorithmSteps}
\end{algorithmSteps}
\end{algorithmSteps}
\end{algorithmSteps}
\end{center}
\vspace{-0.5cm}
\caption{Optimal Margulis number.\label{algo:optimalMargulis}}
\end{algorithm}

The implementation allows for geometric spun-triangulations as well. However, recall from \cite[Section~7.2]{goerner:tiling} that we cannot tile beyond a core curve. Thus, the algorithm fails if the pre-condition about the core curves is not fulfilled (by running into an infinite loop or eventually a failed verification depending on whether we use exact or interval arithmetic). Examples where the algorithm fails are given in Section~\ref{sec:proofcensusMargulisConstant}.

\begin{remark}
Note that we can avoid this problem by switching to finite triangulations.
\end{remark}

\begin{remark} \label{rem:margulisStopperFlag}
Even if the pre-condition fails, we can still verify for small enough $\varepsilon$ that $\varepsilon$ is a Margulis number for $M$. More generally, we can give a stopper value $\mu_\text{stopper}$ to the algorithm and initialize $\mu\leftarrow\mu_\text{stopper}$ in Step~\ref{algoStep:margulisInit} of Algorithm~\ref{algo:optimalMargulis}. The algorithm now returns a lower bound for $\mu(M)$ rather than $\mu(M)$. More precisely, it returns (an interval for) $\min (\mu_\text{stopper}, \mu(M))$. Examples are in Section~\ref{sec:proofcensusMargulisConstant}.
\end{remark}

\subsection{Implementation details} \label{sec:margulisImpl}

\subsubsection{$r(K,\varepsilon)$} \label{sec:implR}

Let $K$ correspond to a closed geodesic of complex length $\lambda$. Then, \cite[Proposition~3.10 and Remark~3.11]{fps:margulis} give us
\[
r(K,\varepsilon)=\max\limits_{n=1,\dots,N}\left\{\cosh^{-1}\sqrt{\max\left(1,\frac{\cosh(\varepsilon)-\cos(n\ImPart(\lambda))}{\cosh(n\RePart(\lambda))-\cos(n\ImPart(\lambda))}\right)}\right\}\quad\text{for any}\quad N\geq \lfloor \varepsilon / \RePart(\lambda)\rfloor.
\]
We can use $N=\left\lfloor \overline{\varepsilon / \RePart(\lambda)}\right\rfloor$ for verified computations.

Let $K$ correspond to a cusped neighborhood $C$ of area $A$ and shape $s\in\C$ (which is such that the torus $\C/(\Z+s\Z)$ is similar to $\partial C$). We have (also see \cite[Section~1.1.12]{goerner:tiling})
\[
r(K,\varepsilon)= \log\left( \frac{2 \sinh(\varepsilon/2)}{L} \right)\quad\text{with}\quad L=\frac{\sqrt{A/\ImPart(s)}}{\min_{(p,q)\in\Z^2\setminus 0} |p + qs|}.
\]
To see this, note that $2\sinh(\varepsilon/2)$ is the Euclidean distance measured along a horosphere $H$ between two points on $H$ with hyperbolic distance $\varepsilon$ and that $L$ is the (Euclidean) length of a shortest non-contractible curve in $\partial C$.

\subsubsection{$\varepsilon(K,r)$} \label{sec:implEps}

Let $K$ correspond to a closed geodesic of complex length $\lambda$. Solving one of the terms contributing to $r(K,\varepsilon)$ for $\varepsilon$, we get
\[
\EpsilonCandidate(r,\sigma,\tau)=\cosh^{-1}\left(\cosh(r)^2 (\cosh(\sigma)-\cos(\tau)) + \cos(\tau)\right)
\]
and, thus, have
\[
\varepsilon(K,r)=\min\limits_{n=1,2,\dots} \EpsilonCandidate(r,n\RePart(\lambda),n\ImPart(\lambda)).
\]
To evaluate only finitely many terms, we use that $E(r,n\RePart(\lambda),0)$ is a lower bound for all terms following $E(r,n\RePart(\lambda),n\ImPart(\lambda))$.

If $K$ is a cusp neighborhood, it is straight-forward to invert the function $r(K,\_)$.

\subsubsection{$\mu(K,K',r)$} \label{sec:implMu}

If $K$ and $K'$ are the same, we use $\mu(K,K',r)=\varepsilon(K,r/2)$. Otherwise, we need to solve for $f(\varepsilon)=r(K,\varepsilon)+r(K',\varepsilon)=r$.

If $K$ and $K$ are cusp neighborhoods, $f(\varepsilon)$ is $2\log(2\sinh(\varepsilon/2))+C$ up to a constant $C$ and, thus, can be solved explicitly.

In the remaining cases, at least one of $K$ and $K'$ is a geodesic and it is not easy to solve $f(\varepsilon)=r$ for $\varepsilon$ explicitly. Note that $f$ has a single discontinuity at $\varepsilon_0$ where $\varepsilon_0$ is (the maximum of) the real length of the geodesic(s). When approaching the discontinuity from the right, $f$ approaches $r_0=f(\varepsilon_0)$ but its derivative becomes infinite. Outside the discontinuity, $f$ is only piece-wise smooth. To compute an interval for $r_0$, we use $r(K,\varepsilon_0)+r(K',\varepsilon_0)$ where we use the continuous versions of $r(K,\varepsilon)$ which is $0$ when $K$ is a geodesic and $\varepsilon$ is smaller than its real length.

If we can verify that $r < r_0$, we know that $\mu(K,K',r)=\varepsilon_0$.

If we can verify that $r > r_0$, we use Newton's method to solve for $f(\varepsilon)=r$. We guard against the singularity by clamping each iteration to get at most 8 times closer to $r_0$. We store the absolute value $|\delta|$ of the error term $\delta=f(\varepsilon)-r$ in either $\delta_+$ or $\delta_-$ depending on the sign of $\delta$. We stop iterating if $|\delta|$ is not less than the previous value of the respective $\delta_+$ or $\delta_-$. To verify the result, we can use the Newton interval method for $f$ even though it has points where the derivative when approaching from the left differs from the derivative when approaching from the right. This requires us to implement the interval version of the derivative $f'(\varepsilon)=r'(K,\varepsilon)+r'(K',\varepsilon)$ of $f$ with respect to $\varepsilon$. We focus on the first term $r'(K,\varepsilon)$. If $K$ is a horoball, we can simply evaluate the expression for $r'(K,\varepsilon)$. If $K$ is a geodesic, recall from Section~\ref{sec:implR} that $r(K, \varepsilon)$ is the maximum of $h_1(\varepsilon),\dots,h_N(\varepsilon)$ where each $h_n$ is a function. Given an interval $\varepsilon$, consider each $h_n(\varepsilon)$ overlapping the interval for $r(K,\varepsilon)$. An interval for $r'(K,\varepsilon)$ is given by the union of all the respective $h'_n(\varepsilon)$. As usual in interval arithmetic, the union of intervals is the smallest interval containing all intervals.

If we can verify neither $r < r_0$ nor $r>r_0$, we use that $\varepsilon_0$ is a lower bound for $\mu(K,K',r)$ and that $\varepsilon(K,r)$ and $\varepsilon(K',r)$ are upper bounds for $\mu(K,K',r)$.

\subsubsection{Other}

Similar to \cite[Section~7.3]{goerner:tiling}, we use a priority queue to find the index $j$ of the smallest $\underline{\varepsilon(K_j,r\big[K_j\big]_{i_j})}$.

Rather than recomputing $d_{i_j, i_{j'}}(K_j,K_{j'})$ from scratch every time, we store its value and accumulate it. That is, update it by taking the minimum of its old value and the distances between tiles not considered yet. 

\subsection{Optimal Margulis number for SnapPy census} \label{sec:proofcensusMargulisConstant}


\begin{proof}[Proof of Theorem~\ref{thm:censusMargulisConstant}]
We use the same geometric spun-triangulations as in Example~\ref{example:censusTimings}. We apply SnapPy's \texttt{Manifold.margulis} which implements Algorithm~\ref{algo:optimalMargulis} to either compute the optimal Margulis number $\mu(M)$ of a manifold $M$ or, following Remark~\ref{rem:margulisStopperFlag}, to verify $\mu(M)>0.8$ as follows:
\begin{verbatim}
    RIF = RealIntervalField(500)
    m = M.margulis(verified=True, bits_prec=500, stopper=RIF("0.81"))
    if not m > RIF("0.8"):
        raise Exception("0.8 is not a Margulis number")
\end{verbatim}

The latter method of verifying $\mu(M)>0.8$ is useful to save computation time. In particular, we choose it for all triangulation with 11 tetrahedra in the orientable cusped census.

The method of verifying $\mu(M)>0.8$ is essential to address cases where the algorithm cannot compute $\mu(M)$ itself. This happens for geometric spun-triangulations where a core curve intersects a $\varepsilon$-thin neighborhood about a non-core curve geodesic (or unfilled cusp) for $\varepsilon=\mu(M)$. The following five geometric spun-triangulations in \cite[\texttt{ClosedManifolds.zip}]{hikmot} of the following five closed census manifolds exhibit such a core curve. Luckily, in each of these cases, we are able to still verify $\mu(M)>0.8$:
\begin{center}
\begin{tabular}{lllll}
\texttt{Closed135.tri}& \texttt{Closed369.tri}&\texttt{Closed739.tri}& \texttt{Closed1280.tri}& \texttt{Closed1401.tri} \\
\texttt{m120(-4,1)}& \texttt{s479(-3,1)}& \texttt{s572(1,2)}& \texttt{s912(0,1)}& \texttt{v2203(-1,2)}
\end{tabular}
\end{center}

For some of the triangulations, it is necessary to increase the precision beyond 500bits.

Note that \texttt{m007(3,1)} has no known geometric spun-triangulation.  
\end{proof}

\begin{remark} \label{rem:margulisVol3}
The following unverified computation indicates that $\muCensus{m007(3,1)}=0.831442\dots$ which is also the systole of \texttt{m007(3,1)}: Since the manifold has no cusps, $\muCensus{m007(3,1)}$ is at least the systole. We need to show that two shortest geodesics in \texttt{m007(3,1)} intersect. To do this, we work in the 3-fold cyclic cover that has the geometric spun-triangulation \texttt{gLLMQcbeefefpjaqupw(1,1)}; also see Example~\ref{example:spunOneWay}. It has exactly two primitive geodesics of complex length $2.49432\dots + 0.446593\dots i$. Each of these two geodesics must cover a primitive geodesic in \texttt{m007(3,1)} three times since \texttt{m007(3,1)} has no primitive geodesic of such complex length (which we determined using the unverified \texttt{Manifold("m007(3,1)")\allowbreak.length\_spectrum(2.5)}). Furthermore, these two geodesics intersect in the inside view:
\begin{verbatim}
   >>> M = Manifold("gLLMQcbeefefpjaqupw(1,1)")
   >>> M.inside_view(geodesics=['aagFEdEdfGFC','aBDcfgFEc'])
\end{verbatim}
Thus, the corresponding geodesics in \texttt{m007(3,1)} must also intersect in \texttt{m007(3,1)}. Both of them and have real length equal to the systole of \texttt{m007(3,1)}.
\end{remark}

\section{Ruling out Cosmetic Surgeries in the cusped census }\label{sec:cosmetic}

We now prove of Theorem~\ref{thm:cosmetic}:

\begin{proof}[Proof of Theorem~\ref{thm:cosmetic}]
We extend the proof of \cite[Theorem 4.3]{FPS:Cosmetic} by considering the orientable cusped census manifolds with 10 and 11 tetrahedra. The new census manifolds require some modifications to the code originally used for \cite[Theorem 4.3]{FPS:Cosmetic}. The modified code is again available at \cite{Schleimer:CosmeticCode}. Besides the cases already listed in \cite[Theorem 4.3]{FPS:Cosmetic}, there are now two new pairs $(M(s), M(t))$ of oriented manifolds that the code cannot distinguish. In both cases, the description given by \cite{Schleimer:CosmeticCode} is the same for $M(s)$ and $M(t)$:
\begin{center}
\renewcommand{\arraystretch}{1.18}%
\begin{tabular}{|c|c|c|c|}
\hline
$M$ & $s$ & $t$ & Description\\ \hline \hline
\texttt{o10\_148196} & $(1,-1)$ & $(1,1)$ & \verb"SFS [D: (2,1) (2,-1)] U/? m011" \\ \hline
\texttt{o11\_451858} & $(1,-1)$ & $(1,1)$ & \verb"SFS [D: (2,1) (2,-1)] U/? m016" \\ \hline
\end{tabular}
\end{center}
Note that each $M(s)$ and $M(t)$ is obtained by gluing a Seifert fiber space $X$ and a hyperbolic census manifold $Y$ along a torus, making $X$ and $Y$ the pieces of their JSJ decompositions. Since $X$ and $Y$ are different, an automorphism of $M(s)$ or $M(t)$ has to fix $X$ and $Y$. Furthermore, such an automorphism has to preserve the orientation since $Y$ is chiral which we verify using SnapPy. Thus, for each of the above pairs $(M(s), M(t))$, all isomorphisms between $M(s)$ and $M(t)$ have to be either orientation preserving or reversing. To show that the latter is the case, we use the function \verb"admit_reversing_homeo" from \cite{Schleimer:CosmeticCode}. Thus, $M(s)$ and $M(t)$ are different as oriented manifolds.
\end{proof}

\section{Performance} \label{sec:performance}

Note that, unlike the Hodgson--Weeks algorithm (see Remark~\ref{rem:hwDirichletIsolation}), the new algorithm is fairly sensitive to the choice of triangulation. Thus, before measuring the performance, we discuss how to heuristically find a suitable triangulation to accelerate the new algorithm in Section~\ref{sec:perfTrigDep}.

For a performance analysis, it is important to measure the actual performance using a test set representing a wide variety of manifolds. This applies even more so when comparing the Hodgson--Weeks and the new algorithm since there is no good predictor for the performance common to both algorithms (the spine radius of the Dirichlet domain is a good predictor for the Hdgson-Weeks algorithm but not the new algorithm). We use the census manifolds as such a test set and revisit the Example~\ref{example:censusTimings} in Section~\ref{sec:perfCensus}. The more nuanced view in that section reveals that while the new algorithm performs much better on average, there are also cases where it performs worse than the Hodgson--Weeks algorithm.

The most extreme of these cases are obtained from gluing regular ideal tetrahedra or octahedra. Thus, Section~\ref{sec:perfTetMfds} continues the performance analysis of the new algorithm by focusing on manifolds obtained from regular ideal tetrahedra. For the resulting triangulations, an asymptotic analysis of the runtime of the new algorithm shows that the runtime is linear in the number of tetrahedra $n$ and exponential in the cut-off length $\mu$. Example~\ref{ex:estimateNew} shows that resulting expression is not just asymptotically correct but also a good predictor for the runtime of practical examples. This makes the new algorithm scale better than the Hodgson--Weeks algorithm for the family of triangulations in Example~\ref{ex:cyclicCoversPerf} but makes it perform worse for the ones in Example~\ref{ex:regTessPerf}.

When comparing the runtimes of the new algorithm and the older Hodgson--Weeks algorithm, one should also keep in mind that the implementations of the two algorithms use different languages and numeric types. This puts the new algorithm at a disadvantage by a factor that could be up to a magnitude as discussed in Section~\ref{sec:implPerf}.

\subsection{Performance dependence on spun-triangulation} \label{sec:perfTrigDep}

The performance of the new algorithm is fairly sensitive to the choice of spun-triangulation of a manifold $N$. That is, it is sensitive to both the choice of surgery description $M_\text{filled}\cong N$ of $N$ and the choice of geometric spun-triangulation of the unfilled manifold $M$.

\begin{remark} \label{rem:hwDirichletIsolation}
In contrast, the performance of the Hodgson--Weeks algorithm is mostly isolated from the choice of triangulation. This is because it uses a Dirichlet domain about a base point $x$ that locally maximizes the injectivity radius and the bulk of the runtime is in the steps subsequent to finding the (unverified) Dirichlet domain. This particularly applies to cases where the spine radius of the Dirichlet domain is large and the length spectrum computation costly such as in Example~\ref{example:fast}.
\end{remark}

Heuristically, we can improve the performance of the new algorithm by:
\begin{itemize}
\item Turning short geodesics into core curves. That is, drilling a short geodesic first and then looking at the  incomplete geometric structure that is completed by the geodesic we just drilled. The acceleration can be explained by the potential reduction in the number of tetrahedra and the large tubes about the core curves which we avoid when tiling. See Example~\ref{example:fast}. Section~\ref{sec:codeSurgeryDescription} gives a function to apply this heuristic strategy.
\item Turning long core curves into geodesics. That is, permanently filling the incomplete cusps. See Example~\ref{example:m004Trigs}. For a closed manifold, this ultimately results in a finite triangulation when we fill the last incomplete cusp. Since SnapPy does not support geometric finite triangulations, we expect but cannot confirm that we can accelerate the new length spectrum algorithm by switching to a finite triangulation for closed manifolds with long systole. We refer the reader to Section~\ref{sec:futureFinTrig} for future work.
\item Simplifying the triangulation. See Example~\ref{ex:simplifyM004}.
\end{itemize}

\begin{example} \label{example:m004Trigs}
Computing the systole for the following spun-triangulation takes 2.5s:
\begin{verbatim}   >>> M = Manifold("m129(-3,1)(0,0)") # isometric to m004
   >>> M.cusp_info()[0]['core_length'] # Somewhat long core curve
   1.08707... + 1.722768...*I
   >>> M.length_spectrum_alt(count=1)
   [Length                           Core curve  Word
    1.08707... - 1.722768...*I       -           abcDB,
    1.08707... + 1.722768...*I       Cusp 0      ac]\end{verbatim} %
If we fill the incomplete cusp, we get the same triangulation of \texttt{m004} that is in the census. Computing the systole now takes 50ms:
 \begin{verbatim}
    >>> M=M.filled_triangulation().with_hyperbolic_structure()
    >>> M.num_tetrahedra() # combinatorially isomorphic to m004
    2
    >>> M.length_spectrum_alt(count=1)
    [Length                           Core curve  Word
     1.08707... - 1.722768...*I       -           aB,
     1.08707... + 1.722768...*I       -           c]\end{verbatim} %
 \end{example} %
 
\begin{example} \label{ex:simplifyM004}
Here is a geometric ideal triangulation of \texttt{m004} with 4 tetrahedra. Computing the systole takes 670ms:
\begin{verbatim}
   >>> M = Manifold("eLPkbcdddmbkgw") # isometric to m004
   >>> M.num_tetrahedra()
   4
   >>> M.length_spectrum_alt(count=1)
   [Length                                      Core curve  Word
    1.08707... - 1.722768...*I       -           abac,
    1.08707... + 1.722768...*I       -           aab]
\end{verbatim}
Simplifying again gives the triangulation of \texttt{m004} that is also in the census. Computing the systole now only takes 50ms:
\begin{verbatim}
    >>> M.simplify()
    >>> M.num_tetrahedra() # combinatorially isomorphic to m004
    2
\end{verbatim}    
\begin{verbatim}
    >>> M.length_spectrum_alt(count=1)
   [Length                                      Core curve  Word
    1.08707... + 1.722768...*I       -           b,
    1.08707... - 1.722768...*I       -           aC]    
\end{verbatim}
\end{example}

\subsubsection{Heuristic code to prepare spun-triangulation for cusped manifolds} \label{sec:codeSurgeryDescription}

To accelerate the computations in Example~\ref{example:censusTimings} and Section~\ref{sec:proofcensusMargulisConstant} for the cusped case, we use the geometric spun-triangulation obtained by applying the following function to the triangulation in the census\footnote{Requires upcoming SnapPy version~3.4 since some census triangulations did not admit a hyperbolic structure.}:
\begin{verbatim}
    def surgery_description(M):
        G = M.fundamental_group(simplify_presentation=False)
        gens = [ (l, g) for g in G.generators()
                        if (l := G.complex_length(g).real()) > 1e-6 ]
        if len(gens) == 0: return M
        l, g = min(gens)
        if l > 0.4: return M
        # verified = True not actually needed: we fill the new 
        # cusp right away so it does not matter what curve we
        # actually drill.
        try: N = M.drill_word(g, bits_prec=1000)
        except: return M
        N.dehn_fill((1,0),-1)
        for i in range(1000):
            if N.solution_type() == 'all tetrahedra positively oriented':
                return N
            N.randomize()
        return M
\end{verbatim}

\begin{remark}
Note that the above code only drills one short geodesic. Users might choose to extend it to drill more than one short geodesic to accelerate some cases even further.
\end{remark}

\begin{example} 
The remark applies, for example, to \texttt{o9\_05486}: We can further accelerate the length spectrum computation by turning the geodesics corresponding to the words \texttt{a} and \texttt{f} into core curves. Similar for the unique double cover of \texttt{o9\_00637} and the words \texttt{a} and \texttt{n}. All these geodesics have length less than 0.1.
\end{example}

\subsection{Performance for census manifolds} \label{sec:perfCensus}

We now revisit Example~\ref{example:censusTimings} and show some more details about the performance of the algorithms for the orientable cusped census manifolds up to 9 tetrahedra:
\begin{center}
\begin{tabular}{r||r||r||r|r}
Census & Tetrahedra & Case & Hodgson--Weeks & New algorithm \\
& & & 212bits & 500bits \\
& & & C++ & Python \\  \hline \hline
\multirow{3}{*}{Cusped up} & \multirow{3}{*}{$\leq 9$} & Average & 16.1s & 251ms \\ \cline{3-5}
 & & Median & 46ms & 161ms \\  \cline{3-5}
 & &  Worst & 18h 26m & 2.72s
\end{tabular}
\end{center}

\begin{remark}
Note that Hodgson--Weeks algorithm requires a cut-off length. Thus, when measuring the time to compute the systole, we first ran the new algorithm so that we can use the result (plus some small amount) as cut-off length for the Hodgson--Weeks algorithm.
\end{remark}

For the Hodgson--Weeks algorithm, the worst case is realized by the manifold \texttt{o9\_00636} which also has the largest spine radius of 3.717\dots~. More generally, the performance of the Hodgson--Weeks algorithm for the census manifolds is largely determined by the spine radius of the Dirichlet domain. For example, all manifolds taking longer than an hour also have a spine radius above 3.38 and all manifolds taking less than an hour have a spine radius below 3.45.

Note that unlike in the average and worst case, the implementation of the new algorithm actually performs worse in the median case. This means that there are many cases where the implementation of the new algorithm is actually slower than implementation of the Hodgson--Weeks algorithm.

For the new algorithm, the worst case is realized by the manifold \texttt{t12067}. It is homeomorphic to the complement of the Borromean rings and can be obtained by gluing together two octahedra such that many symmetries of the octahedra extend to the manifold. This is an example of a more general pattern we revisit in Example~\ref{ex:regTess}: for highly symmetric manifolds, the Hodgson--Weeks algorithm performs very well and better than the new algorithm. 

\subsection{Tetrahedral hyperbolic manifolds} \label{sec:perfTetMfds}

In this section we focus on geometric triangulations $\myTrig$ where all tetrahedra are ideal and regular.

This allows the following asymptotic analysis: Let $n$ be the number of tetrahedra of $\myTrig$. Let $\mu>0$ be the cut-off length up to which we compute the length spectrum.  Recall that Algorithm~\ref{algo:preLen} invokes the tiling algorithm for each tetrahedron. Note that the spine radius for each tetrahedron computed in Section~\ref{sec:spineIncomplete} is bounded by a constant when choosing the horocusp neighborhoods as in Section~\ref{sec:scalingNeighborhoodsIncomplete} (to see this, consider the regular ideal tetrahedron, take the ``large'' horoballs about the vertices that are touching the incenter and take the ``small'' horoballs about the vertices that are each touching three of the large horoballs; since each cusp neighborhood is stopped by the incenter, it leaves the other cusp neighborhoods enough room to expand so that each cusp neighborhood corresponds to a horoball between the small and large horoballs). Thus, for large $\mu$, each invocation needs to tile a ball of radius approximately $\mu$. The volume of this ball and, thus, the number of tiles to cover it grow as $e^{2\mu}$. Hence, we expect the total number of tiles $T$ used to compute the length spectrum to grow as $n e^{2\mu}$. For an asymptotic analysis of the runtime of the algorithm, we need some additional factors such as the depth of the trees managing these tiles. But these factors are minor. Note that the analysis does not take into account the necessity to increase precision as $n$ or $\mu$ grows.

The following example shows that $n e^{2\mu}$ is a good estimate for the runtime of the new algorithm even for small values of $n$ and $\mu$:

\newpage

\begin{example} \label{ex:estimateNew}
We consider cyclic covers of \texttt{m004}. Given the degree $d$, there is a unique such cover which has $n=2d$ regular ideal tetrahedra and one cusp. We compute its length spectrum up to a given cut-off length $\mu$ using the new algorithm:
\begin{verbatim}
    >>> M=Manifold("m004").covers(d, cover_type='cyclic')[0]
    >>> M.length_spectrum_alt(max_len=mu, bits_prec=1000)
\end{verbatim}
We compare the above estimate to the actual runtime for different values of $d$ and $\mu$ in Figure~\ref{fig:runtimeAsymp}.
\begin{figure}[h]
\begin{center}
\includegraphics[width=10cm]{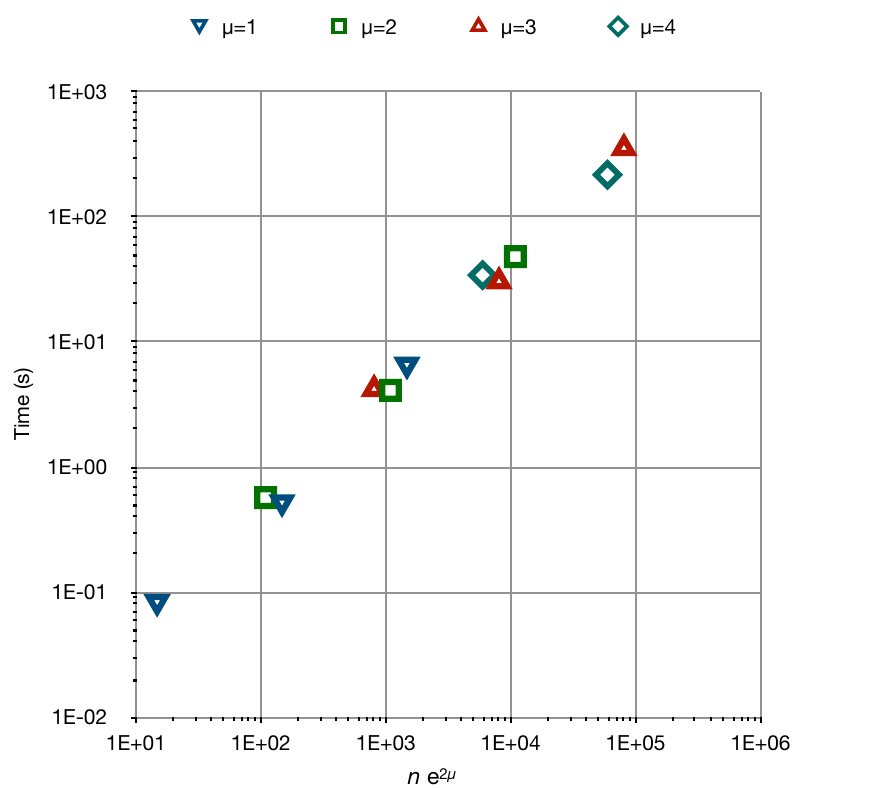}
\end{center}
\caption{Estimated vs actual runtime to compute the length spectrum with the new algorithm. The cut-off length $\mu$ is $1$, $2$, $3$ or $4$ and the manifold is the cyclic cover of \texttt{m004} with degree $d=1$, $10$ or $100$ (without the case $d=100, \mu=4$).\label{fig:runtimeAsymp}}
\end{figure}
\end{example}

\begin{example} \label{ex:cyclicCoversPerf}
We again consider the cyclic covers of \texttt{m004}, but now compare the performance of the new algorithm to that of the Hodgson--Weeks algorithm. For this, we fix the cut-off length $\mu$ to 1.087070\dots which is the systole of each cover. The runtimes are shown in Figure~\ref{fig:m004CoverTimes}. Note that while the new algorithm scales linearly in $d$ as predicted, the runtime of the Hodgson--Weeks algorithm grows exponentially.
\begin{figure}[h]
\begin{center}
\includegraphics[width=8cm]{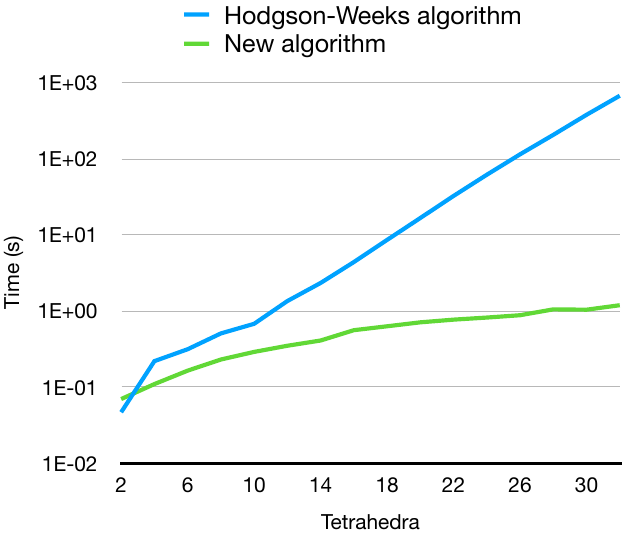}
\end{center}
\caption{Time to compute the systole of the cyclic cover of \texttt{m004} using the Hodgson--Weeks algorithm and the new algorithm. \label{fig:m004CoverTimes}}
\end{figure}
\end{example}

\newpage

\begin{example} \label{ex:regTessPerf}
\label{ex:regTess}
We expect the Hodgson--Weeks algorithm to perform better when the Dirichlet domain is fairly balanced such as in the case of the highly symmetric regular tessellation manifolds (see \cite{BGR:PrinCong} and \cite{Goerner:RegTessLinkComps} for notation and further background). Indeed, computing the systole is a lot faster with the Hodgson--Weeks algorithm than with the new algorithm in these cases:
\begin{center}
{
\renewcommand{\arraystretch}{1.3}
\begin{tabular}{ll||rl|c||r|r}
\multicolumn{2}{c||}{Manifold} & \multicolumn{2}{c|}{Solids} & Systole & Hodgson--Weeks & New algorithm \\
 \cite{Goerner:RegTessLinkComps} & \cite{BGR:PrinCong} & & & & 212bits & 1000bits \\ \hline \hline
$\mathcal{U}^{\{3,3,6\}}_{2}$ & $\big(3, \big\langle 2\big\rangle\big)$ & 10 & Tetrahedra & 2.633915\dots & 0.19s & 7.89s \\ \hline
$\mathcal{U}^{\{3,3,6\}}_{2+\zeta}$ & $\big(3, \big\langle \left(5+\sqrt{-3}~\right)\big/2\big\rangle\big)$ & 28 & Tetrahedra & 3.261210\dots & 4.62s & 62\phantom{.00}s \\ \hline
$\mathcal{U}^{\{3,3,6\}}_3$ & $\big(3, \big\langle 2\big\rangle\big)$ & 54 & Tetrahedra & 3.849694\dots & 43.2\phantom{0}s & 452\phantom{.00}s \\ \hline
$\mathcal{U}^{\{3,4,4\}}_2$ & $\big(1, \big\langle 2\big\rangle\big)$ & 4 & Octahedra & 3.057141\dots & 0.57s & 34.2\phantom{0}s \\ \hline
$\mathcal{U}^{\{3,4,4\}}_{2+i}$ & $\big(1, \big\langle 2+\sqrt{-1}~\big\rangle\big)$ & 5 & Octahedra & 2.633915\dots & 0.42s & 17.0\phantom{0}s \\ \hline
$\mathcal{U}^{\{3,4,4\}}_{2+2i}$ & $\big(1, \big\langle 2+2\sqrt{-1}~\big\rangle\big)$ & 16 & Octahedra & 3.525494\dots & 18.2\phantom{0}s & 344\phantom{.00}s \\ \hline
$\mathcal{U}^{\{4,3,6\}}_{1+\zeta}$ & & 6 & Cubes & 2.633915\dots & 1.23s & 16.5\phantom{0}s \\ \hline
$\mathcal{U}^{\{4,3,6\}}_2$ & & 16 & Cubes & 3.525494\dots & 47.8\phantom{0}s & 355\phantom{.00}s \\
\end{tabular}
}
\end{center}
Note that we the octahedral manifolds $\mathcal{U}^{\{3,4,4\}}_z$ are included even though they do not admit a triangulation by regular ideal tetrahedra (the given triangulations of $\mathcal{U}^{\{3,4,4\}}_z$ are obtained by subdividing each regular ideal octahedron into four ideal tetrahedra). The triangulations used for the cubical manifolds $\mathcal{U}^{\{4,3,6\}}_z$ are obtained by dividing each cube into 5 regular ideal tetrahedra.  
\end{example}

\subsection{Implementation choices and precision} \label{sec:implPerf}

The Hodgson--Weeks algorithm is implemented in C/C++ and uses native doubles with 53bits precision or quad-doubles \cite{qd} with 212bits precision. Using quad-doubles is about 50 times slower than the double implementation. To compute the systole with the Hodgson--Weeks algorithm, double precision is insufficient for about a quarter of all orientable cusped census manifolds and, in particular, for almost all those orientable cusped census manifolds where the computation took long (more than an hour with quad-double precision). For simplicity, we used quad-double for all computations using the Hodgson--Weeks algorithm.

The new algorithm is implemented in Python and, for all computations in this paper, we ran the new algorithm using SageMath's floating-point and interval type \cite{SageMath}. The libraries underlying these types are mpfr \cite{mpfr} and mpfi \cite{mpfi} library, respectively. A quick benchmark we wrote in C++ reveals that (when preallocating memory) mpfr is about half the speed of quad-doubles at the same precision. However, it is harder to measure the performance penalty incurred by using Python over C/C++. When increasing the precision from 50 to 500bits, the runtime of the length spectrum computation only increases by 30\% indicating that, for this range of precisions, the time is dominated by the overhead from the Python wrapping and computations in Python that are not floating-point operations. The length spectrum computations (including the verified ones) for all census manifolds succeeded with 500bits precision. Thus, for simplicity, we used 500bits precision for all computations with the new algorithm unless indicated otherwise.

Unfortunately, a reimplementation of the new algorithm in C++ requires substantial effort and is out of the scope of this project.

\section{Future work} \label{sec:futureFinTrig}

Future work might include support for geometric finite triangulations (based on \cite{casson:geo,heardThesis,matthiasVerifyingFinite}). Such work would enable:
\begin{itemize}
\item Extending the algorithms described here and in the subsequent paper \cite{goerner:drilling} to closed hyperbolic 3-manifolds without known geometric spun-triangulations such as \texttt{m007(3,1)}.
\item Drilling any geodesic and, thus, the isometry signature of any closed hyperbolic 3-manifold; see \cite{goerner:drilling}. This is in contrast to geometric spun-triangulations where we cannot drill geodesics intersecting the incompleteness locus $M_\text{filled}\setminus M$.
\item Computing the optimal Margulis number $\mu(M)$ for all closed hyperbolic 3-manifolds. This includes the examples in Section~\ref{sec:proofcensusMargulisConstant} where we could not compute it using the given geometric spun-triangulations.
\end{itemize}
The ability to use finite geometric triangulations in addition to geometric spun-triangulations also would yield the following benefits in terms of ease of use and performance:
\begin{itemize}
\item For many manifolds, it makes it easier to just find a geometric triangulation; see \cite{matthiasVerifyingFinite} for examples. 
\item We expect an acceleration of the new length spectrum algorithm for some (but not all) closed hyperbolic 3-manifolds --- in particular, the ones with long systole. We already mentioned this in Remark~\ref{rem:finAccel} and Section~\ref{sec:perfTrigDep}.
\end{itemize}

\clearpage

\appendix
\section{Notation}
\label{App:Notation}
For the convenience of the reader, we list some of the notation used in the paper.  

\newcolumntype{L}[1]{>{\raggedright\let\newline\\\arraybackslash\hspace{0pt}}p{#1}}
\newcolumntype{C}[1]{>{\centering\let\newline\\\arraybackslash\hspace{0pt}}p{#1}}
\newcolumntype{R}[1]{>{\raggedleft\let\newline\\\arraybackslash\hspace{0pt}}p{#1}}

\begin{longtable}{L{0.15\textwidth}L{0.79\textwidth}}%
\renewcommand{\arraystretch}{1.4}%
$M$ & Oriented, finite-volume hyperbolic $3$-manifold; complete until Section~\ref{sec:generalizationToSpun} generalizes to geometric spun-triangulations; can always be completed by attaching circles.\\
$\mu(M)$ & Optimal Margulis number of $M$; defined in Section~\ref{sec:margulis}; computed by Algorithm~\ref{algo:optimalMargulis}.\\
$\gamma$ & Primitive, unoriented closed geodesics in $M$.\\
$\lambda$ & Complex length of $\gamma$; computed from $\myPSL(2,\C)$-matrix in Section~\ref{sec:computeLen}.\\
$(\lambda_0, w_0), \dots$ & Length spectrum stream: $\lambda_i$ complex length, $w_i$ word in face-pairing presentation of $\pi_1(M)$; see Section~\ref{sec:putAlgo}; computed by Algorithm~\ref{algo:lenSpec}.\\
$a=[\underline{a}, \overline{a}]$ & Notation for interval; defined in Section~\ref{subsec:intervalNotation} and \cite[Section~2.2]{goerner:tiling}.\\
$\mu$ & Cut-off length for length spectrum; see Section~\ref{sec:immApp}.\\
$\myTrig$ & Geometric finite or ideal triangulation of $M$.\\
$\H^3$ & Hyperboloid model; see Section~\ref{sec:defPrelength}.\\
$P\subset\H^3$ & Fundamental polyhedron for $M$.\\
$\Gamma\subset\SO(1,3)$ & Decktransformations such that $M\cong \Gamma\backslash\H^3$ ($M_\text{filled}\cong \Gamma\backslash\H^3$ when $M$ is incomplete).\\
$m\in\Gamma$ & Element in $\Gamma$ (carries along a word $w$ in the face-pairing presentation of $\pi_1(M)$).\\
$S\subset P$ & A spine; defined in Section~\ref{sec:spineDefDef}.\\
$(\mu_0, m_0), \dots$ & Pre-length spectrum stream: $\mu_i\in\R^{\geq 0}$ cut-off length, $m_i\in\Gamma$; see Section~\ref{sec:defPrelength}; computed by Algorithm~\ref{algo:preLen}.\\
$z_t\in\C\setminus\{0,1\}$ & Cross ratio/shape of tetrahedron of geometric triangulation $\myTrig$; see Section~\ref{subsub:verifiedChain}.\\
$M_\text{filled}$ & Complete hyperbolic $3$-manifold obtained by attaching circles to $M$; see Section~\ref{sec:spineDef}.\\
$T_t\subset \H^3$ & Geodesic tetrahedron such that $P=\cup T_t$.\\
$S_t\subset T_t$ & Piece of spine.\\
$\vec{s}_t\in S_t$ & Incenter of $T_t$ serving as basepoint of $S_t$.\\
$r(S_t)$ & Spine radius: radius of $S_t$ about $\vec{s}_t$; computed by Algorithms~\ref{algo:spine}.\\
$\vec{v}^v_t\in \H^3\cup \posLightCone$ & Finite or ideal vertex of $T_t$; see Section~\ref{sec:spineVertexVectors} and \cite[Section~4.6]{goerner:tiling}.\\
$g^f_t$ & Face-pairing matrix; see \cite[Section~4.6]{goerner:tiling}.\\
$B(l)$ & Horoball defined by light-like vector $l\in \posLightCone$; see \cite[Section~3.1]{goerner:tiling}.\\
$\vec{s}_t^{kl}, \vec{s}_t^f$ & Point on edge $kl$ or face $f$ of $T_t$ spanning $S_t$ together with $\vec{s}_t$ in Section~\ref{sec:spinePoints}.\\
$(r_0, m_0T_{t_0}), \dots$ & Tiling stream: $r_i$ tiling radius, $m_iT_{t_i}$ lifted tetrahedron with $m_i\in\Gamma$; computed by \cite[Algorithm~7.3.1]{goerner:tiling}; see Section~\ref{sec:preLengthAlgo}.\\
$G$ & Set of unoriented geodesics; used by Algorithm~\ref{algo:lenSpec}; implemented in Sections~\ref{sec:setOfGeos} and~\ref{sec:coreCurveAvoidance} (for spun-triangulations).\\
$\Eq_\Gamma(m, m')$ & Predicate that is true if $m$ and $m'\in\Gamma$ are the same; see Section~\ref{sec:matrixEquality}.\\
$\Gamma^\text{hyp}\subset \Gamma$ & Hyperbolic elements of $\Gamma\subset\SO(1,3)$; see Section~\ref{sec:matrixPowerEquality}.\\
$\Eq^\Z(m',m)$ & Prediate that is true if $m'\in\Gamma^\text{hyp}$ is a non-trivial power of $m\in\Gamma^\text{hyp}$.\\
$K$ & Line fixed by $m\in\Gamma^\text{hyp}$; see Section~\ref{sec:graphTraceSet}.\\
$M_\text{filled}$ & Complete manifold obtained from $M$ by attaching circles; see Section~\ref{sec:generalizationToSpun}.\\
$C_i$ & Generalized cusp neighborhood: horocusp neighborhood (if cusp complete) or tube about core curve  (otherwise); see Definition~\ref{def:genCuspNbhd}.\\
$U\subset\H^3$ & Complement of generalized cusp neighborhoods; fulfills swiss cheese property (see Definition~\ref{def:swissCheese}); see Section~\ref{sec:swissCheese}.\\
$\hat{C_i}\supset C_i$ & Enclosing horoneighborhood; see Section~\ref{sec:encloHoro}.\\
$h$ & Scaling factor to obtain edge lengths of $\hat{C_i}$; see Section~\ref{sec:computeHorotriangles} and Figure~\ref{fig:enclosingNbhd}.\\
$r$ & Radius of tube about core curve; see Section~\ref{sec:scalingNeighborhoodsIncomplete} and Figure~\ref{fig:enclosingNbhd}.\\
$s$ & Euclidean slope of cone corresponding to tube about core curve; see Section~\ref{sec:scalingNeighborhoodsIncomplete} and Figure~\ref{fig:enclosingNbhd}.\\
$N^v_t\subset\H^3$ & Lift of generalized cusp neighborhood $C_i$; see Section~\ref{sec:perTetCoreCurve} and Figure~\ref{fig:enclosingNbhd}.\\
$c^v_t$ & Lifted core curve (if vertex $v$ of $T_t$ corresponds to complete cusp), empty (otherwise); see Section~\ref{sec:perTetCoreCurve} and Figure~\ref{fig:enclosingNbhd}; computed in Section~\ref{sec:addCoreCurves}.\\
$\vec{w}^{kl}_t$ & Intersection of boundary of generalized cusp neighborhood with edge $kl$ of $T_t$; see Section~\ref{sec:spunSpine} and Figure~\ref{fig:enclosingNbhd}.\\
$\vec{v}^{kl}_t$ & Intersection of boundary of enclosing horoneighborhood with edge $kl$ of $T_t$; see Section~\ref{sec:spunSpine} and Figure~\ref{fig:enclosingNbhd}.\\
$q\in \C\setminus\{0\}$ & Vertex position in upper half-space model; see Section~\ref{sec:spunSpine} and Figure~\ref{fig:enclosingNbhd}.\\
$\overline{r}(S_t)$ & Upper bound for radius of $S_t$ about $\vec{s}_t$; see Section~\ref{sec:spunSpine}.\\
$T^f_t\subset\H^3$ & Face $f$ of $T_t$; see Section~\ref{sec:modTiling}; defined in \cite[Section~4.6]{goerner:tiling}.\\
$\varepsilon$ & Maximal length of essential loops forming $\varepsilon$-thin part; see Section~\ref{sec:margulis}.\\
$K, K'\subset \H^3$ & Geodesic or horoball corresponding to a closed geodesic or cusp neighborhood of $M$; see Section~\ref{sec:margulisAux} and \cite[Section~5.1]{goerner:tiling}.\\
$\Gamma_K$ & Stabilizer of $K$; see Section~\ref{sec:margulisAux} and \cite[Section~7.2]{goerner:tiling}.\\
$r(K,\varepsilon)$ & Radius $\varepsilon$-thin part about $K$ as neighborhood of $K$; see Section~\ref{sec:margulisAux}; computed in Section~\ref{sec:implR}.\\
$\varepsilon(K, r)$ & Supremal $\varepsilon$ such that neighborhood of radius $r$ of $K$ is $\varepsilon$-thin; see Section~\ref{sec:margulisAux}; computed in Section~\ref{sec:implEps}.\\
$\mu(K, K', r)$ & Supremal $\varepsilon$ such that the $\varepsilon$-thin parts about $K$ and $K'$ are disjoint if $r$ is the distance between $K$ and $K'$ in $M$; see Section~\ref{sec:margulisAux}; computed in Section~\ref{sec:implMu}.\\
$d_M(K, K')$ & Distance of $K$ and $K'$ in $M$; see Section~\ref{sec:margulisAux}; defined in \cite[Section~10.1]{goerner:tiling}.
\end{longtable}

\newpage

\bibliographystyle{hamsalphaMatthias}
\renewcommand{\path}[1]{#1}
\newcommand{\noop}[1]{}

\bibliography{lenSpec}

\end{document}